\documentclass[a4paper,10pt]{article}
\usepackage{amsmath, amsthm, amssymb}
\usepackage{url}

\usepackage{physics}
\usepackage{graphicx,lipsum}
\graphicspath{ {./downloads/} }

\usepackage[colorlinks,citecolor=red,urlcolor=blue,bookmarks=true,hypertexnames=true]{hyperref}
\usepackage{tikz}
\usetikzlibrary{calc}
\usetikzlibrary{shapes}
\usepackage[autostyle]{csquotes}
\makeatletter

\usepackage[colorlinks]{hyperref}
\usepackage[nameinlink,capitalize]{cleveref}
\newtheorem{theorem}{Theorem}[section]
\newtheorem{corollary}{Corollary}[theorem]
\newtheorem{lemma}[theorem]{Lemma}
\newtheorem{proposition}{Proposition}[section]
\newtheorem{conjecture}{Conjecture}[section]
\newtheorem{claim}{Claim}[section]
\newtheoremstyle{named}{}{}{\itshape}{}{\bfseries}{.}{.5em}{\thmnote{#3's }#1}
\theoremstyle{named}

\newtheorem{remark}{Remark}
\theoremstyle{definition}
\newtheorem{definition}{Definition}[section]
\newtheorem{example}{Example}

\theoremstyle{remark}

\theoremstyle{conclusion}

\theoremstyle{observation}

\newtheorem{principle}[theorem]{Principle}

\usepackage{xspace}
\usepackage[margin=1.0in]{geometry}

\newcommand{\Spec}{\operatorname{Spec}}

\usepackage{titlesec}

\usepackage{mathtools}

\DeclarePairedDelimiterX{\inp}[2]{\langle}{\rangle}{#1, #2}
\titleformat{\chapter}
  {\Large\bfseries}
  {}
  {0pt}
  {\huge}

\usepackage{makeidx}  
\usepackage{dsfont}
\usepackage{amsmath,amssymb,amsfonts,amscd}
\usepackage{hyperref}
\usepackage{booktabs} 
\usepackage{appendix}
\usepackage[utf8]{inputenc}

\usepackage{tikz,tikz-cd}
\tikzset{
    invisible/.style={opacity=0},
    visible on/.style={alt={#1{}{invisible}}},
    alt/.code args={<#1>#2#3}{%
      \alt<#1>{\pgfkeysalso{#2}}{\pgfkeysalso{#3}}%
  }
}
\tikzset{
  symbol/.style={
    draw=none,
    every to/.append style={
      edge node={node [sloped, allow upside down, auto=false]{$#1$}}}
  }
}

\usepackage{xcolor}

\begin{document}

\begin{center}
\fontsize{13pt}{10pt}\selectfont
    \textsc{\textbf{NONCOMMUTATIVE HOURGLASSES II 1/2: DERIVED CATEGORY OF BIRATIONAL LINKAGE WITH WELL-FORMED SINGULARITIES}}
    \end{center}
\vspace{0.1cm}
\begin{center}
   \fontsize{10pt}{8pt}\selectfont
    \textsc{HAO XINGBANG\footnotemark}
\end{center}
\vspace{0.2cm}
\begin{center}
   \fontsize{12pt}{10pt}\selectfont
    \textsc{Abstract}
\end{center}
\footnotetext{\emph{In memory of my maternal grandparents.}}

We consider varieties whose singularities are complete intersections within toric singularities, which we call well-formed varieties. Using the construction of equivariant standard resolutions, we can extend several classical formulas in the derived category of bounded coherent sheaves to this singular setting. Finally, we also discuss the Bondal--Orlov localization for certain morphisms between terminal complex threefolds and their stacks.
\tableofcontents

%

\section{Introduction}
\label{sec:introduction}

Derived categories were introduced as a systematic framework for studying complexes in abelian categories, originating from Verdier's abstraction and organization of homological algebra. Their primary applications arise from considering coherent sheaves or constructible constant sheaves on various sites of algebraic varieties, in order to further investigate the intrinsic geometry of the varieties themselves, as well as their relative and birational properties.

For $\mathrm{D}^{b}(\mathrm{Coh}\,X)$, Beilinson and Kapranov studied the case where $X$ is a projective space or a Grassmannian, showing that these categories admit full exceptional collections and are equivalent to explicitly computable derived categories $\mathrm{D}^{b}(R)$. Orlov then established analogous semiorthogonal decompositions for projective bundles and blow-ups along regular embeddings \cite{Orlov1993}; for a complete treatment we refer to \cite{Huybrechts}. Subsequently, Bondal--Orlov studied semiorthogonal decompositions associated with standard flips \cite{BO}; for full details see \cite{BFR}. In the case where $X$ is a complete intersection in a projective space or a Grassmannian, work of Orlov and Kuznetsov \cite{Orlov2009,Kuz2006} shows that their derived categories admit similar descriptions via semiorthogonal decompositions.

On the other hand, the three-dimensional minimal model program (MMP) was developed in parallel during the same period \cite{KM}. Starting from resolution of singularities, one introduces increasingly mild classes of singularities such as terminal and klt singularities, which allow the minimal model program to proceed in a controlled way within this category. In dimension three, one obtains a sequence of birational maps $X \dashrightarrow X^{+}$, including extremal divisorial contractions and flips, eventually producing a model $X^{\mathrm{nef}}$ with sufficiently nef canonical divisor. After a series of works by Kawamata \cite{KaEq,Kalog,KaTor}, it was proposed to ask whether a $\mathrm{K}$-inequality should imply a corresponding $\mathrm{D}$-embedding phenomenon \cite{KawamataDK}. If such a principle holds, one expects a noncommutative version of the minimal model program, where different MMP paths starting from a sufficiently large $\mathrm{D}^{b}(X)$ produce different semiorthogonal decompositions, together with a noncommutative nef model. This perspective has since been developed in various directions, e.g.\ \cite{Bridgeland2002,ChenJC,Halpern-Leistner2015,Ballard2012,KPS}.

However, for birational maps between singular algebraic varieties, the behavior of derived categories $\mathrm{D}^{b}(X)$ remains only partially understood. In this direction, Bondal and Orlov proposed the following folklore conjecture in \cite[Section~5]{BO2002}:
\begin{conjecture}[Bondal--Orlov]
Let $f:Y\to X$ be a proper birational morphism, where $X$ has rational singularities. Then the functor
\[
f_{*}:\mathrm{D}^{b}(Y)\to \mathrm{D}^{b}(X)
\]
identifies $\mathrm{D}^{b}(X)$ with the Verdier localization $\mathrm{D}^{b}(Y)/\ker(f_{*})$.
\end{conjecture}
More generally, they expect to consider a categorical resolution of a triangulated category $\mathcal{D}$, defined as a pair $(\mathcal{C},\mathcal{K})$, where $\mathcal{C}$ is an abelian category of finite homological dimension and $\mathcal{K}\subset \mathrm{D}^{b}(\mathcal{C})$ is a thick subcategory such that $\mathcal{D}\simeq \mathrm{D}^{b}(\mathcal{C})/\mathcal{K}$. The existence of such noncommutative resolutions has been extensively studied in various contexts, e.g.\ \cite{VDB,KuznetsovLunts2015,Abuaf2016}. In particular, Efimov proved the Bondal--Orlov conjecture for proper birational morphisms with nonrational centers \cite{Efimov2020}, with applications covering most Fano blow-ups \cite{KS}. For certain resolutions, Mauri and Shinder \cite{MS} showed that the image of the functor $f_*$ generates $\mathrm{D}^{b}(X)$ as a triangulated category. The conjecture appears trivial, and perhaps it is. Even so, the author finds it instructive to work through a few simple cases.

In this setting, we consider singularities $X$ which are formally locally complete intersections. When the dimension is sufficiently large, such spaces satisfy various purity properties \cite{SGA2}, e.g.\ for the punctured \'etale fundamental group or locally constant cohomology. This allows us to introduce additional cyclic group actions and interpret the corresponding quotient as a canonical stack. On the other hand, as developed in our previous work \cite{Hao2025b}, using Bourbaki-type sequences together with equivariant standard resolutions, for any torsion-free module $M$ on $X$ and a big class divisorial valuation $val$ on $X$, we obtain a sequence of submodules $M_i \subset M$ induced by the standard resolution construction.

Assume we have a divisorial contraction associated to a valuation $val$,
\[
\begin{tikzcd}
E \ar[d, "f_E"'] \ar[r, "\iota"] & Y \ar[d, "f"] \\
C \ar[r, "\kappa"] & X
\end{tikzcd}
\]
Using the standard resolution constructed above, one obtains an admissible and fully faithful functor $\mathcal{R}f^{*}:\mathrm{D}^{b}(X)\to \mathrm{D}^{b}(Y)/\mathcal{K}$ for a certain obstruction subcategory $\mathcal{K}$. We actually use the following principle in different proofs:
\begin{principle}\label{P1}
If for any $F\in \mathrm{D}^{b}(X)$ we can construct an inverse system with objects
$\mathcal{R}^{\mathbf{s}}f^*F\in \mathrm{D}^{b}(Y)$, filtered by $\mathbf{s}\in \mathbb{N}^{k}$,
satisfying the following conditions:
\begin{enumerate}
\item $\mathcal{R}^{\mathbf{s}}f^* F$ forms a filtered system and
$\varprojlim_{\mathbf{s}} \mathcal{R}^{\mathbf{s}}f^*F \simeq f^{*}F$.
\item For any morphism $\mathcal{R}^{\mathbf{t}}f^*F \to \mathcal{R}^{\mathbf{s}}f^*F$ in this filtered
system with $\mathbf{t}>\mathbf{s}$ (under a zigzag order in $\mathbb{N}^{k}$),
the cone $\tau$ lies in $\langle\ker f_*\rangle$.
\end{enumerate}
Then
\[
f_*:\mathrm{D}^{b}(Y)/\langle\ker f_*\rangle \simeq \mathrm{D}^{b}(X),
\]
with inverse functor $\mathcal{R}f^*$, induced by $F \mapsto \mathcal{R}^{\mathbf{s}}f^{*}F$, where a morphism
$\phi \colon F \to G$ maps to
\[\mathcal{R}^{\mathbf{t}}f^{*}\phi \colon \mathcal{R}^{\mathbf{t}}f^{*}F \to \mathcal{R}^{\mathbf{s}}f^{*}G\]
for $\mathbf{t} \gg \mathbf{s}$.
\end{principle}

If we expect to obtain more information about $\langle \ker f_{*}\rangle$, we need to perform further computations.

\begin{principle}\label{P2}
For $\kappa_{*}\mathcal{O}_{C}\in \mathrm{D}^{b}(X)$, we can construct objects
$\mathcal{R}^{\mathbf{s}}f^{*}(\kappa_{*}\mathcal{O}_{C})\in \mathrm{D}^{b}(Y)$, filtered by
$\mathbf{s}\in \mathbb{N}^{k}$, satisfying the following conditions:
\begin{enumerate}
\item $\mathcal{R}^{\mathbf{s}}f^{*}(\kappa_{*}\mathcal{O}_{C})$ forms a filtered system and
$\varprojlim_{\mathbf{s}} \mathcal{R}^{\mathbf{s}}f^{*}(\kappa_{*}\mathcal{O}_{C}) \simeq f^{*}(\kappa_{*}\mathcal{O}_{C})$.
\item For any morphism $\mathcal{R}^{\mathbf{t}}f^{*}(\kappa_{*}\mathcal{O}_{C}) \to \mathcal{R}^{\mathbf{s}}f^{*}(\kappa_{*}\mathcal{O}_{C})$
in this filtered system, with $\mathbf{t}>\mathbf{s}>0$ under a zigzag order in $\mathbb{N}^{k}$,
the cone $\tau_{\mathbf{t},\mathbf{s}}$ lies in $\langle\ker f_*\rangle$.
\end{enumerate}
Then by analyzing the objects $\tau_{\mathbf{t},\mathbf{s}}$, one obtains a set of generators
for $\langle\ker f_*\rangle$.
\end{principle}

It is worth mentioning that, in classical approaches, there are only a few ways to construct a cone $\tau$ satisfying $\tau \in \langle\ker f_*\rangle$. For a general morphism $f$, these include spectral sequence methods and tilting with respect to stability conditions. The method used here is based on a kind of monomial stability filtration, inspired by the proof of Theorem~2.14 in \cite{BondalKapranovSchechtman2018}, which employs a similar argument but replaces Principle~\ref{P2} with the tilting heart of Bridgeland--Van den Bergh. In previous work \cite{Hao2025a}, the author obtained analogous results via a more involved $E_{2}$-spectral sequence argument. Clearly, these approaches share the same underlying core idea.

Using these two principles, we extend the classical Orlov blow-up formula to a broader class of varieties with well-formed singularities or their canonical stacks. We obtain the following theorem:
\begin{theorem}[Cor.\ \ref{corollaryin1} and \ref{corollaryin2}]
For a broad class of three-dimensional terminal divisorial contractions to a point or a smooth curve which can be realized as weighted blow-ups or Kawamata blow-ups, the Bondal--Orlov localization formula holds. Moreover, $\langle \ker f_{*}\rangle$ is generated by $\iota_*\ker (f_{E*})$.
\end{theorem}
Based on the existence of controlled resolutions of singularities, we obtain the following result:
\begin{theorem}[Cor.\ \ref{corollaryin3} and \ref{corollaryin5}]
Let \(f:Y\to X\) be a proper birational morphism between terminal threefolds. Then the Bondal--Orlov localization formula holds on the level of varieties. If \(Y\) is \(\mathds Q\)-factorial, the Bondal--Orlov localization formula also holds on the level of root stacks.
\end{theorem}

Combining the preceding two results, we obtain the following consequence.

\begin{corollary}
Let \(Y\) be a three-dimensional \(\mathds Q\)-factorial terminal variety, and let \(f:Y\to X\) be any divisorial contraction,  flipping contraction, or flipped contraction. Then the Bondal--Orlov localization formula holds both at the level of varieties and canonical stacks.
\end{corollary}

Moreover, the case of standard flips is essentially parallel to Kawamata blow-ups. In particular, we extend the Bondal--Orlov standard flip formula to three-dimensional terminal section flips endowed with additional stacky structure in Corollary~\ref{corollaryin4}. The method here is closely related to the derived category approach to VGIT developed in \cite{Halpern-Leistner2015,Ballard2012}. However, in order to obtain good functors and extend the construction to the global setting, one perhaps needs to employ more classical constructions, such as those appearing in \cite{BFR}.

We end with some open-ended remarks. The main technique of this article relies on the construction of standard resolutions for complete intersections in \cite[Section~2]{Hao2025b}, where the codimension $\leq 2$ case is explained in detail. The higher codimension case follows by an entirely analogous, though more involved, combination; however, the codimension $\leq 2$ case already suffices for the results above. In particular, for iterated weighted blow-ups along points, one can construct a Rees pullback step by step; however, a one-step construction of standard
resolutions remains unclear. This issue is closely related to the use of $\mathbb{Z}^{k}$-valued divisorial valuations, where the corresponding ``division'' structure encodes symmetries that are deeply connected to flips. At present, the method does not allow a naive gluing construction. For weighted blow-ups along general centers, the construction of standard resolutions seems to require a gluing procedure analogous to Hironaka's gardening theory of permissible centers \cite{Hironaka1977}. A more flexible approach, based on gluing along stability conditions, appears to be natural but remains largely unexplored. In the present work, the use of affine space as the key variety together with complete intersection models partially conceals this difficulty. Finally, we note that the blow-up formula also admits analogous and interesting applications in the derived category of constructible sheaves, for instance in the work of Fujiwara \cite{Fujiwara}.
\subsection{Notation}
\begin{enumerate}
\item In this paper, we assume that the (local) ring $R$ or the variety $X$ is integral, (essentially) finitely generated and separated over a field $\mathrm{k}$. We use $\widehat{R}$ or $\widehat{X}_{C}$ to denote either the formal completion of $R$ at its maximal ideal $\mathfrak{m}$ or along $C$. Since we only consider coherent sheaves, we do not distinguish between $\operatorname{Spec} \widehat{R}$ and $\operatorname{Spf} \widehat{R}$.
\item We consider separated Deligne--Mumford stacks $\mathcal{X}$ of finite type over $\mathrm{k}$ (e.g., with proper and unramified diagonal). When the appropriate conditions are satisfied, $\widetilde{X}$ denotes the canonical stack of $X$.

\item We denote by $\mu_r$ the algebraic cyclic group scheme over $\mathrm{k}$, with irreducible representations $\nu^i$. Over a stacky closed point, we sometimes write $\nu_{x}^{i} := \mathrm{k}_x \otimes \nu^i$. We denote by $\mathbb{G}_m$ the multiplicative group scheme over $\mathrm{k}$, with irreducible representations $\chi^i$.

\item We usually denote by $\mathcal{P}(a_1,\dots,a_n)$ the weighted projective stack.

\item We write $\mathrm{D}^*(-)$ for the derived category of coherent modules (or coherent sheaves) on a (formal) ring $R$ (or space, stack $X$, $\mathcal{X}$). For a space or stack, we work with the big \'etale site; however, by descent, we do not distinguish between sites. Where $* = b, +, -$ refers, respectively, to the bounded, bounded below, and bounded above derived categories. Throughout, all standard functors such as $\pi_*, \pi^*$, etc., are understood in their derived sense unless otherwise specified.
\end{enumerate}

\section{Elementary Linkage and the Derived Category}
\subsection{Roots construction}
\begin{lemma}\label{spanlemma1}
Let $(R,\mathfrak m)$ be a Gorenstein local ring and let $\mathrm{k}=R/\mathfrak m$. Let $F \in \mathrm{Coh}(R)$. If
\[
\operatorname{Ext}^i_R(F,\mathrm{k})=0 \quad \text{for all } i,
\]
then $F=0$.
\end{lemma}

\begin{proof}
Consider the Cohen--Macaulay approximation $0 \to P_F \to M_F \to F \to 0$, where $M_F$ is maximal Cohen--Macaulay and $P_F$ has finite projective dimension. From the long exact sequence induced by the approximation and the assumption on $F$, noticing $\operatorname{Ext}^i_R(P_F,\mathrm{k})=0$ for $i \gg 0$ since $P_F$ has finite projective dimension, we deduce $\operatorname{Ext}^i_R(M_F,\mathrm{k})=0$ for $i \gg 0$. Therefore $M_F$ is free, and $F$ has finite projective dimension. By the condition $\operatorname{Ext}^i_R(F,\mathrm{k})=0$ for any $i$, we conclude $F=0$.
\end{proof}
We refer to \cite[Definition 1.47]{Huybrechts} for the basic definitions, we have:
\begin{lemma}
Let $X$ be a Gorenstein variety. Then for any closed point $x \in X$, the skyscraper sheaf $\mathrm{k}_x$ forms a spanning class in $\mathrm{D}^b(X)$.
\end{lemma}
\begin{proof}
Consider the spectral sequence
\[
E_2^{p,q} = \operatorname{Ext}^p_X(H^q(F), \mathrm{k}_x) \implies \operatorname{Ext}^{p+q}_X(F, \mathrm{k}_x)
\]
and the dualizing functor. It suffices to prove the case when $F$ is a sheaf on one side. If $\operatorname{Ext}^i_X(F, \mathrm{k}_x)=0$ for all closed points $x$, then $F_x$ is a module over the local ring $\mathcal{O}_{X,x}$ with
\[
\operatorname{Ext}^i_{\mathcal{O}_{X,x}}(F_x, \mathrm{k}_x)=0.
\]
By Lemma~\ref{spanlemma1}, we have $F_x=0$ for all closed points $x$, hence $F=0$. By the anti-equivalence of the dualizing functor, we obtain the analogous statement.
\end{proof}
\begin{lemma}\label{ptspanningclass}
Let $\mathcal{X}$ be a Gorenstein Deligne--Mumford stack. Then for any stacky closed point $x \in \mathcal{X}$, the objects
\[
\mathrm{k}_x \otimes \chi,
\]
where $\chi$ runs over all irreducible representations of the stabilizer group $G_x$, form a spanning class in $\mathrm{D}^b(\mathrm{Coh}(\mathcal{X}))$.
\end{lemma}

\begin{proof}
Similarly, consider the spectral sequence and the dualizing functor. It suffices to prove the case when $F$ is a sheaf. For any stacky closed point $x$, take an étale atlas or an equivariant formal neighborhood $[U/G_x]$. By the adjunction between restriction and induction for finite group representations over fields,
\[
\operatorname{Ext}^i_{[U/G_x]}(F, \bigoplus_{\chi \in \mathrm{Irr}(G_x)} \mathrm{k}_x \otimes \chi) \simeq \operatorname{Ext}^i_U(\mathrm{Res}\, F, \mathrm{k}_x) = 0.
\]
By Lemma~\ref{spanlemma1}, $\mathrm{Res}\, F_{\widehat{x}} = 0$ for all stacky closed points $x$. By Nakayama's lemma $F_x = 0$, hence $F=0$.
\end{proof}

\begin{remark}
Let $(R,\mathfrak m,\mathrm{k})$ be a local ring. By  faithfully flatness of field extension, we have
\[
\operatorname{Ext}^i_R(F,\mathrm{k})=0 \quad \Longleftrightarrow \quad \operatorname{Ext}^i_R(F,\overline{\mathrm{k}})=0,
\]
where $\overline{\mathrm{k}}$ denotes an algebraic closure of $\mathrm{k}$. It follows that closed points may be replaced by geometric points in the above spanning class whence the extension is finite.
\end{remark}

\begin{lemma}\label{excdis}
Let $\mathcal{X}$ be any Deligne--Mumford stack and $i: D \hookrightarrow \mathcal{X}$ the embedding of a Cartier divisor. Then for any $F$ in $\mathrm{Coh}(\mathcal{X})$, there is an excess distinguished triangle
\[
F(-D)[1] \longrightarrow i^* i_* F \longrightarrow F \in \Delta.
\]
\end{lemma}
\begin{proof}
Consider the natural distinguished triangle
\[
G \longrightarrow i^* i_* F \longrightarrow F \in \Delta.
\]
Apply $i_*$ and compare with the triangle
\[
i_* F(-D) \xrightarrow{\;\;\;0\;\;\;} i_* F \longrightarrow i_* i^* (i_* F) \in \Delta.
\]
We obtain $i_{*}G \simeq i_{*}F(-D)[1]$. Since $F$ is a pure sheaf, it follows that $G \simeq F(-D)[1]$.
\end{proof}

Let $\mathcal{X}$ be any Deligne--Mumford stack, $D \subset \mathcal{X}$ a pseudo-effective Cartier divisor, and let $r$ be a positive integer with $\gcd(r, \operatorname{char} \mathrm{k}) = 1$. Consider the $r$-th root stack
\[
\sqrt[r]{(D/\mathcal{X})}.
\]
We refer to \cite[Sec. 1.3]{FMN} for a precise definition of root stacks.  That is étale locally on $\mathcal{X}$, if $x$ is a stacky point with stabilizer group $G_x$, and $D$ is defined by a local equation $f$, we require that $f$ is $G_x$-equivariant with a compatible choice of $G_x$-character on $t$, where $t^r=f$, such that the root construction is locally well-defined on the stack and admits a compatible $G_x$-equivariant structure. We usually denote this compatibility by $\chi_t^r = \chi_f$. Moreover, by gluing the local constructions, we obtain a tautological line bundle $\mathcal{M}$ on $\sqrt[r]{(D/\mathcal{X})}$ such that
\[
\mathcal{M}^{\otimes r} \simeq \mathcal{O}_\mathcal{X}(D).
\]
\begin{lemma}\label{flatbasechange0}
For any proper morphism $\pi: \mathcal{Y} \to \mathcal{X}$ between Deligne--Mumford stacks, we have the following Cartesian square:
\begin{equation}\label{eq:root-cartesian-square}
\begin{tikzcd}
\sqrt[r]{(\pi^{-1}D/\mathcal{Y})} \arrow[r, "\varpi_\mathcal{Y}"] \arrow[d, "\pi_\mathcal{Y}"'] & \mathcal{Y} \arrow[d, "\pi"] \\
\sqrt[r]{(D/\mathcal{X})} \arrow[r, "\varpi"'] & \mathcal{X}
\end{tikzcd}
\end{equation}
where $\pi^{-1}D$ denotes the total inverse image of $D$.
\end{lemma}
\begin{proof}
By taking a smooth affine étale atlas of $\mathcal{X}$ and forming the fiber product, we may assume $\mathcal{X}=\operatorname{Spec}(R)$ and $\mathcal{Y}=\operatorname{Spec}(S)$ are affine, with $G_S$-equivariant structure which has no action on $\mathcal{X}$. Choose a local trivialization of $\mathcal{O}_{\mathcal{X}}(D)$ such that $D$ is defined by a semi-invariant $f\in R$. Then the root stack admits the presentation
\[
\sqrt[r]{(D/\mathcal{X})}
\simeq
[\operatorname{Spec}(R[t]/(t^r-f))/\mu_r].
\]
After base change to $S$, we have $R[t]/(t^r-f)\otimes_R S \simeq S[t]/(t^r-f)$,
with a compatible $\mu_r\times G_S$-action. Hence the desired Cartesian property follows. Globally, it can be given by the product of the Cartesian diagrams defined in \cite[Sec.\ 1.3]{FMN}.
\end{proof}

Since the root stack is generally not representable, we need to provide some explanation.
\begin{lemma}\label{fbcofroot}
For any morphism $\pi: \mathcal{Y} \to \mathcal{X}$ as in Lemma~\ref{flatbasechange0}, we have the flat base change formula for \eqref{eq:root-cartesian-square}. That is, for any object $\mathcal{F}$ in $\mathrm{D}^b(\mathcal{Y})$, the natural base change morphism
\[
\varpi^* \pi_* \mathcal{F} \longrightarrow (\pi_{\mathcal{Y}})_* \varpi_{\mathcal{Y}}^* \mathcal{F}
\]
is an isomorphism.
\end{lemma}

\begin{proof}
By descent property, proper pushforward and pullback are compatible with passage to the étale site. As in the scheme case, take an étale atlas of $\mathcal{X}$ of the form $[U/G_x]$. Thus we reduce to the affine case $\mathcal{X}=\operatorname{Spec} R$ with a $G_x$-equivariant structure, and consider the Cartesian diagram over
\[
[\operatorname{Spec}(R[t]/(t^r-f))/\mu_r] \longrightarrow \operatorname{Spec} R.
\]
Note that $R[t]/(t^r-f)$ is faithfully flat over $R$. Hence, after forgetting the $\mu_r\times G_S$-equivariant structure, the usual flat base change theorem yields isomorphism
\[
\varpi^* \pi_* \mathcal{F} \longrightarrow (\pi_{\mathcal{Y}})_* \varpi_{\mathcal{Y}}^* \mathcal{F},
\]
it suffices to check that this is compatible with the $\mu_r\times G_S$-equivariant \v{C}ech complex locally computing the quotient stack. This compatibility follows formally from the scheme-theoretic flat base change, and the proof is identical to the classical case.
\end{proof}

Similar statements of the following consequence in related settings have previously been established in \cite{IU,KAC,BD} using different but closely related arguments. Although the situation considered here is slightly different, the proof is formally identical.

\begin{proposition}\label{derivedroot}
Let $\mathcal{X}$ be a Gorenstein Deligne--Mumford stack and $D \subset \mathcal{X}$ a Cartier divisor satisfying the compatibility condition above. Let
$\varpi:\sqrt[r]{(D/\mathcal{X})}\longrightarrow \mathcal{X}$
be the $r$-th root stack. Then we have a semiorthogonal decomposition
\[
\mathrm{D}^{b}\!\left(\sqrt[r]{(D/\mathcal{X})}\right)
\simeq
\left\langle
\mathrm{\Upsilon}_{r-1},\dots,\mathrm{\Upsilon}_{1},\mathrm{\Upsilon}:=\mathrm{Im}\,\varpi^{*}
\right\rangle .
\]
The embedding functor of $\mathrm{\Upsilon}_{i}$ is defined by the commutative diagram
\[
\begin{tikzcd}[column sep=small,row sep=small,ampersand replacement=\&]
\& {[D/\mu_{r}]} \arrow[r,"l"] \arrow[d,"\pi_{D}"] \& {\sqrt[r]{(D/\mathcal{X})}} \\
\& D
\end{tikzcd}
\]
via
\[
i_{\mathrm{\Upsilon}_{i}}(-)
:=
l_{*}\pi_{D}^{*}(-)\otimes \mathcal{M}_{D}^{\otimes i},
\qquad
i_{\mathrm{\Upsilon}_{i}}^{*}(-)
=
\pi_{D\,*}l^{*}\!\left(\mathcal{M}_{D}^{\otimes -i}\otimes (-)\right).
\]
where all functors are admissible and fully faithful.
\end{proposition}
\begin{proof}[Outline of proof]
First, the local ring map of $\varpi$ is flat, so its pullback preserves boundedness, and the fully faithfulness of $\varpi^*$ is immediate. Consider the subcategories
\(\Upsilon_1, \dots, \Upsilon_{r-1}\); the orthogonality with $\Upsilon$ is also evident. To verify left orthogonality, it suffices to show that for any closed stacky points $x,y$ in $D$ we have
\[
\mathrm{Ext}_{\sqrt[r]{\mathcal{X}}}^*(\Upsilon_i(\nu^a_x), \Upsilon_i(\nu^b_y)) = \mathrm{Ext}_{\sqrt[r]{\mathcal{X}}}^*(\nu^a_x, \nu^b_y), \quad \text{and} \quad
\mathrm{Ext}_{\sqrt[r]{\mathcal{X}}}^*(\Upsilon_i(\nu^a_x), \Upsilon_j(\nu^b_y)) = 0 \text{ if } i<j.
\]
For generation, note that $\varpi^* \mathrm{k}_{z}$ has a filtration whose factors are $\nu_{\bar z}^i$ for $i=1,\dots,r-1$ in the root stack, where $\nu$ denotes the representation associated with the stacky point $\bar z$ (the inverse image of $z$ in the root stack) of $\mathcal{M}|_{\bar z}$.
\end{proof}

\subsection{Canonical stack}\label{Canonicalstack}

Let $X$ be a normal algebraic variety of dimension $\ge 3$ over a field $\mathrm{k}$. Assume that the singular locus $\operatorname{Sing}(X)$ is a disjoint union of smooth projective  subvarieties
\[
\operatorname{Sing}(X)=\bigsqcup_i C_i.
\]
For each $i$, the formal completion of $X$ along $C_i$ is isomorphic to a cyclic quotient
\[
\widehat{X}_{C_i} \simeq \bigl[ Z(\mathcal{I}_i) \subset \widehat{\mathbb{A}}_{C_i}(E_i) \, / \, \mu_r \bigr],
\]
where $\widehat{\mathbb{A}_{C_i}(E_i)}$ denotes the formal completion of the total space of a vector bundle $E_i$ over $C_i$ along $C_i$, and $\mathcal{I}_i$ is a local complete intersection ideal on $C_i$ defining a closed subscheme $Z(\mathcal{I}_i)\subset \mathbb{A}_{C_i}(\mathcal{E}_i)$ with a $\mu_r$-action. Equivalently, after locally trivializing $E_i$, we obtain a collection of compatible representations, and the following compatibility conditions hold:
\begin{enumerate}
\item $\mathcal{O}_{C_i}[[x_1,\dots,x_s]]/I_i$ is normal, satisfies Serre's condition $\mathrm{R}_2$, and is a local complete intersection over $\mathcal{O}_C$. For verifiability, we interpret the l.c.i. condition as for every prime ideal $\mathfrak{q}$ of $\mathcal{O}_{C_i}[[x_1,\dots,x_s]]/I_i$, lying over a prime $\mathfrak{p} \subset \mathcal{O}_C$, the completed local ring $\widehat{\left(\mathcal{O}_{C_i}[[x_1,\dots,x_s]]/I_i\right)_{\mathfrak{q}}}$ is a complete intersection over $\widehat{\mathcal{O}_{C,\mathfrak{p}}}$, we refer to \cite[Lemma 23.9.5]{StacksProject} for this criterion.
\item \label{fdet}$I_i$ is equivariantly finitely determined, i.e., it is $\mu_r$-equivariant and there exists an $\mathcal{O}_{C_i}$-equivariant automorphism of $\mathcal{O}_{C_i}[[x_1, \dots, x_s]]$ sending $I_i$ to an ideal generated by elements algebraic in $x_1, \dots, x_s$, such that, in these coordinates, the formal $\mu_r$-action is induced by an algebraic action.
\item the $\mu_r$-action is free in codimension one, with $\gcd(r, \operatorname{char} \mathrm{k}) = 1$. Typically, we require that the fixed locus is exactly $C$;\footnote{Hence the action is free outside \(C\) by assumption. One may replace \(C\) by a smooth stacky center; this only introduces additional combinatorial complexity.}
\end{enumerate}
Such a variety $X$ is said to be \emph{well-formed}.
\begin{definition}[Pre-canonical stack, {\cite[Definition 4.4]{FMN}}]
Let $\mathcal{X}$ be a normal Deligne--Mumford stack, and let $\varepsilon:\mathcal{X}\to X$ denote the structure morphism to its coarse moduli space. We say that $\mathcal{X}$ is \emph{pre-canonical} if $\varepsilon$ is an isomorphism in codimension one.
\end{definition}
In this setting, one can construct the pre-canonical stack for a well-formed variety $X$. We now recall Vistoli's construction \cite{Vis} in the setting of a formal neighbourhood along $C := C_i$.
By finite determinacy (\ref{fdet}), there exists an algebraic subscheme
$Z(\mathcal{I}) \subset C \times \mathbb{A}(E)$ equipped with a $\mu_r$-action,
such that the formal completion of $X$ along $C$ is isomorphic to
\[
\widehat{X}_C \cong \widehat{(Z(\mathcal{I})/\mu_r)}_C.
\]
By Artin approximation, there exists a scheme $W$ and a common \'etale diagram
\[
\begin{tikzcd}[column sep=0.5em]
& & W \arrow[dl] \arrow[dr] & \\
& X && Z(I)/{\mu_r}
\end{tikzcd}
\]
Set $U := Z(I) \times_{Z(I)/\mu_r} W$, which is normal and carries a natural $\mu_r$-action free in codimension one. Since $C$ is smooth and projective  we may shrink $U$ suitably so that the above \'etale map becomes quasi-finite.
Over a small neighborhood of $C$, consider the stack presentation
\[
[(U \times_X U)^\mathrm{nor} \rightrightarrows U],
\]
where $(U \times_X U)^\mathrm{nor}$ denotes the normalization of $U \times_X U$. For any natural projection, the morphism
\((U \times_X U)^\mathrm{nor} \to U\) is \'etale in codimension one, and by the following lemma we conclude it is \'etale everywhere:

\begin{lemma}
$U$ is a complete intersection in every codimension.
\end{lemma}

\begin{proof}
By passing to closed points, we may assume $U=\operatorname{Spec} R$, where $R$ is a local ring, and let $\mathfrak q\subset R$ be the prime ideal corresponding to the center $C$. For any prime ideal $\mathfrak p$ with $\mathfrak q\not\subset \mathfrak p$, the localization $R_{\mathfrak p}$ is regular. If $\mathfrak q\subseteq \mathfrak p$, then by the assumption that $\mathcal{O}_C[[x_1,\dots,x_s]]/I$ is a local complete intersection over $\mathcal{O}_C$, together with the regularity of $C$, it follows that the completion $\widehat{R_{\mathfrak p}}$ admits a presentation
\[
\widehat{R_{\mathfrak p}}
\cong
\mathrm{k}(\mathfrak p)[[x_1,\dots,x_{n_{\mathfrak p}},\dots,x_{n_{\mathfrak p}+s}]]/I,
\]
where $I$ is a complete intersection ideal over $\mathrm{k}(\mathfrak p)[[x_1,\dots,x_{n_{\mathfrak p}}]]$. Hence, for any prime ideal $\mathfrak p' \subset \mathfrak p$, the localization $R_{\mathfrak p'}$ is formally a complete intersection. It follows that $R_{\mathfrak p}$ is a complete intersection for any $\mathfrak p$ in the sense of Grothendieck.
\end{proof}
\begin{lemma}[Purity lemma, {\cite[X. Theorem 3.4]{SGA2}}]
Suppose $R$ is a local ring. If $R$ is regular of dimension $\ge 2$, or a complete intersection with $\dim(R)\ge 3$, then $R$ is pure.
\end{lemma}
The following theorem for regular targets remains valid for targets pure in every codimension.
\begin{theorem}[Purity of the branch locus, {\cite[Lemma 58.21.4]{StacksProject}}]
Let $f:X\to Y$ be a dominant morphism between varieties. Assume that
\begin{enumerate}
\item $X$ and $Y$ are normal,
\item $Y$ is pure in every codimension,
\item $f$ is quasi-finite,
\item $f$ is \'etale in codimension one.
\end{enumerate}
Then $f$ is \'etale.
\end{theorem}
Thus, the stack presentation \([(U \times_X U)^\mathrm{nor} \rightrightarrows U]\)
defines a Deligne-Mumford stack with a natural morphism to algebraic space
\[
\{U/\mu_r \times_{X} U/\mu_r \rightrightarrows U/\mu_r \}\simeq X,
\]
which satisfies the pre-canonical condition.
By our assumption on the singular locus, the local atlases can be glued easily.
We conclude:
\begin{corollary}
For any well-formed variety $X$, Vistoli's construction produces a pre-canonical
Deligne-Mumford stack $\widetilde{X}$.
\end{corollary}
Let $\mathcal{X}$ be a normal Deligne-Mumford stack of dimension $\ge 3$.
Assume that for any geometric point $x \to \mathcal{X}$, there exists an \'etale atlas
\[
[U_x / G_x] \to \mathcal{X}
\]
such that the following conditions hold:
\begin{enumerate}
\item $U_x$ is a normal and $\mathrm{R2}$ whose singular locus is a smooth subvariety $C$;
\item for each $C$, the formal completion of $U_x$ along $C$ satisfies
\[
\widehat{\mathcal{O}}_{U_x, C} \cong \mathcal{O}_{C}[[x_1, \dots, x_s]] / I_i,
\]
where $I_i$ is a locally complete intersection ideal.
\item $G_x = \mu_r$ is the stabilizer group of $x$, acting generically freely, with $\gcd(r, \operatorname{char} \mathrm{k}) = 1$;
\end{enumerate}
Such a Deligne-Mumford stack $\mathcal{X}$ is said to be \emph{well-founded}.\\

It remains to show that the classical properties of pre-canonical stacks are satisfied for our construction, even in the presence of singularities. We first adjust Definition 4.2 of \cite{FMN} for a broader result.
\begin{definition}
A morphism $f \colon X \to Y$ between stacks is called \emph{codimension preserving} if $f$ is dominant and for any small locus $Z \subset Y$ (i.e., $\operatorname{codim}_Y Z \ge 2$), the inverse image $f^{-1}(Z)$ is also small (i.e., $\operatorname{codim}_X f^{-1}(Z) \ge 2$).
\end{definition}
\begin{lemma}[Theorem 4.6, {\cite{FMN}}; Lemma 5.1, {\cite{Kacrepant}}]\label{canonicalstackkey}
Let $\mathcal{Y}$ be a pre-canonical and well-founded Deligne--Mumford stack with coarse moduli space $Y$, and let $\varepsilon:\mathcal{Y}\to Y$ be the structure morphism. Let $f:X\to Y$ be a codimension-preserving morphism, where $\mathcal{X}$ is a well-founded Deligne--Mumford stack. Then there exists a unique morphism $g:\mathcal{X}\to\mathcal{Y}$, unique up to a unique $2$-arrow, such that the following diagram commutes:
\[
\begin{tikzcd}[column sep=1.5em,row sep=1.5em]
\mathcal{X} \arrow[dr,"f"'] \arrow[r,dashed,"g"] & \mathcal{Y} \arrow[d,"\varepsilon"] \\
& Y
\end{tikzcd}
\]
\end{lemma}

\begin{proof}
The uniqueness in the original proof does not rely on any properties beyond normality and separatedness, i.e. \cite[Proposition A.1]{FMN}.
It suffices to consider the existence of a lift.
Let $\mathcal{Y}$ be a pre-canonical Deligne-Mumford stack and consider a local atlas along a singular locus $C$,
\[
\{ [U/\mu_r] \times_{\mathcal{Y}} [U/\mu_r] \rightrightarrows [U/\mu_r] \}.
\]
Without loss of generality, assume $[U/\mu_r] \to \mathcal{Y}$ is quasi-finite and \'etale.
Then $[U/\mu_r] \times_{\mathcal{Y}} [U/\mu_r]$ can be represented similarly as $[U^{(1)}/G]$,
where $U^{(1)} \to U$ is \'etale, so that all conditions for a well-formed cover are also satisfied.
A similar procedure applies to a local well-formed cover $[W/H]$ of $X$.
Thus, a morphism $\mathcal{X} \to \mathcal{Y}$ is equivalent to a compatible system of maps
$W^{(i)} \to [U^{(i)}/G]$ satisfying the usual descent conditions. Suppose we are given a morphism $\mathcal{X} \to Y$,
which induces locally a morphism $W^0 \to [U^0/G]$ on an open subset, we may assume that $W^0 \subset W$ and $U^0 \subset U$
are open subsets whose complements have codimension at least $2$, and that both $W^0$ and $U^0$ are smooth. It suffices to show that any $G$-principal bundle over $W^0$ admits a unique extension to a $G$-principal bundle over $W$,
since by \cite[Lemma~4.1]{FMN} any $G$-equivariant morphism to $U^0$ uniquely extends to $U$.
Moreover, it is enough to take $W^0$ as the smooth locus of $W$,
since extensions from smooth loci are canonical.
\begin{claim}
There is a natural isomorphism
\[
\mathrm{H}^1_{\mathrm{\acute{e}t}}(W^0, \mu_r) \;\simeq\; \mathrm{H}^1_{\mathrm{\acute{e}t}}(W, \mu_r).
\]
\end{claim}
Since Kummer short exact sequence
\[
0 \longrightarrow \mu_r \longrightarrow \mathbb{G}_m \xrightarrow{(\cdot)^r} \mathbb{G}_m \longrightarrow 0,
\]
induces long exact sequence in étale cohomology
\[
\cdots \longrightarrow \mathrm{H}^0_{\mathrm{\acute{e}t}}(W, \mathbb{G}_m) \xrightarrow{(\cdot)^r} \mathrm{H}^0_{\mathrm{\acute{e}t}}(W, \mathbb{G}_m)
\longrightarrow \mathrm{H}^1_{\mathrm{\acute{e}t}}(W, \mu_r)
\longrightarrow \mathrm{H}^1_{\mathrm{\acute{e}t}}(W, \mathbb{G}_m) \xrightarrow{(\cdot)^r} \mathrm{H}^1_{\mathrm{\acute{e}t}}(W, \mathbb{G}_m) \longrightarrow \cdots
\]
Notice that $\mathrm{H}^0_{\mathrm{\acute{e}t}}(W^0, \mathbb{G}_m) \simeq \mathrm{H}^0_{\mathrm{\acute{e}t}}(W, \mathbb{G}_m)$ because $W$ is normal, and that $\mathrm{H}^1_{\mathrm{\acute{e}t}}(W^0, \mathbb{G}_m) \simeq \operatorname{Pic}(W^0) \simeq \operatorname{Cl}(W)$, $\mathrm{H}^1_{\mathrm{\acute{e}t}}(W, \mathbb{G}_m) \simeq \operatorname{Pic}(W)$. It therefore suffices to show that the $r$-torsion parts $\operatorname{Pic}(W)[r]$ and $\operatorname{Cl}(W)[r]$ are naturally isomorphic. Assume that $W \setminus W^0 = C$ is an irreducible center, and let $M$ be a reflexive rank-$1$ sheaf on $W$ such that $M^{(r)} \simeq \mathcal{O}_W$. Consider its restriction to the formal neighbourhood of any closed point in $C$, corresponding to a local ring $\operatorname{Spec} A$. Then $M$ defines an $r$-torsion element in the local divisor class group $\operatorname{Cl}(A)$. However, by construction $A$ is pure in every codimension. Hence, by \cite[Lemma 4]{Cutkosky}, $\operatorname{Cl}(A)$ has no torsion coprime to $\operatorname{char}(\mathrm{k})$. It follows that $M$ is free at every such restriction, and therefore $M$ is locally free on $W$.
\end{proof}
We obtain the following results.
\begin{corollary}[Canonical property of pre-canonical stacks]
Let $\mathcal{X}$ (resp.\ $\mathcal{Y}$) be a pre-canonical well-founded Deligne-Mumford stack with coarse moduli space $X$ (resp.\ $Y$), and let $f : X \to Y$ be an isomorphism. Then there exists a unique lifting isomorphism $\tilde{f} : \mathcal{X} \to \mathcal{Y}$.
\end{corollary}
\begin{corollary}
Vistoli's construction defines a fully faithful functor from the category of well-formed varieties to the category of well-founded Deligne-Mumford stacks with codimension preserved morphisms. In the absence of ambiguity, the image of this functor may be referred to as the \emph{well-formed stack}. We usually denote this functor by $X \mapsto \widetilde{X}$ and call $\widetilde{X}$ the \emph{canonical stack} of $X$.
\end{corollary}
However, we should note that this functor does not preserve certain morphism properties, such as closed immersions, since the latter may fail to be representable.
\begin{example}
We note that a basic class of examples of well-formed singularities is given by cyclic quotients of isolated hypersurface singularities. Up to formal completion, we may choose coordinates such that
\[
\widehat{R} \cong \mathrm{k}[[x_1,\dots,x_n]]/(f),
\]
equipped with a diagonal $\mu_r$-action.
By the $\mu_r$-equivariant version of the Weierstrass preparation theorem, which applies in this setting since a basis of $\mu_r$-eigenvectors is given by monomials, one obtains an equivariant form of the implicit function theorem for $\mathrm{k}[[x_1,\dots,x_n]]$.
It follows that there exist $\mu_r$-equivariant formal coordinates $x_i'$ satisfying $x_i' \equiv x_i \bmod \mathfrak m J(f)$ such that $f$ is equivalent to a polynomial in $\mathrm{k}[[x_1',\dots,x_n']]$, and the $\mu_r$-action remains diagonal and algebraic in these coordinates. The proof is identical to that of \cite[Theorem 9.1.4]{DeJongPfister}.
\end{example}
Besides Vistoli's construction, there is also a local Cox realization. Let $X = \Spec R$ be a well-formed variety with a singular center $C$, and assume that $\operatorname{Cl}_{\mathrm{tor}}(R) \simeq \mu_r$. One can then define
\[
\operatorname{Cox}^{\mathrm{loc}}(X) := \Spec_X \Bigl( \bigoplus_{M \in \operatorname{Cl}_{\mathrm{tor}}(R)} M \Bigr),
\]
which carries a natural $\mu_r$-action. This yields a quotient stack
\[
[\operatorname{Cox}^{\mathrm{loc}}(X)/\mu_r].
\]
In particular, if $\operatorname{Cl}_{\mathrm{tor}}(R) \simeq \operatorname{Cl}_{\mathrm{tor}}(R_{\widehat{C}})$, then after completion along $C$ and some group cohomology computations using \cite[Sec.\ 16]{Fossum} the above construction yields a pre-canonical stack of $X$. This construction is standard when $X$ is a germ of an analytic space. If the $\mathds{Q}$-index of $K_X$ is exactly $r$, then
\[
\operatorname{Cox}^{\mathrm{loc}}(X) \simeq \Spec_X \Bigl( \bigoplus^{r-1}_{i=0} \mathcal{O}_X(-iK_X) \Bigr),
\]
which recovers the index-one cover. This construction is similar to \cite[Definition 6.1]{KawamataDK}.
\begin{remark}
In the definition of the well-formed singularities, one could allow $C$ to have certain quotient singularities, or more generally consider local complete intersection quotients by a finitely generated abelian group $G$, i.e. formal toric varieties. However, such generalizations typically introduce only combinatorial complications without yielding essentially new phenomena.
\end{remark}
\subsection{Bourbaki sequence}
To simplify the discussion, we first consider a normal integral formal ring
\[
S = R/I, \qquad R = \mathrm{k}[[x_1,\dots,x_n]],
\]
where $S$ has codimension $c$ and $n$ equals its embedding dimension; it is also equipped with a diagonal $\mu_r$-action. We assume that the action is free in codimension one, with weights $a_i \ge 0$ on the variables $x_i$, that $\mathrm{k}$ is an infinite field, and that $\operatorname{char}(\mathrm{k})$ is coprime to $r$. The $\mu_r$-action induces a natural equivariant decomposition
\[
R = \bigoplus_{i=0}^{r-1} R_i, \qquad S = \bigoplus_{i=0}^{r-1} S_i,
\]
where $R_i = (R \otimes \nu^i)^{\mu_r}$, and the decomposition descends to $S$ since $f$ is $\mu_r$-equivariant. For any $\mu_r$-equivariant prime ideal $\mathfrak{p}$ (or any equivariant element), define the equivariant localization
\[
S_{[\mathfrak{p}]} := \Bigl\{ a/b \ \Big|\ a,b \in S \text{ are equivariant elements, } b \notin \mathfrak{p} \Bigr\}.
\]
The $\mu_r$-representation on $S$ naturally induces a decomposition
\[
S_{[\mathfrak{p}]} = \bigoplus_{i=0}^{r-1} (S_{[\mathfrak{p}]})_i, \quad
(S_{[\mathfrak{p}]})_i := \Bigl\{ a/b \in S_{[\mathfrak{p}]} \ \Big|\ a,b \text{ are equivariant and the weight difference is } \nu^i \Bigr\}.
\]
In particular, if $\mathfrak{p} = 0$, we obtain a decomposition of the $\mu_r$-equivariant function field $K := K([S/\mu_r])$.
Let $M$ be a finitely generated, $\mu_r$-equivariant, torsion-free $S$-module. Any $\mu_r$-equivariant basis $\{m_i\}$ induces the natural equivariant decomposition $M = \bigoplus_{i=0}^{r-1} M_i$, and by torsion-freeness tensoring with the function field $K = K([S/\mu_r])$ gives a natural embedding e.g. \cite{BourbakiCommutativeAlgebra},
\[
M \hookrightarrow \bigoplus_{i=1}^{\operatorname{rank} M} K \otimes \nu^{n_i},
\]
endowing $M$ with the structure of a $\mu_r$-equivariant lattice. We have the following result which generalizes \cite{BourbakiCommutativeAlgebra} and \cite[Corollary 2.4]{Kumashiro}.
\begin{lemma}\label{Bexactsequence}
Let $M$ be any $\mu_r$-equivariant torsion-free $S$-module. Then there exists a $\mu_r$-equivariant short exact sequence
\[
0 \longrightarrow \bigoplus_{i=1}^{\operatorname{rank}(M)-1} S \otimes \nu^{m_i} \longrightarrow M \longrightarrow I \otimes \nu^{m_0} \longrightarrow 0,
\]
where $I \subset S$ is a $\mu_r$-equivariant ideal.
\end{lemma}

\begin{proof}
It suffices to show the following. Let $M$ be a $\mu_r$-equivariant torsion-free $S$-module. If $\operatorname{rank}(M) \ge 2$, there exists a $\mu_r$-equivariant element $m \in M$ such that, under the natural map
\[
S \longrightarrow S_{[{\mathfrak p}]} \longrightarrow \mathrm{k}([{\mathfrak p}]) := S_{[{\mathfrak p}]} / {\mathfrak p} S_{[{\mathfrak p}]},
\]
its image $m([\mathfrak p])$ is nonzero for any $\mu_r$-equivariant prime ideal of height one, yielding a $\mu_r$-equivariant injective morphism
\[
0 \longrightarrow S \otimes \nu^{m_1} \longrightarrow M \longrightarrow M' \longrightarrow 0,
\]
with $M'$ torsion-free of rank $\operatorname{rank}(M)-1$ by \cite[Chap. VII, §4.9, Lemma 7]{BourbakiCommutativeAlgebra}. For torsion-free equivariant rank-$1$ modules $M$, there exists a $\mu_r$-equivariant embedding $M \hookrightarrow K \otimes \nu^{m_0}$, so that there is a $\mu_r$-equivariant element $f$ with $f M \subset S$, and hence $M \simeq I \otimes \nu^{m_0}$ for some $\mu_r$-equivariant ideal $I \subset S$.

\begin{claim}
For any $\mu_r$-equivariant prime ideal $\mathfrak{p}$ of height one, all equivariant components $M_i$ in the decomposition of $M$ have nonzero image under the natural map
\[
M \longrightarrow M \otimes_S \mathrm{k}([\mathfrak{p}]).
\]
\end{claim}
If not, suppose that there exists $i$ such that $M_i \subset \mathfrak{p} M_{[\mathfrak{p}]}$. Choose a nonzero $\mu_r$-equivariant element $m_j \in M_j$ whose image in $M \otimes_S \mathrm{k}([\mathfrak{p}])$ is nonzero; such an element exists since, by torsion-freeness, $M_{[\mathfrak{p}]} \simeq S_{[\mathfrak{p}]}^{\operatorname{rank}(M)}$, so there exists a nonzero element $a/b \in M_{[\mathfrak{p}]}$ with nonzero image, and we may take $m_j = a$. Consider the $i$-th equivariant component $(S m_j)_i$. By assumption, it satisfies
\[
(S m_j)_i \subset \mathfrak{p} M_{[\mathfrak{p}]} \cap (S m_j)_i = \bigl(\mathfrak{p} M_{[\mathfrak{p}]} \cap S m_j\bigr)_i.
\]
On the other hand, we have
\[
\mathfrak{p} M_{[\mathfrak{p}]} \cap S m_j
= (\mathfrak{p} S_{[\mathfrak{p}]})^{\operatorname{rank}(M)} \cap S m_j
= \mathfrak{p} S m_j.
\]
It follows that $(S m_j)_i \subset (\mathfrak{p} S m_j)_i$, hence $S_{k} \subset (\mathfrak{p} S)_{k}$, where $k \equiv i - j \pmod r$ for certain equivariant prime height one.
We claim that $\mathfrak{p}$ contains an ideal height $2$. Let $\bar{\mathfrak{p}}$ denote the inverse image of $\mathfrak{p}$ in $R$. By assumption $ I \subset \mathfrak{m}^2$ and irreducible, any subvariety of $S$ defined by a single equation $x_i=0$ has height one. By the assumption that the $\mu_r$-action is free in codimension one, the weights satisfy $\gcd(a_1,\dots,\widehat{a_i},\dots,a_n,r)=1$. Hence there exists a monomial $m$ not involving $x_i$ of weight $1$ in $R_1$, and thus $m^{\mathrm{k}} \in R_{\mathrm{k}}$ for all $\mathrm{k}$. If $S_{k} \subset \mathfrak{p}$, then there exists $i_{(1)} \ne i$ such that $x_{i_{(1)}} \in \bar{\mathfrak{p}}$. Since $Z(x_{i_{(1)}})$ has height one in $S$, the same argument shows that $\gcd(a_1,\dots,\widehat{a_{i_{(1)}}},\dots,a_n,r)=1$, and hence there exists ${i_{(2)}} \ne {i_{(1)}}$ with $x_{i_{(2)}} \in \bar{\mathfrak{p}}$. By Krull's Principal Ideal Theorem, since $\langle x_{i_{(1)}},x_{i_{(2)}} \rangle \subset S$ is nontrivial, then $1\leq\operatorname{ht}_S(\langle x_{i_{(1)}},x_{i_{(2)}} \rangle) \le 2$. In the case $\operatorname{ht}_S(\langle x_{i_{(1)}},x_{i_{(2)}} \rangle) = 1$, it follows that the codimension-one locus
\[
\operatorname{Spec}\Bigl(\mathrm{k}[x_1,\ldots,x_n]/\langle I, x_{i_{(1)}}, x_{i_{(2)}} \rangle\Bigr) \subset \operatorname{Spec} S,
\]
and by freeness in codimension one we have
\[
\gcd(a_1,\dots,\widehat{a_{i_{(1)}}},\ldots,\widehat{a_{i_{(2)}}},\dots,a_n,r)=1.
\]
Thus there exists ${i_{(3)}} \ne {i_{(1)}},{i_{(2)}}$ such that $x_{i_{(3)}} \in \bar{\mathfrak{p}}$. In particular, $x_{i_{(1)}},x_{i_{(2)}},x_{i_{(3)}}$ form a regular sequence in $\bar{\mathfrak{p}}$, and we replace the ideal $\langle x_{i_{(1)}},x_{i_{(2)}} \rangle$ in $S$ with $\langle x_{i_{(1)}},x_{i_{(2)}},x_{i_{(3)}} \rangle$. If $\operatorname{ht}_S(\langle x_{i_{(1)}},x_{i_{(2)}},x_{i_{(3)}} \rangle) = 1$, we repeat the above procedure until we obtain an ideal of height not one in $S$. If this process does not terminate, we eventually obtain an ideal
\[
\langle x_{i_{(1)}},x_{i_{(2)}},x_{i_{(3)}},\dots \rangle \subset \mathfrak{p}
\]
of length $c+2$, since $\operatorname{ht}_R(I) = c$ we have $\operatorname{ht}_S(\mathfrak{p}) > 1$, yielding a contradiction. \\

We return to the proof of the main lemma. Choose a nonzero $\mu_r$-equivariant element $n_i \in M_i$. Then there exist only finitely many $\mu_r$-equivariant prime ideals $\mathfrak{p}$ of height one such that the image $n_i([\mathfrak{p}])$ in $M \otimes_S \mathrm{k}([\mathfrak{p}])$ vanishes. Proceeding by induction on the number of such primes, we may assume that there is exactly one such $\mathfrak{p}$.
By the claim, we may choose a nonzero element $n'_i \in M_i$ such that $n'_i([\mathfrak{p}]) \ne 0$. Moreover, there are only finitely many $\mu_r$-equivariant prime ideals $\mathfrak{p}$ of height one for which $n'_i([\mathfrak{p}])$ and $n_i([\mathfrak{p}])$ are linearly dependent over $\mathrm{k}([\mathfrak{p}])$ by \cite[Chap. VII, §4.9, Lemma 7]{BourbakiCommutativeAlgebra}, say $n_i([\mathfrak{p}]) = \alpha([\mathfrak{p}]) n'_i([\mathfrak{p}])$ with $\alpha([\mathfrak{p}]) \in \mathrm{k}([\mathfrak{p}])$.
Consider $m = n_i + \alpha n'_i$ with $\alpha \in \mathrm{k}$. Since $\mathrm{k}$ is infinite, we may choose non-zero $\alpha$ such that $\alpha \ne \alpha([\mathfrak{p}])$ for all such $\mathfrak{p}$. It follows that $m([\mathfrak{p}]) \ne 0$ for every $\mu_r$-equivariant prime ideal $\mathfrak{p}$ of height one.
\end{proof}

\begin{remark}
The condition that the $\mu_r$-action is free in codimension one is in general difficult to verify. However, there is a weaker numerical criterion. Suppose $R/I$ is equipped with a diagonal $\mu_r$-action with weights $a_1,\dots,a_n$. If for any subset $A \subset \{a_1,\dots,a_n\}$ of cardinality $d-1$ where $d := n-c$ one has \(\gcd(A,r)=1\), then the action is free in codimension one.
\end{remark}

We consider a natural variation to a normal integral homogeneous coordinate ring
\[
S = R/I, \qquad R = \mathrm{k}[x_1,\dots,x_n],
\]
where $S$ has codimension $c$ with $n$ equal to its embedding dimension and is equipped with a diagonal $\mathbb{G}_m$-action with weights $a_i \ge 0$ on the variables $x_i$. We assume that $\mathrm{k}$ is an infinite field and that $\operatorname{char}(\mathrm{k})$ is coprime to $r$. Let $Z = \operatorname{Spec}(S/J)$, where $J$ is the ideal generated by those coordinates $x_1,\dots,x_b$ with nonzero weight. Consider the weighted projective stack
\[
\operatorname{Proj}^{(a_i)} S := [(\operatorname{Spec} S \setminus Z)/\mathbb{G}_m].
\]
assume that this stack which is naturally realized as a closed substack of \(\mathcal{P}(a_1,\dots,a_b) \times \mathbb{A}^{n-b}\) defined by the ideal $I$ has trivial stabilizer in codimension one. Fix any integer $N \in \mathbb{Z}$. The $\mathbb{G}_m$-action induces a truncated $\mathbb{N}$-graded decomposition
\[
R = \bigoplus_{i \ge N} R_i, \qquad S = \bigoplus_{i \ge N} S_i.
\]
Since $I$ is homogeneous, the decomposition on $R$ descends to $S$. For any $\mathbb{G}_m$-equivariant prime ideal $\mathfrak{p}$ (or any homogeneous element), define the equivariant localization with a $\mathbb{N}$-graded decomposition
\[
S_{[\mathfrak{p}]} = \bigoplus_{i\geq N} (S_{[\mathfrak{p}]})_i, \quad
(S_{[\mathfrak{p}]})_i := \Bigl\{ a/b \in S_{[\mathfrak{p}]} \ \Big|\ a,b \text{ are homogeneous and } \deg(a)-\deg(b)=i \Bigr\}.
\]
We have a similar result as follows.
\begin{lemma}
Let $M$ be a $\mathbb{G}_m$-equivariant torsion-free $S$-module. Then there exists an integer $N$, such that the truncated module $M_{\ge N}$ fits into a $\mathbb{G}_m$-equivariant short exact sequence
\[
0 \longrightarrow \bigoplus_{i=1}^{\operatorname{rank}(M)-1} S \otimes \chi^{m_i} \longrightarrow M_{\ge N} \longrightarrow I \otimes \chi^{m_0} \longrightarrow 0,
\]
where $I \subset S$ is a $\mathbb{G}_m$-equivariant ideal.
\end{lemma}
\begin{proof}
The argument is essentially the same as Lemma \ref{Bexactsequence}. For any $\mathbb{G}_m$-equivariant prime ideal $\mathfrak{p}$ of height one, there exists an integer $N$, depending only on $S$ and the $\mathbb{G}_m$-action, such that all homogeneous components $M_i$ with $i \ge N$ in the decomposition of $M$ have nonzero image under the natural map
\[
M \longrightarrow M \otimes_S \mathrm{k}([\mathfrak{p}]).
\]
This follows from Frobenius coin problem: if $\gcd(a_1,\dots,a_n) = 1$, then there exists a positive integer $N$ such that any $n \ge N$ can be expressed as $n = \sum r_i a_i$ with $r_i \ge 0$. Combined with the assumption that the stack has no stabilizers in codimension one, the claim is clear.
\end{proof}

\begin{corollary}\label{Bexactwps}
Let $\mathcal{X}$ be a normal integral substack of a weighted projective stack $\mathcal{P}$ over infinite field $\mathrm{k}$, and assume that $\mathcal{X}$ has trivial stabilizers in codimension one. Then for any coherent torsion-free sheaf $\mathcal{F}$ on $\mathcal{X}$, there exists an exact sequence
\[
0 \longrightarrow \bigoplus_{i=1}^{\operatorname{rank}(\mathcal{F})-1} \mathcal{O}_{\mathcal{X}}(m_i)
\longrightarrow \mathcal{F}
\longrightarrow \mathcal{I} \otimes \mathcal{O}_{\mathcal{X}}(m_0)
\longrightarrow 0,
\]
where $\mathcal{I}$ is an ideal sheaf.
\end{corollary}
\begin{proof}
Consider the Serre functor $\Gamma \colon \mathrm{Coh}(\mathcal{X}) \to \mathrm{Qgr}(S):\Pi$. If $\mathcal{F}$ is torsion-free, then all non-zero associated primes of $\Gamma(\mathcal{F})$ are contained in the irrelevant locus; since $\Gamma(\mathcal{F})$ is finitely generated, for sufficiently large $N$ we have $\operatorname{Ass}(\Gamma(\mathcal{F})_{\ge N}) = \{0\}$, so $\Gamma(\mathcal{F})_{\ge N}$ is torsion-free. Applying the above lemma and descending via $\Pi$ yields the desired short exact sequence.
\end{proof}

\begin{remark}
\begin{enumerate}
\item In Lemma 2.13, if one ignores the truncation, one may still expect a similar $\mathbb{G}_m$-equivariant short exact sequence, as in \cite[Theorem 2.1]{HKS}.
\item If the base field is not infinite, the original Bourbaki-style argument \cite{BourbakiCommutativeAlgebra} is no longer available.
\item The above conclusions can be suitably generalized to certain GIT quotient stacks for a finitely generated abelian group $G$, requiring only that there are trivial stabilizers in codimension one. A classical example is a normal integral subvariety of the (formal) stacky vector bundle
\[
\mathds{A} := \mathbb{A}_{\mathcal{P}}\Bigl[\Bigl[\bigoplus_i \mathcal{O}_{\mathcal{P}}(b_i)\Bigr]\Bigr]  
\]
over a weighted projective stack $\mathcal{P}$.
\item However, for a general ring with a $\mu_r$-quotient, the author is not clear whether a similar result holds, and we do not pursue this further here. In what follows, the Bourbaki sequence need not be
indispensable or unique. It does, however, offer a clean
construction, with the limitation that non-normal rings
are not covered.
\end{enumerate}
\end{remark}

\subsection{Cohen-Macaulay approximation}
We first consider the formal neighborhood of a closed point $x$ on the singular locus $C$ of a well-formed variety, given by
\[
S = \mathcal{O}_{C,x}[[x_1,\dots,x_s]]/I,
\]
equipped with a diagonal $\mu_r$-action with weights $a_1,\dots,a_s$, where $a_i\geq0$.  We assume that $S$ is a normal local complete intersection; in particular, it is Gorenstein and hence Cohen--Macaulay. Moreover, we have an integral extension $S_0 := S^{\mu_r} \subset S$, and $S_0$ is also Cohen--Macaulay. It follows that for any finitely generated $S_0$-module, one has a classical Cohen--Macaulay approximation.
\begin{theorem}[Theorem A, {\cite{AB}}]
Let $(R,\mathfrak{m},\mathrm{k})$ be a commutative Noetherian local ring admitting a dualizing module $\omega$. For any finitely generated $R$-module $N$, there exist finitely generated $R$-modules $M_N$ and $I_N$, together with an $R$-linear map
\[
d_N:M_N\longrightarrow I_N,
\]
such that:
\begin{enumerate}
\item The image of $d_N$ is isomorphic to $N$.

\item $M_N$ is maximal Cohen--Macaulay and $I_N:=\ker(d_N)$ has finite injective dimension.

\item $\operatorname{coker}(d_N)$ is maximal Cohen--Macaulay.

\item There exists an integer $n\ge0$ such that $d_N$ factors as
\[
M_N \xrightarrow{\,j\,} \omega^{\oplus n} \xrightarrow{\,p\,} I_N,
\]
where $j$ is injective and $p$ is surjective.
\end{enumerate}
\end{theorem}
In our situation, the Auslander--Buchweitz approximation theorem yields that for any finitely generated $S_0$-module $N$, there exists a short exact sequence
\[
0 \longrightarrow I_N \longrightarrow M_N \longrightarrow N \longrightarrow 0,
\]
or equivalently
\[
0 \longrightarrow N \longrightarrow I_N \longrightarrow M^N \longrightarrow 0,
\]
where $M_N$ is maximal Cohen--Macaulay and $I_N$ has finite injective dimension. Moreover, both modules are finitely generated over $S_0$. It is straightforward to extend the above construction to the derived category. For any $F \in \mathrm{D}^b(S_0)$, choose a locally free bounded above resolution $P^\bullet \xrightarrow{\sim} F$. By boundedness, for $k \ll 0$ the naive truncation $\sigma^{\le k} P^\bullet$ is acyclic and hence quasi-isomorphic to a shifted sheaf $\mathcal{N}[m]$, with cone lying in $\mathrm{D}^{\mathrm{perf}}(S_0)$. Using the above properties, this yields a two-step filtration of $F$.

For convenience, we give a geometric formulation, $\varepsilon : \mathcal{X} := [\operatorname{Spec} S/\mu_r] \to X := \operatorname{Spec} S_0$, and let $\kappa \colon [C/\mu_r] \hookrightarrow \mathcal{X}$, $\kappa_X \colon C \hookrightarrow X$ be the natural embeddings.
\begin{lemma}\label{finiteinjdim}
Let $I \in \mathrm{D}^b(X)$ be an object of finite injective dimension. Then $\varepsilon^! I$ also has finite injective dimension.
\end{lemma}
\begin{proof}
Note that $S$ and $S_0$ are local with residue fields $\mathrm{k}_S \cong \mathrm{k}_{S_0} \cong \mathrm{k}$. Finite injective dimension is equivalent to the boundedness of $\mathrm{Ext}^*(\mathrm{k},I)$. By Grothendieck--Verdier duality,
\[
\mathrm{Ext}^*_{\mathcal{X}}(\mathrm{k}, \varepsilon^! I) \cong \mathrm{Ext}^*_X(\mathrm{k}, I),
\]
which is bounded; hence $\varepsilon^! I$ has finite injective dimension.
\end{proof}
\begin{lemma}\label{projinjdim}
For any $F \in \mathrm{D}^b(\mathcal{X})$, $F$ has finite injective dimension if and only if $F$ has finite projective dimension.
\end{lemma}
\begin{proof}
We work locally, hence $\mathcal{X}$  has the resolution property. First assume that $F$ has finite projective dimension. Take a finite projective resolution of $F$. For $k \ll 0$, consider a sufficiently negative truncation; its cohomology yields a Cohen--Macaulay approximation $M$ of $F$. Since projective objects are locally free and the ring is (equivariant) local Gorenstein, $M$ is finite free. In particular, using $\operatorname{inj.dim}(\mathcal{O}_{\mathcal{X}}) = \dim \mathcal{X}$, we deduce that $F$ has finite injective dimension.

Conversely, assume that $F$ has finite injective dimension. Consider a similar Cohen--Macaulay approximation $M$ of $F$, it suffices to show that $M$ has finite projective dimension. Since $\mathcal{X}$ is Gorenstein, Serre duality gives a contravariant equivalence on maximal Cohen--Macaulay objects
\[
\operatorname{Ext}_{\mathcal{X}}^*(M,\mathrm{k}) \cong \operatorname{Ext}_{\mathcal{X}}^{-*}(\mathrm{k}[-\dim \mathcal{X}], M^\vee),
\]
so finite injective dimension of $M$ is equivalent to finite projective dimension of $M^\vee$. Hence $M^\vee$ has finite projective dimension, thus finite injective dimension again. Applying duality once more and using the natural identification \(M \simeq (M^\vee)^\vee\) conclude that $M$ has finite projective dimension.
\end{proof}
\begin{lemma}\label{kernalapporximation}
We obtain a distinguished triangle
\[
\tau \longrightarrow \varepsilon^{*}(\kappa_X)_* \mathcal{O}_{C} \longrightarrow \kappa_* \mathcal{O}_{C} \longrightarrow \tau[1],
\]
where $\tau$ is quasi-isomorphic to a complex supported in non-positive degrees, and for any naive truncation $\sigma^{\geq k}\tau$, it can be generated by
\[
\kappa_* \mathcal{O}_{C} \otimes \nu^{k}, \quad k = 1, \dots, r-1.
\]
\end{lemma}
\begin{proof}
We compute
\[
S \otimes^{\mathbf{L}}_{S_0} S_0 / \mathfrak{\overline{p}},
\]
where $S = \bigoplus_{i=0}^{r-1} S_i$ is the eigenspace decomposition. Here $\mathfrak{p}$ defines the center $C$ on $S$, and $\overline{\mathfrak{p}}$ is its image in $S_0$. We keep track of the $\mu_r$-equivariant structure throughout, and it suffices to compute $S_i \otimes^{\mathbf{L}}_{S_0} S_0 / \mathfrak{\overline{p}}$ for $i \neq 0$.
We take a projective $S_0$-resolution of $S_i$:
\[
\begin{tikzcd}
\cdots \arrow[r] & S_0^{m_3} \arrow[r, "F_3"] & S_0^{m_2} \arrow[r, "F_2"] & S_0^{m_1} \arrow[r, "F_1"] & S_0^{m_0}
\arrow[r, "F_0", "\simeq"'] & S_i \arrow[r] & 0
\end{tikzcd}
\]
where $F_0$ is defined by a minimal set of $\mu_r$-equivariant generators $m_0, \dots, m_\ell$ of $S_i$. Since $\bigoplus_{i=0}^{r-1} S_i$ carries a natural $\mu_r$-equivariant $S$-module structure, it induces a $\mu_r$-equivariant structure on the projective resolution. In particular, we have
\[
\begin{tikzcd}
\Big\{
\cdots \arrow[r] & S_0^{m_3} \arrow[r, "F_3"] & S_0^{m_2} \arrow[r, "F_2"] & S_0^{m_1} \arrow[r, "F_1"] & S_0^{m_0}
\Big\}
\otimes \nu^{-i}
\arrow[r, "F_0", "\simeq"'] & S_i \arrow[r] & 0
\end{tikzcd}
\]
Therefore $S_i \otimes^{\mathbf{L}}_{S_0} S_0 / \mathfrak{p}$ is a complex whose terms are direct sums of copies of $\mathcal{O}_C \otimes \nu^{-i}$.
\end{proof}
\begin{corollary}
Let $F \in \mathrm{D}^b(\mathcal{X})$ with $\varepsilon_* F = 0$, then
\[
F \in \mathcal{K}_- := \langle \kappa_* \mathcal{O}_C \otimes \nu^1, \dots, \kappa_* \mathcal{O}_C \otimes \nu^{r-1} \rangle.
\]
\end{corollary}
\begin{proof}
By local adjunction \cite[Lemma 3.7]{Hao2025b}, for any $F \in \mathrm{D}^b(\mathcal{X})$ there is an isomorphism
\[
\operatorname{Ext}^{*}_{\mathrm{D}^b(\mathcal{X})/\mathcal{K}_-}(\kappa_* \mathcal{O}_C, F)
\simeq
\operatorname{Ext}^{*}_{\mathrm{D}^b(X)}(\kappa_{X,*} \mathcal{O}_C, \varepsilon_* F).
\]
If $\varepsilon_* F = 0$, then $F$ can be obtained by a finite sequence of extensions from
\[
\kappa_* \mathcal{O}_C \otimes \nu^i, \quad i = 0, \dots, r-1.
\]
It follows from the long exact sequence of $\operatorname{Ext}$ that $\operatorname{Ext}^{*}_{\mathrm{D}^b(\mathcal{X})/\mathcal{K}_-}(F, F) = 0$. Hence $F \in \mathcal{K}_-$.
\end{proof}

\begin{lemma}\label{Brss}
For any $F \in \mathrm{D}^{+}(\mathcal{X})$, if $\varepsilon_* F = 0$, then for all $i$ we have
\[
\varepsilon_* \mathcal{H}^i(F) = 0.
\]
\end{lemma}
\begin{proof}
We consider the spectral sequence
\[
E_2^{p,q} := \mathbf{R}^p \varepsilon_* \, \mathcal{H}^q(F) \;\Longrightarrow\; \mathcal{H}^{p+q}(\varepsilon_* F).
\]
and $\mathbf{R}^p \varepsilon_* \mathcal{H}^q(F) = 0$ for $p>0$, we deduce that $\varepsilon_* \mathcal{H}^i(F) = 0$ for all $i$.
\end{proof}

\begin{corollary}\label{CBrss}
For any reflexive sheaf $M$ on $X$, the cohomological truncations $\tau^{\leq k}\varepsilon^! M$ for $k \gg 0$ give rise to a filtration of $\varepsilon^! M$, whose successive quotients are sheaves supported on $C$ and lie in the category $\mathcal{K}_-$.
\end{corollary}
\begin{proof}
 Consider the natural morphism $\widetilde{M} \to \varepsilon^! M$, where $\widetilde{M}$ denotes the natural extension of $M$ as a reflexive sheaf on $\mathcal{X}$. Its cone $\tau$ is bounded below. Since $\varepsilon$ is finite, $\varepsilon_* \simeq \mathbf{R}^0\varepsilon_*$ and $M$ is reflexive, $\varepsilon_* \widetilde{M} = \varepsilon_* \varepsilon^! M = M$. By Lemma \ref{Brss}, each of the cohomology sheaves of the cone $\tau$  lies in $\mathcal{K}_-$.
\end{proof}

\begin{corollary}
For any perfect complex $P$ on $X$, the cohomological truncations $\tau^{\leq k}\varepsilon^! P$ for $k \geq l$ induce a filtration of $\tau^{\leq l}\varepsilon^! P$ (for $l \gg 0$), whose successive quotients are sheaves supported on $C$ and lie in the subcategory $\mathcal{K}_-$.
\end{corollary}

\begin{proof}
If $P=\mathcal{O}_X^{\oplus a}$, the statement follows directly from the Corollary \ref{CBrss}. Now let $P$ be an arbitrary perfect complex. Since $P$ is quasi-isomorphic to a bounded complex of free sheaves, it suffices to treat the case of a morphism
\[
f:\mathcal{O}_X^{\oplus a} \longrightarrow \mathcal{O}_X^{\oplus b}
\]
representing $P$. By the adjunction isomorphism there exists $l$ big enough such that
\[
\operatorname{Hom}_{\mathcal{X}}(\widetilde{\mathcal{O}_X}^{\oplus a}, \varepsilon^!\mathcal{O}_X^{\oplus b})
=
\operatorname{Hom}_{\mathcal{X}}(\widetilde{\mathcal{O}_X}^{\oplus a}, \tau^{\leq l}\varepsilon^!\mathcal{O}_X^{\oplus b})
=
\operatorname{Hom}_{X}(\mathcal{O}_X^{\oplus a}, \mathcal{O}_X^{\oplus b}),
\]
we may lift $f$ to a morphism
\[
\widetilde{f}:\widetilde{\mathcal{O}_X}^{\oplus a} \longrightarrow \tau^{\leq l}\varepsilon^! \mathcal{O}_X^{\oplus b}.
\]
Repeating this construction along a locally free resolution of $P$, we obtain an object $\widetilde{P}$ fitting into a morphism $\widetilde{P}\to \varepsilon^! P$ such that $\varepsilon_*\widetilde{P}\simeq P$, and the cone $\tau$ has cohomology sheaves lying in $\mathcal{K}_-$.
\end{proof}

Thus, for any $F \in \mathrm{D}^b(X)$, take a sheaf approximation $N[k] \to F \to P \to N[k+1]$ with $N$ a coherent sheaf and $P$ a perfect complex, followed by a Cohen--Macaulay approximation $0 \to N \to I_N \to M^N \to 0$. Applying $\varepsilon^!$, we obtain $\varepsilon^! I_N \to \varepsilon^! M^N$. By Lemmas~\ref{finiteinjdim} and~\ref{projinjdim}, $\varepsilon^! I_N$ is quasi-isomorphic to a bounded perfect complex, hence for $k \gg 0$,
\[
\operatorname{Hom}_{\mathcal{X}}(\varepsilon^! I_N, \varepsilon^! M^N)
=
\operatorname{Hom}_{\mathcal{X}}(\varepsilon^! I_N, \tau^{\leq k} \varepsilon^! M^N)
=
\operatorname{Hom}_{X}(I_N, M^N).
\]
Set $\widetilde{N} := \tau^{\leq k} \varepsilon^! N$, which is bounded and fits into a distinguished triangle $\widetilde{N} \to \varepsilon^! I_N \to \tau^{\leq k} \varepsilon^! M^N$, and satisfies $\varepsilon_* \widetilde{N} \simeq N$.  For a perfect complex $P$, using
\[
\operatorname{Ext}^1_{\mathcal{X}}(\widetilde{P}, \varepsilon^! N[k])
\simeq
\operatorname{Ext}^1_X(\widetilde{P}, \tau^{\leq l}\varepsilon^! N[k]),
\]
we obtain, by the same argument, an object $\widetilde{F}$ such that $\varepsilon_* \widetilde{F} \simeq F$, and the natural morphism $\widetilde{F} \longrightarrow \varepsilon^! \widetilde{F}$ has cone whose cohomology sheaves lie in $\mathcal{K}_-$. Thus we can define a (quasi-)functor $\widetilde{(-)} : \mathrm{D}^b(X) \to \mathrm{D}^b(\mathcal{X})$ with some ambiguity.
\begin{proposition}\label{canonicalcover}
The above construction induces a natural well-defined triangulated functor
\[
\widetilde{(-)} : \mathrm{D}^b(X) \xrightarrow{\sim} \mathrm{D}^b(\mathcal{X})/\mathcal{K}_-,
\]
which defines an equivalence of triangulated categories. In particular, for every $F \in \mathrm{D}^b(X)$, its image is isomorphic (for $l \gg 0$) to $\tau^{\le l}\varepsilon^! F$.
\end{proposition}

\begin{proof}
First, by the classical construction of the quotient category in \cite[Lemma 3.9]{Hao2025b}, we obtain natural isomorphisms
\[
\operatorname{Ext}^*_{\mathrm{D}^b(\mathcal{X})/\mathcal{K}_-}(G,\widetilde{F})
\simeq
\operatorname{Ext}^*_{\mathrm{D}^b(X)}(\varepsilon_* G, F),
\]
which induce the desired isomorphism on morphism spaces and hence define a fully faithful functor.
Next, for any distinguished triangle $F \to G \to H$ in $\mathrm{D}^b(X)$, it induces a triangle in $\mathrm{D}^b(\mathcal{X})$ of the form
\[
\widetilde{F} \longrightarrow \tau^{\leq l}\varepsilon^! G \longrightarrow \tau,
\]
where $\tau$ is bounded and satisfies $\varepsilon_* \tau \simeq H$, hence $\tau \simeq \widetilde{H}$ in $\mathrm{D}^b(\mathcal{X})/\mathcal{K}_-$. So the construction is compatible with distinguished triangles and therefore defines a triangulated functor.
Finally, for any $F \in \mathrm{D}^b(\mathcal{X})$, we have $F \simeq \widetilde{\varepsilon_* F}$ in $\mathrm{D}^b(\mathcal{X})/\mathcal{K}_-$, and for any $G \in \mathrm{D}^b(X)$, one has $\varepsilon_* \widetilde{G} \simeq G$. This shows that the functor is essentially surjective and hence an equivalence with inverse induced by $\varepsilon_*$.
\end{proof}
We now return to the general situation where $\mathcal{X}$ is a well-formed stack with coarse moduli map $\varepsilon:\mathcal{X}\to X$, and with finite singular centers $C_i \subset X$. There is a well-defined Verdier pullback $\varepsilon^!: \mathrm{D}^+(\mathrm{Qcoh}(X)) \to \mathrm{D}^+(\mathrm{Qcoh}(\mathcal{X}))$. Étale locally, we may choose an atlas of the form $[V/\mu_r]$, fitting into a Cartesian diagram
\[
\begin{tikzcd}
{[V/\mu_r]} \arrow[r, hook, "\widetilde{\iota}"] \arrow[d, "\widetilde{\varepsilon}"] & \mathcal{X} \arrow[d, "\varepsilon"] \\
{V/\mu_r} \arrow[r, hook, "\iota"] & X
\end{tikzcd}
\quad \text{(étale locally)}.
\]
More generally, there is a natural isomorphism $\widetilde{\varepsilon}^! \, l^* \simeq \widetilde{l}^* \, \varepsilon^!$, which induces a functor $\mathrm{D}^b(X) \to \mathrm{D}^+(\mathcal{X})$ without assuming the resolution property for $X$. Since each $C_i$ is projective, hence quasi-compact, for any closed point of $C_i$ we may choose such a local chart $[V/\mu_r]$, in which the completed local ring is of the form $\mathcal{O}_{C,x}[[x_1,\dots,x_n]]/I$. These local models satisfy the assumptions of Proposition~\ref{canonicalcover}. It follows that for any $F \in \mathrm{D}^b(X)$, there exists $k \gg 0$ such that the truncation $\tau^{\le k}\varepsilon^! F$ satisfies the same local properties along all $C_i$.
\begin{proposition}\label{CMappro}
There is a natural equivalence of triangulated categories
\[
\widetilde{(-)} : \mathrm{D}^b(X) \xrightarrow{\sim} \mathrm{D}^b(\mathcal{X})/\mathcal{K}_-,
\]
where $\mathcal{K}_-$ is the full saturated subcategory generated by $\langle \kappa_*\mathrm{D}^b(C_i) \otimes \nu^j \mid i,\; 1 \le j < r \rangle$. Moreover, for every $F \in \mathrm{D}^b(X)$, its image is isomorphic to $\tau^{\le l}\varepsilon^! F$ for $l \gg 0$.
\end{proposition}

\begin{remark}
\begin{enumerate}
\item
By an analogous argument, one can construct \(\tau^{\ge l}\varepsilon^*F\) for \(l\ll0\) and show that it is isomorphic to \(\tau^{\le l}\varepsilon^!F\) in the quotient category.

\item
Suppose that the fixed locus of the group action admits a stratification with stabilizers
\[
\mu_{r_1}\subset \mu_{r_2}\subset \cdots \subset \mu_r,
\]
and centers
\[
C_1\supset C_2\supset \cdots \supset C.
\]
Then one obtains a sequence of partial stacky quotient moduli spaces. Considering the analogous quotient by the iterated stacky subcategories
\[
\langle \mathrm D^b(\widetilde{C_j})\otimes \nu_j^k\rangle,
\]
one obtains an analogous localization theorem. So the derived category of the rigidified stack is obtained as a Verdier quotient by the corresponding stacky subcategory.
\end{enumerate}
\end{remark}

\subsection{Kawamata blow-up}\label{Kawamatablowup}
We first consider the formal local neighborhood of a closed point of $C$ in a well-formed stack over infinite field. Thus $X = \operatorname{Spec} S_0$, where $S := \mathrm{k}[[x_1,\dots,x_n]]/I$ is endowed with a diagonal $\mu_r$-action of weights $a_1,\dots,a_n$, such that $x_1,\dots,x_s$ define the singular center $C$, with $a_i>0$ for $1\le i\le s$ and $a_i=0$ for $s+1\le i\le n$.\\

First we consider the case $I=0$. This induces a natural weighted valuation
\[
val : K([R/\mu_r]) \longrightarrow \mathbb{Z},
\]
defined by setting, for any $\mu_r$-equivariant element $f=\sum_i c_i m_i$ with $m_i$ monomials,
\[
val(f) := \min_i \deg(m_i),
\]
where $\deg(x_j):=a_j$. In particular, this defines a filtration of ideals
\[
P_i := \{\, f \in R \mid f \text{ is $\mu_r$-equivariant and } val(f) \ge i \,\}.
\]
We always assume that, with respect to the measured order of the valuation, the ideal $I$ admits a measured basis \cite[Definition 2.7]{Hao2025b} consisting of a regular sequence, and that the leading terms with respect to this valuation also form a regular sequence. This condition will be referred to as the \emph{measured condition}.
\label{cond:measured}

Then consider the stacky Proj of the $\mathbb{N}$-graded $R$-algebra $\bigoplus_{i\ge 0} P_i$, namely
\[
\mathcal{P}\mathrm{roj}_R\!\left(\bigoplus_{i\ge 0} P_i\right)
:=
\left[\operatorname{Spec}_R\!\left(\bigoplus_{i\ge 0} P_i\right)\setminus Z \,\big/\, \mathbb{G}_m\right],
\]
here \(Z\) denotes the irrelevant locus. Since each $P_i$ is $\mu_r$-equivariant, the graded algebra $\bigoplus_{i\ge 0} P_i$ carries a natural $\mu_r$-action, therefore form the quotient stack
\[
w\mathrm{Bl}^{(a_i)}_C \mathcal{X}
:=
\left[\mathcal{P}\mathrm{roj}_R\!\left(\bigoplus_{i\ge 0} P_i\right)\big/\mu_r\right].
\]
note that this is a $(\mathbb{G}_m \times \mu_r)$-quotient stack since the two group actions commute, the order of taking quotients can be interchanged. We denote its coarse moduli space by $w\mathrm{Bl}^{(a_i)/r}_C X$, and its canonical stack by $\widetilde{w\mathrm{Bl}^{(a_i)/r}_C X}$, and refer to them as the (stacky) Kawamata blow-up. We note that the weights $a_i$ are not required to satisfy $a_i<r$. Moreover, if $f \in R_i$, then $val(f) \equiv i \pmod{r}$.
\begin{lemma}\label{rootsblowup}
The $\mu_r$-action on $\mathcal{P}\mathrm{roj}_R\!\left(\bigoplus_{i\ge 0} P_i\right)$ is trivial along the exceptional divisor
\[
\mathrm{M}:= \mathcal{P}\mathrm{roj}\!\left(\bigoplus_{i\ge 0} P_i /P_{i+1}\right).
\]
In particular, there is a natural identification
\[
w\mathrm{Bl}^{(a_i)}_C \mathcal{X}
\simeq
\sqrt[r]{\,E_\mathcal{Y}\,/\,\widetilde{w\mathrm{Bl}^{(a_i)/r}_C X}\,}.
\]
\end{lemma}
\begin{proof}
Because the degree is compatible with the $\mu_r$-action, it is easy to see that each $P_i/P_{i+1}$ is $\mu_r$-equivariant. Hence, under the induced action on $\mathcal{P}\mathrm{roj}\!\left(\bigoplus_{i\ge 0} P_i /P_{i+1}\right)$, every point is fixed.
\end{proof}
For a general complete intersection $I$, we consider the embedding $\mathcal{X} \hookrightarrow [\operatorname{Spec} R/\mu_r]$. The birational transformation defined above gives the strict transform of $\mathcal{X}$, which yields the stacks $w\mathrm{Bl}^{(a_i)}_C \mathcal{X}$, $w\mathrm{Bl}^{(a_i)/r}_C X$, and $\widetilde{w\mathrm{Bl}^{(a_i)/r}_C X}$. Of course, if we require this to be a canonical stack we should additionally assume that $w\mathrm{Bl}^{(a_i)/r}_{C}X$ still has well-formed singularities and exceptional divisor is prime. Under this assumption, Lemma~\ref{rootsblowup} remains valid, and we obtain the following commutative diagram
\[
\begin{tikzcd}
\mathcal{Z} :=w\mathrm{Bl}^{(a_i)}_C \mathcal{X} \arrow[r,"\varpi"] \arrow[d,"\pi"] & \mathcal{Y} := \widetilde{w\mathrm{Bl}^{(a_i)/r}_C X} \arrow[d] \\
\mathcal{X} \arrow[r] & X
\end{tikzcd}
\]
\begin{lemma}
For any $F \in \mathrm{D}^b(\mathcal{X})$, there exists a distinguished triangle
\[
M[k] \longrightarrow F \longrightarrow P \in \Delta,
\]
where $P$ is a perfect complex and $M$ is a $\mu_r$-equivariant maximal Cohen--Macaulay sheaf on $\mathcal{X}$.
\end{lemma}
\begin{proof}
Take a bounded above free resolution $P^\bullet \xrightarrow{\sim} F$, which exists since we work locally. For $k \gg 0$, the naive truncation $\sigma^{\le -k} P^\bullet$ is quasi-isomorphic to a shift of a sheaf $M[k]$, and $M$ is equivariant maximal Cohen--Macaulay by the depth lemma.
\end{proof}

The following construction is an equivariant generalization of \cite[Section~2]{Hao2025b}. Since the syzygy divisor algorithm preserves monomiality under the group action, we obtain completely analogous conclusions. Let $M$ be an equivariant Cohen--Macaulay module. Then $M$ is torsion-free, hence admits an equivariant Bourbaki sequence over $S$ by Lemma \ref{Bexactsequence}:
\[
0 \longrightarrow \bigoplus_{i=1}^{\operatorname{rank}(M)-1} S \otimes \nu^{m_i}
\longrightarrow M \longrightarrow I \otimes \nu^{m_0} \longrightarrow 0.
\]
Fix a standard basis of $R$ refining the weighted valuation defined above. For any ideal $I \subset S$, consider its inverse ideal $\bar{I}$ in $R$. Since the $\mu_r$-action preserves monomials and the chosen monomial order, an equivariant syzygy resolution of $\bar{I}$ over $R$ can be constructed:
\[
\begin{tikzcd}
0 \arrow[r]
& \bigoplus_{l=1}^{m_k} R \otimes \nu^{-n^{(k)}_l} \arrow[r, "F_k"]
& \cdots \arrow[r, "F_1"]
& \bigoplus_{l=1}^{m_0} R \otimes \nu^{-n^{(0)}_l} \arrow[r, "F_0"]
& R \arrow[r]
& R/\bar{I} \arrow[r]
& 0.
\end{tikzcd}
\]
Each map $F_i$ is expressed with respect to a standard basis refining the weighted valuation with coordinate contributions defined naturally. In particular, for any integer $r$, this exact sequence induces a family of filtered exact sequences:
\[
\begin{tikzcd}
0 \arrow[r]
& \bigoplus^{m_{k}}_{l=1} P_{r-\mathrm{a}^{(k)}_l} \otimes \nu^{-n^{(k)}_{l}} \arrow[r, "F_k"]
& \cdots \arrow[r, "F_1"]
& \bigoplus^{m_0}_{l=1} P_{r-\mathrm{a}_l} \otimes \nu^{-n^{(0)}_{l}} \arrow[r, "F_0"]
& P_{r} \arrow[r]
& P_{r}/I \arrow[r]
& 0,
\end{tikzcd}
\]
where $P_i := R$ for $i \le 0$.\\

For a complete intersection $S \subset R$, iterating the above division procedure yields a resolution of $S/I$ over $S$ of the form
\[
\begin{tikzcd}
\cdots \arrow[r, "F'_2"]
& \bigoplus_{l=1}^{m'_1} S \otimes \nu^{-n'^{(1)}_l} \arrow[r, "F'_1"]
& \bigoplus_{l=1}^{m'_0} S \otimes \nu^{-n'^{(0)}_l} \arrow[r, "F'_0"]
& S \arrow[r]
& S/I \arrow[r]
& 0.
\end{tikzcd}
\]
This resolution also induces a family of filtered exact sequences on $S$:
\[
\begin{tikzcd}
\cdots \arrow[r, "F'_2"]
& \bigoplus^{m'_1}_{l=1} P_{r-\mathrm{a}'^{(1)}_l} \otimes \nu^{-n'^{(1)}_{l}} \arrow[r, "F'_1"]
& \bigoplus^{m'_0}_{l=1} P_{r-\mathrm{a}'_l} \otimes \nu^{-n'^{(0)}_{l}} \arrow[r, "F'_0"]
& P_{r} \arrow[r]
& P_{r}/I \arrow[r]
& 0.
\end{tikzcd}
\]
For each integer $r$, these exact sequences carry the structure of higher matrix factorizations generated by finitely many matrices; their construction and properties follow from the previously developed framework. Finally, since the equivariant structure is preserved under all division and syzygy constructions, the resulting resolutions are $\mu_r$-equivariant.\\

Let us consider the stacky Serre theorem for $w\mathrm{Bl}^{(a_i)}_C \mathcal{X}$. A standard stacky form is
\[
\mathrm{Coh}\!\left(\left[\mathcal{P}roj_R\!\left(\bigoplus_{i \ge 0} P_i\right)\big/\mu_r\right]\right)
\simeq
\mathrm{Coh}^{\mathbb{G}_m}\!\left(\left[\Spec_R \bigoplus_{i \ge 0} P_i \big/ \mu_r\right]\right)
\Big/
\mathrm{Coh}^{\mathbb{G}_m}_Z\!\left(\left[\Spec_R \bigoplus_{i \ge 0} P_i \big/ \mu_r\right]\right).
\]
Since $\mathrm{Coh}_Z$ is a Serre subcategory of $\mathrm{Coh}$, we denote this functor by $\Gamma$, and its inverse by $\Pi$. As a simple example, for any integer $k$:
\[
\Pi\!\left(\bigoplus_{i\geq 0} P_{i+k}\right) = \mathcal{O}_{\mathcal{Z}}(-k\mathrm{M}) := \mathcal{M}^{\otimes -k},
\]
where $\mathrm{M}$ is the stacky exceptional divisor in $w\mathrm{Bl}^{(a_i)}_C \mathcal{X}$. Tensoring with any representation $\nu^l$ yields
\[
\Pi\!\left(\bigoplus_{i\geq 0} P_{i+k} \otimes \nu^l\right)
= \mathcal{O}_{\mathcal{Z}}(-k\mathrm{M})\otimes \nu^l
:= \mathcal{M}^{\otimes -k} \otimes \nu^l.
\]
Finally, note that $\mathcal{M}\big|_\mathrm{M} = \mathcal{O}_{E_\mathcal{Y}}(-1) \otimes \nu$ in the equivariant representation. Taking the direct sum over all such filtered exact sequences, one obtains, via $\Pi$, an exact sequence on $w\mathrm{Bl}^{(a_i)}_C \mathcal{X}$ of the form
\[
\begin{tikzcd}
\cdots \arrow[r]
& \bigoplus^{m'_2}_{l=1}\mathcal{O}_\mathcal{Z}\!\left(\mathrm{a}^{(2)}_l \mathrm{M}\right) \otimes \nu^{-n^{(2)}_{l}} \arrow[r]
& \bigoplus^{m'_1}_{l=1} \mathcal{O}_\mathcal{Z}\!\left(\mathrm{a}^{(1)}_l \mathrm{M}\right) \otimes \nu^{-n^{(1)}_{l}} \arrow[r]
& \mathcal{O}_\mathcal{Z} \arrow[r]
& \mathcal{O}_\mathcal{Z}/\mathcal{I} \arrow[r]
& 0
\end{tikzcd}
\]
Comparing this exact sequence with the above pullback complex
\[
\begin{tikzcd}
\cdots \arrow[r]
& \bigoplus^{m'_2}_{l=1} \mathcal{O}_{\mathcal{Z}} \otimes \nu^{-n^{(2)}_{l}} \arrow[r]
& \bigoplus^{m'_1}_{l=1} \mathcal{O}_{\mathcal{Z}} \otimes \nu^{-n^{(1)}_{l}} \arrow[r]
& \mathcal{O}_{\mathcal{Z}} \arrow[r]
& \pi^*(\mathcal{O}_{\mathcal{X}}/I) \arrow[r]
& 0
\end{tikzcd}
\]yields a collection of Postnikov systems in $\mathrm{D}^b(\mathcal{Z})$.\\

In particular, if $\mathcal{O}_{\mathcal{X}}/(I \simeq I_C)$, it induces a natural morphism $\pi^*\mathcal{O}_{C} \to F^s([\mathcal{O}_{C}]^+)$ in $\mathrm{D}^-(\mathcal{Z})$ which extends to the filtration:
\begin{equation}
\begin{tikzcd}[column sep=0.5em]
\pi^*\mathcal{O}_{C}\arrow{rr} && F^s([\mathcal{O}_{C}]^+)\arrow{dl}\arrow{rr}&& F^{s-1}\arrow{dl}&&  \\
&\tau_s\arrow[ul,dashed,"\Delta"]&& \iota^+_{*}\mathcal{D}_{s}\otimes\nu^{-s}[1]\arrow[ul,dashed,"\Delta"]&&
\end{tikzcd}
\quad\cdots\quad
\begin{tikzcd}[column sep=0.5em]
 & F^1\arrow{rr}&& F^0\arrow{rr}\arrow{dl}&& 0\arrow{dl} \\
 && \iota^+_{*}\mathcal{D}_{1}\otimes\nu^{-1}[1]\arrow[ul,dashed,"\Delta"] && \iota^+_{*}\mathcal{D}_{0}[1]\arrow[ul,dashed,"\Delta"]
\end{tikzcd}
\label{pullbackcompu}
\end{equation}
If $e := \sum_{i=1}^s a_i - \deg_{val} I > 0$ is the Fano index, then for any $i \leq e-1$, the object $\mathcal{D}_i$ coincides with the $i$-th dual exceptional (or pre-tilting) object appearing in the semi-orthogonal decomposition of $\mathrm{D}^b(E_\mathcal{Y})$ as a hypersurface fibration from the construction in \cite{Orlov2009}:
\[
\mathrm{D}^b(E_\mathcal{Y}) =
\left\langle
\mathcal{A}_{-1},\, \mathcal{O}_{E_\mathcal{Y}},\,
\mathcal{O}_{E_\mathcal{Y}}(1),\,
\dots,\,
\mathcal{O}_{E_\mathcal{Y}}(e-1)
\right\rangle.
\]
For $i \geq e$, we have $\mathcal{D}_i \in \mathcal{A}_{-1}$. Note also that $\mathrm{D}^b(\mathrm{M}) = \mathrm{D}^b([E_\mathcal{Y}/\mu_r])$, and since the $\mu_r$-action is free, the equivariant structure produces additional contributions of the form $\nu^l$. Then, pushing forward the above filtration along $\varpi_*$ cancels all non-trivial representations:
\[
\begin{tikzcd}[column sep=0.5em]
 \varpi_*\pi^*\mathcal{O}_{C}\arrow{rr} && F_{(0)}^s([\mathcal{O}_{C}]^+)\arrow{dl}\arrow{rr}&& F_{(0)}^{s-1}\arrow{dl}&&  \\
&\tau^{(0)}_s\arrow[ul,dashed,"\Delta"]&& \iota^+_{*}\mathcal{D}_{sr}[1]\arrow[ul,dashed,"\Delta"]&&
\end{tikzcd}
\quad\cdots\quad
\begin{tikzcd}[column sep=0.5em]
 & F_{(0)}^1\arrow{rr}&& F_{(0)}^0\arrow{rr}\arrow{dl}&& 0\arrow{dl} \\
 && \iota^+_{*}\mathcal{D}_{r}[1]\arrow[ul,dashed,"\Delta"] && \iota^+_{*}\mathcal{D}_{0}[1]\arrow[ul,dashed,"\Delta"]
\end{tikzcd}
\]
Similarly, for the tensor-twisted version, we have:
\begin{equation}
\begin{tikzcd}[column sep=0.5em]
 \varpi_*\pi^*(\mathcal{O}_{C}\otimes\nu^i) \arrow{rr} && F_{(i)}^s([\mathcal{O}_{C}]^+) \arrow{dl} \arrow{rr} && F_{(i)}^{s-1} \arrow{dl} &&  \\
 & \tau^{(i)}_s \arrow[ul, dashed, "\Delta"] && \iota_{*}\mathcal{D}_{sr+i}[1] \arrow[ul, dashed, "\Delta"] &&
\end{tikzcd}
\quad\cdots\quad
\begin{tikzcd}[column sep=0.5em]
 & F_{(i)}^1 \arrow{rr} && F_{(i)}^0 \arrow{rr} \arrow{dl} && 0 \arrow{dl} \\
 && \iota_{*}\mathcal{D}_{r+i}[1] \arrow[ul, dashed, "\Delta"] && \iota_{*}\mathcal{D}_{i}[1] \arrow[ul, dashed, "\Delta"]
\end{tikzcd}
\label{pullbackcompu1}
\end{equation}
for any $0 \leq i < r$. Furthermore, for $s \gg 0$, the cohomology $\mathcal{H}^j(\tau^{(i)}_s)$ vanishes for all $j > -s + N$, where $N$ is a fixed integer depending on $e$ and $\sum_{i=1}^s a_i$. Consequently, $F^{(i)}([\mathcal{O}_{C}]^+)$ provides an approximation to $\varpi_*\pi^*\mathcal{O}_{C}\otimes\nu^i$, and in fact it's:
\[
\varpi_*\pi^*(\mathcal{O}_{C}\otimes\nu^i)\simeq \varprojlim_{s \to \infty} F_{(i)}^s([\mathcal{O}_{C}]^+).
\]
Completely analogously, for any maximal Cohen--Macaulay module $M$ on $\mathcal{X}$, we can construct an object $\mathcal{R}(M)\in \mathrm{D}^b(\mathcal{Z})$ together with a filtration obtained by iteratively applying the above procedure:
\[
\begin{tikzcd}[column sep=0.5em]
 & \mathcal{F}^s(\mathcal{R}(M)) \arrow{rr}&& \mathcal{F}^{s-1}\arrow{dl}&& \\
&&  \iota^+_{*}\operatorname{gr}^{s}(M)\arrow[ul,dashed,"\Delta"]&&
\end{tikzcd}
\quad\cdots\quad
\begin{tikzcd}[column sep=0.5em]
 &   \mathcal{F}^2\arrow{rr}&&  \mathcal{F}^1\arrow{rr}\arrow{dl}&& \mathcal{R}(M)\arrow{dl} \\
 && \iota^+_{*}\operatorname{gr}^2(M)\arrow[ul,dashed,"\Delta"]  &&\iota^+_{*}\operatorname{gr}^1(M)\arrow[ul,dashed,"\Delta"]
\end{tikzcd}
\]
for which,  we have $\operatorname{gr}^i(M)\in \mathcal{A}_{-1}\otimes \nu^{\ell(i)}$ for every $i$. Moreover, there is a natural morphism $\pi^*M \to \mathcal{F}^s(\mathcal{R}(M))$, which extends canonically to the diagrams
\[
\begin{tikzcd}[column sep=0.5em]
 & \pi^*M \arrow{rr} && \mathcal{F}^{s} \arrow{dl} \arrow{rr} && \mathcal{F}^{s-1} \arrow{dl} \\
 && \tau^{s} \arrow[ul, dashed, "\Delta"] && \iota^+_{*}\operatorname{gr}^{s}(M) \arrow[ul, dashed, "\Delta"]
\end{tikzcd}
\quad\cdots\quad
\begin{tikzcd}[column sep=0.5em]
 & \mathcal{F}^2 \arrow{rr} && \mathcal{F}^1 \arrow{rr} \arrow{dl} && \mathcal{R}(M) \arrow{dl} \\
 && \iota^+_{*}\operatorname{gr}^2(M) \arrow[ul, dashed, "\Delta"] && \iota^+_{*}\operatorname{gr}^1(M) \arrow[ul, dashed, "\Delta"]
\end{tikzcd}
\]
we have  \[
\pi^*M = \varprojlim_{s \to \infty} \mathcal{F}^s(\mathcal{R}(M)).
\]
\begin{lemma}[Corollary 3.4.7, \cite{Hao2025b}]
For any $\mu_r$-equivariant Cohen--Macaulay sheaf $M$ on $\mathcal{X}$, the object $\mathcal{R}(M)\in \mathrm{D}^b(\mathcal{Z})$ constructed above satisfies $\pi_* \mathcal{R}(M) \simeq M$.
\end{lemma}

\begin{lemma}[Proposition 3.2, \cite{Hao2025b}]\label{kernallemma}
For any $F \in \mathrm{D}^b(\mathcal{Z})$, if $\pi_* F = 0$, then $F$ lies in the full saturated subcategory
\[\mathcal{T}:=
\left\langle
\iota^+_*\mathcal{O}_{\mathrm{M}}(-e+1)\otimes \bigoplus_{i=0}^{r-1} \nu^i,\, \ldots,\, \iota^+_*\mathcal{O}_{\mathrm{M}}(-1)\otimes \bigoplus_{i=0}^{r-1} \nu^i,\, \left\langle \iota^+_*\mathcal{A}_{-1}\otimes \bigoplus_{i=0}^{r-1} \nu^i \right\rangle
\right\rangle.
\]
Here the first $e-1$ terms are pre-tilting objects. We denote by $\mathcal{K}_{-1}^+$ the last term, which is the full saturated subcategory generated by $\mathcal{A}_{-1}$ tensored with all representations $\nu^i$.
\end{lemma}

Thus, for any $F \in \mathrm{D}^b(\mathcal{X})$, take a sheaf approximation $M[k] \to F \to P \to M[k+1]$ with $M$ maximal Cohen--Macaulay and $P$ a perfect complex, we define $\mathcal{R}(M) \to \mathcal{R}\pi^*F \to \pi^*P$, and hence obtain a (quasi-)functor $\mathcal{R}\pi^*: \mathrm{D}^b(\mathcal{X}) \to \mathrm{D}^b(\mathcal{Z})$, defined up to some ambiguity.
We obtain the following result:
\begin{proposition}[Proposition 3.5, \cite{Hao2025b}]
If $e>0$, there is a natural well-defined fully faithful right admissible embedding of triangulated categories
\[
\mathcal{R}\pi^* : \mathrm{D}^b(\mathcal{X}) \hookrightarrow \mathrm{D}^b(\mathcal{Z})/\mathcal{K}_{-1}^+,
\]
and the following semi-orthogonal decomposition:
\[
\mathrm{D}^b(\mathcal{Z})/\mathcal{K}_{-1}^+
=
\left\langle
\iota^+_*\mathcal{O}_{\mathrm{M}}(-e+1)\otimes \bigoplus_{i=0}^{r-1} \nu^i,\,
\ldots,\,
\iota^+_*\mathcal{O}_{\mathrm{M}}(-2)\otimes \bigoplus_{i=0}^{r-1} \nu^i,\,
\iota^+_*\mathcal{O}_{\mathrm{M}}(-1)\otimes \bigoplus_{i=0}^{r-1} \nu^i,\,
\operatorname{Im}(\mathcal{R}\pi^*)
\right\rangle.
\]
\end{proposition}\label{admlemma}
\begin{proposition}[Proposition 3.4, \cite{Hao2025b}]
If $e>0$, there exists a natural well-defined fully faithful left admissible embedding of triangulated categories
\[
\mathcal{\widetilde{R}}\pi^*
:= \mathcal{R}\pi^! \otimes \mathcal{O}(-\mathcal{K}_{\mathcal{Z}/\mathcal{X}})
:\mathrm{D}^b(\mathcal{X}) \hookrightarrow \mathrm{D}^b(\mathcal{Z})/\mathcal{K}_{-1}^+.
\]
In particular, it is isomorphic to $\mathcal{R}\pi^*$, and we have an adjoint triple
\[
\pi_! := \pi_* \otimes \mathcal{O}(\mathcal{K}_{\mathcal{Z}/\mathcal{X}}) \dashv \mathcal{R}\pi^* \dashv \pi_*.
\]
\end{proposition}

\begin{proof}
The category $\mathrm{D}^b(\mathcal{Z})/\mathcal{K}_{-1}^+$ contains a semi-orthogonal decomposition
\[
\left\langle
\operatorname{Im}(\mathcal{\widetilde{R}}\pi^*),
\iota^+_*\mathcal{O}_{\mathrm{M}} \otimes \bigoplus_{i=0}^{r-1}\nu^i,\,
\iota^+_*\mathcal{O}_{\mathrm{M}}(1)\otimes \bigoplus_{i=0}^{r-1}\nu^i,\,
\ldots,\,
\iota^+_*\mathcal{O}_{\mathrm{M}}(e-2)\otimes \bigoplus_{i=0}^{r-1}\nu^i
\right\rangle,
\]
and these objects generate the whole category. On the other hand, we have the same semi-orthogonal decomposition
\[
\left\langle
\operatorname{Im}(\mathcal{R}\pi^*),
\iota^+_*\mathcal{O}_{\mathrm{M}} \otimes \bigoplus_{i=0}^{r-1}\nu^i,\,
\iota^+_*\mathcal{O}_{\mathrm{M}}(1)\otimes \bigoplus_{i=0}^{r-1}\nu^i,\,
\ldots,\,
\iota^+_*\mathcal{O}_{\mathrm{M}}(e-2)\otimes \bigoplus_{i=0}^{r-1}\nu^i
\right\rangle.
\]
Finally, since $\pi_* \mathcal{\widetilde{R}}\pi^* = \pi_* \mathcal{R}\pi^* = \mathrm{id}$ by Corollary 3.4.4-8 \cite{Hao2025b}, it follows that $\mathcal{\widetilde{R}}\pi^* \cong \mathcal{R}\pi^*$.
\end{proof}

Now we define the full saturated triangulated subcategory of $\mathrm{D}^b(\mathcal{Y})$ generated by the above-defined category on $E_\mathcal{Y}$ as
\[
\mathcal{K}_{-1} \;:=\; \left\langle \iota_*\mathcal{A}_{-1} \right\rangle
\]
We consider the natural projection functor \(\varpi_* : \mathrm{D}^b(\mathcal{Z}) \longrightarrow \mathrm{D}^b(\mathcal{Y})/\mathcal{K}_{-1}\).
It is straightforward to verify that every morphism passing to the subcategory $\mathcal{K}_{-1}^+$ is isomorphism under $\varpi_*$ in $\mathrm{D}^b(\mathcal{Y})/\mathcal{K}_{-1}$. Hence, by the universal property of the Verdier quotient, the functor $\varpi_*$ factors through a canonical functor
\[
\varpi_* : \mathrm{D}^b(\mathcal{Z})/\mathcal{K}_{-1}^+ \longrightarrow \mathrm{D}^b(\mathcal{Y})/\mathcal{K}_{-1}.
\]
But the functor $\varpi^!:\mathrm{D}^b(\mathcal{Y})/\mathcal{K}_{-1}\to \mathrm{D}^b(\mathcal{Z})/\mathcal{K}_{-1}^+$ is not well-defined in general. Indeed, for $N\in \mathcal{A}_{-1}$, the object $\varpi^!\iota_*N$ admits a filtration with successive factors $\iota^+_*N(-l)\otimes \nu^l$ for $l=0,\dots,r-1$ by Lemma \ref{derivedroot}, and therefore $\varpi^!\iota_*N\notin \mathcal{K}_{-1}^+$. Thus, one should consider $\varpi^!:\mathrm{D}^b(\mathcal{Y})/\mathcal{K}_{-1}\to \mathrm{D}^b(\mathcal{Z})/\varpi^!\mathcal{K}_{-1}$. However, we have the following:
\begin{lemma}
Assume $e \ge r$. Then the induced functor
\[
\pi_* \varpi^!:
\mathrm{D}^b(\mathcal{Y})/\mathcal{K}_{-1}
\longrightarrow
\mathrm{D}^b(\mathcal{X})
\]
is a  well-defined triangulated functor.
\end{lemma}
\begin{proof}
Note that $\iota^+_* N(-l)\otimes \nu^l, \ l=0,\dots,r-1,$ belong to $\iota^+_* \mathcal{A}_{-l}\otimes \nu^{l-1}$ for $l=1,\dots,r$. Consider the semi-orthogonal decomposition
\[
\mathrm{D}^b(E_\mathcal{Y})=\left\langle \mathcal{A}_{-e},\,\mathcal{O}_{E_\mathcal{Y}}(-e+1),\,\dots,\,\mathcal{O}_{E_\mathcal{Y}}(-1),\,\mathcal{O}_{E_\mathcal{Y}} \right\rangle.
\]
Since $r\le e$, we have $\pi_* \varpi^! \iota_* N=0$. The distinguished triangulates also preserve.
\end{proof}

\begin{lemma}
Assume $e \ge r$. Then
\[
\varpi_* \mathcal{R}\pi^* : \mathrm{D}^b(\mathcal{X}) \longrightarrow \mathrm{D}^b(\mathcal{Y})/\mathcal{K}_{-1},
\quad
\pi_* \varpi^! : \mathrm{D}^b(\mathcal{Y})/\mathcal{K}_{-1} \longrightarrow \mathrm{D}^b(\mathcal{X})
\]
form an adjoint pair
\[
\varpi_* \mathcal{R}\pi^* \;\dashv\; \pi_* \varpi^!.
\]
\end{lemma}

\begin{proof}
It suffices to prove that for any $F \in \mathrm{D}^b(\mathcal{X})$ and $G \in \mathrm{D}^b(\mathcal{Y})/\mathcal{K}_{-1}$, we have
\[
\operatorname{Ext}^*_{\mathrm{D}^b(\mathcal{Y})/\mathcal{K}_{-1}}
(\varpi_* \mathcal{R}\pi^* F,\, G)
\simeq
\operatorname{Ext}^*_{\mathrm{D}^b(\mathcal{X})}
(F,\, \pi_* \varpi^! G).
\]
Consider natural map
\[
\left[
F \xrightarrow{\, g \,} \pi_* \varpi^! G
\right]
\;\longmapsto\;
\left[
\begin{tikzcd}[column sep=small, row sep=small]
& \mathcal{R}^{s\gg 0}\pi^* F \arrow[dl, "s"] \arrow[dr, "\pi^*g"] & \\
\mathcal{R}\pi^* F \arrow[dr, dashed, "f"] & & \varpi^! G \arrow[dl, dashed, "t"] \\
& G^- &
\end{tikzcd}
\right]
\;\longmapsto\;\left[
\begin{tikzcd}[column sep=small, row sep=small]
& \varpi_* \mathcal{R}^{s\gg 0}\pi^* F\arrow[dl, "\varpi_* s"'] \arrow[dr, "\varpi_* \pi^*g"] & \\
\varpi_* \mathcal{R}\pi^* F & & G
\end{tikzcd}
\right].
\]
where $\mathrm{Cone}(s) \in \mathcal{K}_{-1}^+$ and $\varpi_* \mathrm{Cone}(s) \in \mathcal{K}_{-1}$. It follows that the above correspondence is both surjective and injective, here one may consider inside $\mathrm{D}^b(\mathcal{Z})$ a quotient by a full saturated subcategory generated by  $\varpi^!\mathcal{K}_{-1}$ and $\mathcal{K}_{-1}^+$, and then take the corresponding pushout. Moreover its natural inverse map is given by
\[
\left[
\begin{tikzcd}[column sep=small, row sep=small]
\varpi_* \mathcal{R}\pi^* F \arrow[dr, "f"] & & G \arrow[dl, "t"'] \\
& G^- &
\end{tikzcd}
\right]
\quad \longmapsto \quad
\left[
F \xrightarrow{\ \pi_* \varpi^! f\ } \pi_* \varpi^! G^- = \pi_* \varpi^! G
\right],
\]
In the last step, since $\tau := \mathrm{Cone}(t) \in \mathcal{K}_{-1}$, we have $\pi_* \varpi^! \tau = 0$ which is a well-defined natural map.
\end{proof}
Similarly, for $N \in \mathcal{A}_{-1}$, we have that $\varpi^* \iota_* N \otimes \mathcal{O}(\mathcal{K}_{\mathcal{Z}/\mathcal{X}})$ admits a filtration whose successive factors belong to $\iota^+_* \mathcal{A}_l$ tensor certain representations with $l$ in the range $-e \le l \le -(e-r+1)$. If we assume $e \ge r$, then $\pi_! \varpi^* \iota_* N = 0$, and hence we obtain a naturally defined triangulated functor
\[
\pi_! \varpi^* :
\mathrm{D}^b(\mathcal{Y})/\mathcal{K}_{-1}
\longrightarrow
\mathrm{D}^b(\mathcal{X}),
\]
together with the following properties.
\begin{lemma}
Assume $e \ge r$. Then $\varpi_* \mathcal{R}\pi^*$ is admissible and fits into an adjoint triple
\[
\pi_! \varpi^* \;\dashv\; \varpi_* \mathcal{R}\pi^* \;\dashv\; \pi_* \varpi^!.
\]
\end{lemma}
\begin{proof}
We have a similar natural bijection
\[
\left[
\pi_! \varpi^* F \xrightarrow{\, g \,} G
\right]
\;\longmapsto\;
\left[
\begin{tikzcd}[column sep=small, row sep=small]
& G^+ \arrow[dl, dashed, "t"] \arrow[dr, dashed, "f"] & \\
\ \varpi^*F \arrow[dr, "\pi^* g"] & & \widetilde{\mathcal{R}}\pi^* G \arrow[dl, "s"] \\
& \widetilde{\mathcal{R}}^{s\gg 0}\pi^* G &
\end{tikzcd}
\right]
\;\longmapsto\;\left[
\begin{tikzcd}[column sep=small, row sep=small]
F \arrow[dr, "\varpi_*\pi^* g"'] & & \varpi_* \widetilde{\mathcal{R}}\pi^* G \arrow[dl, "\varpi_* s"] \\
& \varpi_*\widetilde{\mathcal{R}}^{s\gg 0}\pi^* G &
\end{tikzcd}
\right].
\]
and we also know have $\widetilde{\mathcal{R}}\pi^* G = \mathcal{R}\pi^* G$ from Proposition \ref{admlemma}.
\end{proof}

\begin{lemma}\label{orddualcomp}
Consider the following (dual) semi-orthogonal decompositions of $\mathrm{D}^b(E_\mathcal{Y})$:
\[
\mathrm{D}^b(E_\mathcal{Y})
\simeq
\left\langle
\mathcal{A}_{-1},\,\mathcal{O}_{E_\mathcal{Y}}\,\dots,\,\mathcal{O}_{E_\mathcal{Y}}(e-1)
\right\rangle
\simeq
\left\langle
\mathcal{A}_{-1},\,\mathcal{D}_{e-1},\,\dots,\,\mathcal{D}_0
\right\rangle.
\]
Then for $0 \le i,j < e$, we have
\[
\operatorname{Ext}^*_{\mathrm{D}^b(E_\mathcal{Y})}
\bigl(\mathcal{O}_{E_\mathcal{Y}}(i), \mathcal{D}_j\bigr)
=
\begin{cases}
\mathcal{O}_C, & i = j,\\
0, & i \ne j.
\end{cases}
\]
\end{lemma}
\begin{proof}
Since $\mathcal{D} := \mathrm{D}^b(E_\mathcal{Y}) / \mathcal{A}_{-1}$ is a full triangulated subcategory of $\mathrm{D}^b(E_\mathcal{Y})$, we have
\[
\operatorname{Ext}^*_{\mathrm{D}^b(E_\mathcal{Y})}
\bigl(\mathcal{O}_{E_\mathcal{Y}}(i), \mathcal{D}_j\bigr)
=
\operatorname{Ext}^*_{\mathcal{D}}
\bigl(\mathcal{O}_{E_\mathcal{Y}}(i), \mathcal{D}_j\bigr).
\]
Moreover, $\mathcal{O}_{E_\mathcal{Y}}(i)$ and $\mathcal{D}_j$ form a full (dual) exceptional (pre-tilting) collection in $\mathcal{D}$, result follows from the classical argument e.g. \cite[Proposition 3.1]{Hao2025a}.
\end{proof}
\begin{lemma}
Assume $e \ge r$, we have
\[
\pi_* \varpi^!
\Bigl(
\varpi_* \mathcal{R}\pi^* (\kappa_*\mathcal{O}_C \otimes \nu^i)
\Bigr)
\simeq
\kappa_*\mathcal{O}_C \otimes \nu^i,
\quad \text{for } 0 \le i < r-1.
\]
\end{lemma}

\begin{proof}
We know from (\ref{pullbackcompu1}) that in $\mathrm{D}^b(\mathcal{Y})/\mathcal{K}_{-1}$, we have
$\varpi_*\mathcal{R}\pi^*(\mathcal{O}_C \otimes \nu^i)
\simeq F_{(i)}^s([\mathcal{O}_{C}]^+)$
for any $sr+i \ge e$, and we have a filtration of $F_{(i)}^s([\mathcal{O}_{C}]^+)$ given by
\[
\begin{tikzcd}[column sep=0.5em]
  F_{(i)}^s([\mathcal{O}_{C}]^+)\arrow{rr}&& F_{(i)}^{s-1}\arrow{dl}&&  \\
& \iota_{*}\mathcal{D}_{sr+i}[1]\arrow[ul,dashed,"\Delta"]&&
\end{tikzcd}
\quad\cdots\quad
\begin{tikzcd}[column sep=0.5em]
 & F_{(i)}^2\arrow{rr}&& F_{(i)}^1\arrow{rr}\arrow{dl}&& 0\arrow{dl} \\
 && \iota_{*}\mathcal{D}_{r+i}[1]\arrow[ul,dashed,"\Delta"] && \iota_{*}\mathcal{D}_{i}[1]\arrow[ul,dashed,"\Delta"]
\end{tikzcd}
\]
For $\mathcal{D}_{\lambda r+i}$ with $\lambda r+i\ge e$, we have $\mathcal{D}_{\lambda r+i}\in \mathcal{A}_{-1}$ and $\pi_*\varpi^!\iota_*\mathcal{D}_{\lambda r+i}=0$, so these terms do not contribute any summands. Hence we only need to consider the factors $\mathcal{D}_{\lambda r+i}$ with $\lambda r+i<e$. We compute $\varpi^*\iota_*\mathcal{D}_{\lambda r+i}$ using Lemma \ref{derivedroot}, which admits a second filtration whose successive factors are $\iota^+_*\mathcal{D}_{\lambda r+i}(j)\otimes \nu^{-j}$ for $0\le j\le r-1$. Together with an additional tensor factor $\mathcal{M}^{r-1}$, this yields factors $\iota^+_*\mathcal{D}_{\lambda r+i}(-r+1+j)\otimes \nu^{r-1-j}$ with factors for the pairs $(\lambda,j)$ with $\lambda r+i<e$ and $0\le j\le r-1$. Then pushing forward along $\pi_*$, we compute via Lemma \ref{orddualcomp}:
\[
\pi_* \iota^+_* \mathcal{D}_{\lambda r+i}(-r+1+j)\otimes \nu^{r-1-j}
\simeq \operatorname{Ext}^*_{E_\mathcal{Y}}\big(\mathcal{O}_{E_\mathcal{Y}}(r-1-j), \mathcal{D}_{\lambda r+i}\big)\otimes \nu^{r-1-j}
\simeq \mathcal{O}_C \cdot \delta_{\lambda r+i,\, r-1-j}\otimes \nu^{r-1-j}.
\]
Hence, by the non-vanishing condition and the range of parameters, the only contribution occurs when $\lambda=0$ and $i+j=r-1$. Therefore $j=r-1-i$, and the resulting contribution is only $\mathcal{O}_C\otimes \nu^i$.
\end{proof}
We obtain the following corollaries:
\begin{corollary}\label{gencor}
Assume $e \ge r$. Then the functor $\varpi_* \mathcal{R}\pi^*$ is fully faithful.
\end{corollary}
\begin{proof}
We know that $\varpi_* \mathcal{R}\pi^*$ is admissible. Since we assume that $\mathcal{O}_C$ is regular, the residue field $\mathrm{k}:=S/\mathfrak{m}$ admits a Koszul resolution generated by direct sums of $\mathcal{O}_C$, and hence $\mathcal{O}_C\otimes \nu^l$ for $0\le l<r-1$ forms a spanning class $\Omega$ for $\mathrm{D}^b(\mathcal{X})$. The full faithfulness of $\varpi_* \mathcal{R}\pi^*$ is equivalent to
\[
\operatorname{Ext}^*_{\mathrm{D}^b(\mathcal{Y})/\mathcal{K}_{-1}}\big(\varpi_* \mathcal{R}\pi^* a,\varpi_* \mathcal{R}\pi^* b\big)
=
\operatorname{Ext}^*_{\mathrm{D}^b(\mathcal{X})}(a,b)
\quad \text{for all } a,b\in \Omega,
\]
by \cite[Proposition 1.49]{Huybrechts}. Moreover, we have
\[
\operatorname{Ext}^*_{\mathrm{D}^b(\mathcal{Y})/\mathcal{K}_{-1}}\big(\varpi_* \mathcal{R}\pi^* a,\varpi_* \mathcal{R}\pi^* b\big)
\simeq
\operatorname{Ext}^*_{\mathrm{D}^b(\mathcal{X})}\big(a,\pi_* \varpi^! \varpi_* \mathcal{R}\pi^* b\big)
=
\operatorname{Ext}^*_{\mathrm{D}^b(\mathcal{X})}(a,b),
\]
so the conclusion follows.
\end{proof}
\begin{corollary}\label{gencor1}
In this situation, there are natural isomorphisms
\[
\pi_* \varpi^! \varpi_* \mathcal{R}\pi^* \simeq
\pi_! \varpi^* \varpi_* \mathcal{R}\pi^* \simeq \mathrm{id}.
\]
\end{corollary}

\begin{lemma}\label{mutationlemma}
For any $1 \leq k \leq e-1$, we have :
\[
\langle \mathcal{D}_{e-1}, \dots, \mathcal{D}_{e-k} \rangle
\simeq
\langle \mathcal{O}_{E_\mathcal{Y}}(-k), \dots, \mathcal{O}_{E_\mathcal{Y}}(-1) \rangle.
\]
in $\mathrm{D}^b(E_\mathcal{Y})/\mathcal{A}_{-1}$.
\end{lemma}
\begin{proof}
By the following equivalence
\[
\begin{aligned}
\langle \mathcal{D}_{e-1}, \dots, \mathcal{D}_{e-k}, \mathcal{D}_{e-k+1}, \dots, \mathcal{D}_0 \rangle
&\simeq
\langle \mathcal{D}_{e-1}, \dots, \mathcal{D}_{e-k}, \mathcal{O}_{E_{\mathcal{Y}}}, \dots, \mathcal{O}_{E_{\mathcal{Y}}}(e-k+1) \rangle \\
&\simeq
\langle \mathcal{O}_{E_{\mathcal{Y}}}(-k), \dots, \mathcal{O}_{E_{\mathcal{Y}}}(-1), \mathcal{O}_{E_{\mathcal{Y}}}, \dots, \mathcal{O}_{E_{\mathcal{Y}}}(e-k+1) \rangle \\
&\simeq
\langle \mathcal{O}_{E_{\mathcal{Y}}}(-k), \dots, \mathcal{O}_{E_{\mathcal{Y}}}(-1), \mathcal{D}_{e-k+1}, \dots, \mathcal{D}_0 \rangle,
\end{aligned}
\]
we obtain the conclusion.
\end{proof}
\begin{lemma}\label{genlemma}
In $\mathrm{D}^b(E_\mathcal{Y})/\mathcal{A}_{-1}$, we have
\[
\operatorname{Ext}^*\big(\iota_*\mathcal{O}_{E_\mathcal{Y}},\,\iota_*\mathcal{O}_{E_\mathcal{Y}}(-l)\big)
\simeq
\begin{cases}
\mathcal{O}_C, & l=0,\\
0, & 1 \le l \le e-r-1.
\end{cases}
\]
\end{lemma}
\begin{proof}
First, by the excess distinguished triangle and noting that $\mathcal{O}_{E_\mathcal{Y}}(-E_\mathcal{Y})\simeq \mathcal{O}_{E_\mathcal{Y}}(r)$, together with the semi-orthogonal decompositions of $\mathrm{D}^b(E_\mathcal{Y})$, for any $N \in \mathcal{A}_{-1}$ we have
\[
\operatorname{Ext}^*_{\mathrm{D}^b(\mathcal{Y})}\big(\iota_*N,\,\iota_*\mathcal{O}_{E_\mathcal{Y}}(-l)\big)=0
\quad \text{for all } 0\le l\le e-r-1.
\]
Hence in this range
\[
\operatorname{Ext}^*_{\mathrm{D}^b(\mathcal{Y})/\mathcal{K}_{-1}}\big(\iota_*\mathcal{O}_{E_\mathcal{Y}},\,\iota_*\mathcal{O}_{E_\mathcal{Y}}(-l)\big)
\simeq
\operatorname{Ext}^*_{\mathrm{D}^b(\mathcal{Y})}\big(\iota_*\mathcal{O}_{E_\mathcal{Y}},\,\iota_*\mathcal{O}_{E_\mathcal{Y}}(-l)\big).
\]
The latter is similarly computed using the excess distinguished triangle, which gives the desired result.
\end{proof}
\begin{lemma}\label{genlemma1}
For $1 \leq l \leq e-r$, we have
\[
\pi_*\varpi^!\,\iota_*\mathcal{O}_{E_\mathcal{Y}}(-l)=0.
\]
\end{lemma}
\begin{proof}
We know from Lemma  \ref{derivedroot} that $\varpi^!\,\iota_*\mathcal{O}_{E_\mathcal{Y}}(-l)$ admits a natural filtration with factors
\[
\iota^+_*\mathcal{O}_{E_\mathcal{Y}}(-l-i)\otimes \nu^i,\quad 0 \le i \le r-1.
\]
each vanishes under $\pi_*$.
\end{proof}
\begin{proposition}
The local stacky Kawamata blow-up admits a semi-orthogonal decomposition
\[
\mathrm{D}^b(\mathcal{Y})/\mathcal{K}_{-1}
=
\left\langle
\mathcal{O}_{E_\mathcal{Y}}(-(e-r)),\dots,\mathcal{O}_{E_\mathcal{Y}}(-1),\operatorname{Im}\Theta
\right\rangle,
\]
where $\Theta := \varpi_* \mathcal{R}\pi^*$ defines a well-defined fully faithful admissible triangulated functor
\[
\Theta : \mathrm{D}^b(\mathcal{X}) \to \mathrm{D}^b(\mathcal{Y})/\mathcal{K}_{-1},
\]
and the objects $\mathcal{O}_{E_\mathcal{Y}}(-(e-r)),\dots,\mathcal{O}_{E_\mathcal{Y}}(-1)$ are pretilting or exceptional, under assuming $e \ge r$.
\end{proposition}
\begin{proof}
We only need to prove the generation property using Corollary \ref{gencor}, Lemma \ref{genlemma} and \ref{genlemma1}. Let $b\in \mathrm{D}^b(\mathcal{Y})$, and consider the natural adjunction morphism, which induces a distinguished triangle
\[
\Theta\, \Theta^! b \longrightarrow b \longrightarrow c \in\Delta
\quad \text{in } \mathrm{D}^b(\mathcal{Y})/\mathcal{K}_{-1}.
\]
Applying $\Theta^!$ and using the isomorphism $\Theta^!\Theta \simeq \mathrm{id}$, we obtain $\Theta^! c = \pi_* \varpi^! c = 0$. Hence $\varpi^! c \in \ker(\pi_*)$, and therefore it lies in the triangulated subcategory
\[
\mathcal{T}:=
\left\langle
\iota^+_* \mathcal{O}_{\mathrm{M}}(-e+1)\otimes \bigoplus_{i=0}^{r-1}\nu^i,\ \ldots,\
\iota^+_* \mathcal{O}_{\mathrm{M}}(-1)\otimes \bigoplus_{i=0}^{r-1}\nu^i,\
\left\langle \iota^+_* \mathcal{A}_{-1}\otimes \bigoplus_{i=0}^{r-1}\nu^i \right\rangle
\right\rangle.
\]
Since $c \simeq \varpi_* \varpi^! c$, we deduce that $c$ lies in the essential image of $\tau$ in $\mathrm{D}^b(\mathcal{Y})/\mathcal{K}_{-1}$, which is generated by
\[
\iota_* \mathcal{O}_{E_\mathcal{Y}}(-e+1),\ldots,\iota_* \mathcal{O}_{E_\mathcal{Y}}(-1),
\]
and equivalent to the category generated by
\[
\iota_* \mathcal{D}_{e-1},\ldots,\iota_* \mathcal{D}_{1}.
\]
By the Lemma \ref{mutationlemma}, the objects $\iota_* \mathcal{D}_{e-1},\ldots,\iota_* \mathcal{D}_{r}$ lie in this generated subcategory. Moreover, from the computation of $\pi_*(\mathcal{O}_C\otimes \nu^j)$ for $0\le j<r$ in (\ref{pullbackcompu1}), we also obtain that $\iota_* \mathcal{D}_{r-1},\ldots,\iota_* \mathcal{D}_{0}$ belong to the same subcategory. Hence $c$ lies in the generated category, and therefore $b$ also belongs to it.
\end{proof}

When $e \ge r$, there are several analogous variants whose proofs are entirely similar. The first useful variant is obtained by considering the full subcategory $\mathcal{K}_{-r}:=\langle \iota_* \mathcal{A}_{-r}\rangle$.

\begin{proposition}
The local stacky Kawamata blow-up admits a semi-orthogonal decomposition
\[
\mathrm{D}^b(\mathcal{Y})/\mathcal{K}_{-r}
=
\left\langle
\mathcal{O}_{E_\mathcal{Y}}(-e+1),\dots,\mathcal{O}_{E_\mathcal{Y}}(-r),\operatorname{Im}\Xi
\right\rangle,
\]
where $\Xi := \varpi_! \mathcal{R}\pi^*$ defines a well-defined fully faithful admissible triangulated functor
\[
\Xi : \mathrm{D}^b(\mathcal{X}) \to \mathrm{D}^b(\mathcal{Y})/\mathcal{K}_{-r},
\]
and the objects $\mathcal{O}_{E_\mathcal{Y}}(-e+1),\dots,\mathcal{O}_{E_\mathcal{Y}}(-r)$ are pretilting or exceptional, assuming $e \ge r$.
\end{proposition}

Another generalization is to consider a formal well-formed variety along a projective  center $C$, where the cover formal neighborhood can be realized as a complete intersection in GIT quotient of the form
\[
\mathds{A} := \mathbb{A}_{C}\Bigl[\Bigl[\bigoplus_i \mathcal{O}_{C}(b_i)\Bigr]\Bigr].
\]
here $C$ is a smooth complete intersection inside a weighted projective stack $\mathcal{P}$.
\begin{lemma}[Dévissage]
In the global situation, one can compute formally
\[
\mathcal{R}\pi^*(\kappa_* (F\otimes \nu^i)),
\]
for any $F\in \mathrm{D}^b(C)$ and any $0\le i<r$. It admits the following filtration:
\[
\begin{tikzcd}[column sep=0.6em]
\pi^*(\kappa_*F\otimes \nu^i)\arrow{rr} && \mathcal{F}^s([F\otimes \nu^i]^+)\arrow{dl}\arrow{rr} && \mathcal{F}^{s-1}\arrow{dl} \\
& \tau_s\arrow[ul,dashed,"\Delta"] && \iota^+_*(f^{+*}F\otimes \nu^i)\otimes\mathcal{D}_s\otimes \nu^{-s}[1]\arrow[ul,dashed,"\Delta"]
\end{tikzcd}
\cdots
\]
\[
\begin{tikzcd}[column sep=0.6em]
& \mathcal{F}^1\arrow{rr} && \mathcal{F}^0\arrow{rr}\arrow{dl} && 0\arrow{dl} \\
&& \iota^+_*(f^{+*}F\otimes \nu^i)\otimes \mathcal{D}_1\otimes \nu^{-1}[1]\arrow[ul,dashed,"\Delta"] && \iota^+_*(f^{+*}F\otimes \nu^i)\otimes \mathcal{D}_0[1]\arrow[ul,dashed,"\Delta"]
\end{tikzcd}
\]
where $\mathcal{D}_i$ is defined as the $i$-fold left mutation of $\mathcal{O}_{E_{\mathcal{Y}}}(i)$ along the semi-orthogonal decomposition
\[
\mathrm{D}^b(E_{\mathcal{Y}})=
\left\langle
\mathcal{A}_{-1},
f_{E_{\mathcal{Y}}}^*\mathrm{D}^b(C),\,
f_{E_{\mathcal{Y}}}^*\mathrm{D}^b(C)\otimes \mathcal{O}_{E_{\mathcal{Y}}}(1),\,
\ldots,\,
f_{E_{\mathcal{Y}}}^*\mathrm{D}^b(C)\otimes \mathcal{O}_{E_{\mathcal{Y}}}(e-1)
\right\rangle,
\]
for $0\le i<e$, and moreover $\mathcal{D}_i\in \mathcal{A}_{-1}$ for $i\ge e$.
\end{lemma}
\begin{proof}
Consider the following section of the cover of our formal neighbourhood
\[
\begin{tikzcd}[column sep=3.2em,row sep=2.2em]
C \arrow[dr,"\mathrm{id}"]
\arrow[r,hook,"\kappa"]
& \mathds{A}:=X_{\widehat C}
\arrow[d,"\psi"]
\\
& C
\end{tikzcd}
\]
We have
\[
\pi^*\kappa_*F
=
\pi^*(\kappa_*\mathcal{O}_C\otimes \psi^*F)
=
\pi^*(\kappa_*\mathcal{O}_C)\otimes (\pi^*\psi^*F),
\]
where one may properly work in the derived category of $\mathrm{Qcoh}$. Moreover, we have a natural filtration of $\pi^*(\kappa_*\mathcal{O}_C)$ as  (\ref{pullbackcompu}), together with
\[
\iota^{+*}\pi^*\psi^*F = \kappa^{*}\psi^*\pi_E^{*}F= \pi_E^{*}F,
\]
and since the new factors are coherent termwise, the conclusion follows if we define $\mathcal{R}\pi^*(F\otimes \nu^i)$ to be $\mathcal{F}^s([F\otimes \nu^i]^+)$ for $s\gg 0$.
\end{proof}
We can continue similar computations for $\varpi_*\mathcal{R}\pi^*(F\otimes \nu^i)$ for $F\in \mathrm{D}^b(C)$. In fact, under our assumptions these computations all reduce to the trivial case; we refer to this simplification as a \textit{dévissage argument}. In analogy with mutations for exceptional objects, we now consider mutations with respect to admissible components, and obtain the following:
\begin{proposition}
The local stacky Kawamata blow-up admits a semi-orthogonal decomposition
\[
\mathrm{D}^b(\mathcal{Y})/\mathcal{K}_{-1}
=
\left\langle
\operatorname{Im}\Theta_{-e+r},\ldots,\operatorname{Im}\Theta_{-1},\operatorname{Im}\Theta
\right\rangle,
\]
here the functors $\Theta$ and $\Theta_i$ are defined via the following diagram:
\[
\begin{tikzcd}
E_{\mathcal{Y}} \arrow[r,"\iota"]\arrow[d,"f_{E_{\mathcal{Y}}}"]
& \mathcal{Y}\arrow[d,"f"]
& \sqrt[r]{E_{\mathcal{Y}}/\mathcal{Y}} \arrow[r,"\varpi"]\arrow[d,"\pi"]
& \mathcal{Y}\arrow[d,"f"] \\
C \arrow[r,"\kappa_X"]
& X
& \mathcal{X} \arrow[r,"\varepsilon"]
& X
\end{tikzcd}
\]
we set
\[
\Theta(-):=\varpi_* \mathcal{R}\pi^*(-),
\qquad
\Theta_i(-):=\iota_* f_{E_{\mathcal{Y}}}^*(-)\otimes \mathcal{O}_{E_{\mathcal{Y}}}(i),
\quad i\in\mathbb{Z}.
\]
where $E_{\mathcal{Y}}$ is a complete intersection inside a weighted projective fibration over $C$, and admits a semi-orthogonal decomposition
\[
\mathrm{D}^b(E_{\mathcal{Y}})=
\left\langle
f_{E_{\mathcal{Y}}}^*\mathrm{D}^b(C)\otimes \mathcal{O}_{E_{\mathcal{Y}}}(-e+1),\ldots,
f_{E_{\mathcal{Y}}}^*\mathrm{D}^b(C)\otimes \mathcal{O}_{E_{\mathcal{Y}}}(-1),
\mathcal{A}_{-1},
f_{E_{\mathcal{Y}}}^*\mathrm{D}^b(C)
\right\rangle.
\]
$\mathcal{K}_{-1}$ denotes the full saturated triangulated subcategory of $\mathrm{D}^b(\mathcal{Y})$ generated by $\iota_*\mathcal{A}_{-1}$. $\Theta$ defines a well-defined fully faithful admissible triangulated functor
\[
\Theta:\mathrm{D}^b(\mathcal{X}) \to \mathrm{D}^b(\mathcal{Y})/\mathcal{K}_{-1},
\]
assuming $e \ge r$.
\end{proposition}

We now consider the global situation. Let $\mathcal{X}$ be a well-formed stack over an infinite field $\mathrm{k}$. Assume that it admits a well-formed neighborhood along a singular center $C$, which is formal locally given by  a complete intersection in an affine GIT quotient. Moreover, there exist global GIT coordinates $x_1,\dots,x_s$ equipped with a $\mu_r$-action with weights $a_i$, inducing all local coordinate quotients. We assume that the defining ideal $I$ of the well-formed singularity can be realized as a complete intersection in GIT coordinates and satisfies the measured condition (\ref{cond:measured}). Then we can then define the Kawamata blow-up of $\mathcal{X}$. Namely, for any well-formed atlas locally of the form $[U:=\operatorname{Spec} R/\mu_r]$, we consider the valuation induced by the weights of the quotient, this gives rise to a natural weighted valuation
\[
val : K([R/\mu_r]) \longrightarrow \mathbb{Z},
\]
which in turn defines a filtration of ideals $\{P_i\}$. Using this filtration, one constructs the Kawamata blow-up as in the formal case. Since the weights are globally defined, these local constructions glue compatibly. Moreover, the resulting formal completion along the singular locus $C$ agrees with the previously described formal model. We may even assume that $\mathcal{Y}$ is not well-formed,  but once $\mathcal{Y}$ is well-formed (or $Y$ is well-formed and the exceptional divisor $E_{Y}$ is prime), we have $\mathcal{Y}\simeq \widetilde{Y}$.
\[
\begin{tikzcd}
\mathcal{Z} :=w\mathrm{Bl}^{(a_i)}_C \mathcal{X} \arrow[r,"\varpi"] \arrow[d,"\pi"] & \mathcal{Y} := \widetilde{w\mathrm{Bl}^{(a_i)/r}_C X} \arrow[d] \\
\mathcal{X} \arrow[r] & X
\end{tikzcd}
\]
The morphism $\pi$ is typically not representable. Nevertheless, one can consider the derived functors $\pi_* : \mathrm{D}^+(\mathrm{QCoh}(\mathcal{Z})) \to \mathrm{D}^+(\mathrm{QCoh}(\mathcal{X}))$ and $\pi^* : \mathrm{D}^-(\mathrm{QCoh}(\mathcal{X})) \to \mathrm{D}^-(\mathrm{QCoh}(\mathcal{Z}))$, which later satisfy flat base change on $\mathrm{QCoh}$. This allows us to avoid imposing additional assumptions such as the resolution property on $\mathcal{X}$, and yields induced functors $\pi^* : \mathrm{D}^-(\mathrm{Coh}(\mathcal{X})) \to \mathrm{D}^-(\mathrm{Coh}(\mathcal{Z}))$ and $\pi^! : \mathrm{D}^+(\mathrm{Coh}(\mathcal{Z})) \to \mathrm{D}^+(\mathrm{Coh}(\mathcal{X}))$, compatible with the corresponding local constructions. By the same classical arguments, for any $F \in \mathrm{D}^b(\mathcal{X})$, we consider its formal completion along $C$ under a local atlas, e.g. $\widehat{F} \in \mathrm{D}^b([U/\mu_r]^\wedge_C)$. One constructs $\mathcal{R}\pi^*\widehat{F}$ fitting into a distinguished triangle $\pi^*(\widehat{F}) \to \mathcal{R}\pi^*(\widehat{F}) \to \tau$, where $\tau$ is a bounded above complex supported on the exceptional divisor $\mathrm{M}$. This yields an algebraic counterpart \[\pi^*F \to \mathcal{R}\pi^*F \to \tau\in \Delta\,\] and the correspondence is bijective by construction. Gluing over the atlas then produces a global functor $\mathcal{R}\pi^* : \mathrm{D}^b(\mathcal{X}) \to \mathrm{D}^b(\mathcal{Z})/\mathcal{K}_{-1}^+$.\\

In this situation, we use a completely analogous local-to-global argument as in \cite[Lemma~3.13]{Hao2025b} (possibly in a different version). We have:
\begin{proposition}
The stacky Kawamata blow-up admits a semi-orthogonal decomposition
\[
\mathrm{D}^b(\mathcal{Y})/\mathcal{K}_{-1}
=
\left\langle
\operatorname{Im}\Theta_{-e+r},\ldots,\operatorname{Im}\Theta_{-1},\operatorname{Im}\Theta
\right\rangle,
\]
here the functors $\Theta$ and $\Theta_i$ are defined via the following diagram:
\[
\begin{tikzcd}
E_{\mathcal{Y}} \arrow[r,"\iota"]\arrow[d,"f_{E_{\mathcal{Y}}}"]
& \mathcal{Y}\arrow[d,"f"]
& \sqrt[r]{E_{\mathcal{Y}}/\mathcal{Y}} \arrow[r,"\varpi"]\arrow[d,"\pi"]
& \mathcal{Y}\arrow[d,"f"] \\
C \arrow[r,"\kappa_X"]
& X
& \mathcal{X} \arrow[r,"\varepsilon"]
& X
\end{tikzcd}
\]
we set
\[
\Theta(-):=\varpi_* \mathcal{R}\pi^*(-),
\qquad
\Theta_i(-):=\iota_* f_{E_{\mathcal{Y}}}^*(-)\otimes \mathcal{O}_{E_{\mathcal{Y}}}(i),
\quad i\in\mathbb{Z}.
\]
where $\Theta$ defines a well-defined fully faithful admissible triangulated functor
\[
\Theta:\mathrm{D}^b(\mathcal{X}) \to \mathrm{D}^b(\mathcal{Y})/\mathcal{K}_{-1},
\]
assuming $e \ge r$.
\end{proposition}

\subsubsection{Kawamata blow-up II}
In this section, we consider the stacky Kawamata blow-up introduced above and study the case $e < r$:
\[
\begin{tikzcd}[column sep=3.2em, row sep=2.2em]
\mathcal{Z} := w\mathrm{Bl}^{(a_i)}_C \mathcal{X}
\arrow[r,"\varpi"] \arrow[d,"\pi"']
& \mathcal{Y} := \widetilde{w\mathrm{Bl}^{(a_i/r)}_C X} \arrow[d,"f"] \\
\mathcal{X} \arrow[r,"\varepsilon"]
& X
\end{tikzcd}
\]
We first consider the local model. One can define a natural functor
\[
\varpi_* \mathcal{R}\pi^* : \mathrm{D}^b(\mathcal{X}) \longrightarrow \mathrm{D}^b(\mathcal{Y})/\mathcal{K}_{-1},
\]
which satisfies\label{pullbackcomII}
\[
\begin{tikzcd}[column sep=0.6em]
\varpi_*\pi^*(\kappa_*\mathcal{O}_C\otimes \nu^i)\arrow{rr} && \mathcal{F}_{(i)}^s([\mathcal{O}_{C}]^+)\arrow{dl}\arrow{rr} && \mathcal{F}_{(i)}^{s-1}\arrow{dl} \\
& \tau^{(i)}_s\arrow[ul,dashed,"\Delta"] && \iota_*\mathcal{D}_{sr+i}[1]\arrow[ul,dashed,"\Delta"]
\end{tikzcd}
\quad\cdots\quad
\begin{tikzcd}[column sep=0.6em]
& \mathcal{F}_{(i)}^1\arrow{rr} && \mathcal{F}_{(i)}^0\arrow{rr}\arrow{dl} && 0\arrow{dl} \\
&& \iota_*\mathcal{D}_{r+i}[1]\arrow[ul,dashed,"\Delta"] && \iota_*\mathcal{D}_{i}[1]\arrow[ul,dashed,"\Delta"]
\end{tikzcd}
\]
For any $i\geq e$, all the above factors lie in $\mathcal{K}_{-1}$. Moreover, we have computed that
\[
\varpi_* \mathcal{R}\pi^*\bigl(\kappa_*\mathcal{O}_C\otimes \nu^i\bigr)=0.
\]
Hence we consider the full saturated triangulated subcategory
\[
\mathcal{L}_- := \left\langle \kappa_*\mathcal{O}_C\otimes \nu^{e},\ldots,\kappa_*\mathcal{O}_C\otimes \nu^{r-1} \right\rangle.
\]
By the triangulated property of the functor, we have $\varpi_* \mathcal{R}\pi^*(\mathcal{L}_-)=0$, and therefore it induces a well-defined triangulated functor
\[
\varpi_* \mathcal{R}\pi^* : \mathrm{D}^b(\mathcal{X})/\mathcal{L}_- \longrightarrow \mathrm{D}^b(\mathcal{Y})/\mathcal{K}_{-1}.
\]
Similarly, in the opposite direction, we define a functor
\[
\pi_* \varpi^! : \mathrm{D}^b(\mathcal{Y}) \longrightarrow \mathrm{D}^b(\mathcal{X})/\mathcal{L}_-.
\]
For any $N \in \mathcal{A}_{-1}$, the object $\varpi^!\iota_*N$ admits a filtration whose factors are of the form $\iota^+_*N(-l)\otimes \nu^l$ for $l=0,\ldots,r-1$, where $N(-l)\in \mathcal{A}_{-1-l}$. For $0\le l< e$, we have $\mathrm{H}^*(E_\mathcal{Y}, N(-l))=0$, while for $e\le l<r$, the objects $\pi_*\iota^+_*N(-l)\otimes \nu^l$ are supported on $C$ and admit equivariant presentations $\nu^l$. It follows that
\[
\pi_*\varpi^!\iota_*N=0 \quad \text{in } \mathrm{D}^b(\mathcal{X})/\mathcal{L}_-.
\]
Consequently, $\pi_* \varpi^!$ descends to a well-defined triangulated functor
\[
\pi_* \varpi^! : \mathrm{D}^b(\mathcal{Y})/\mathcal{K}_{-1} \longrightarrow \mathrm{D}^b(\mathcal{X})/\mathcal{L}_-.
\]
\begin{lemma}
The functors defined above form an adjoint pair
\[
\varpi_* \mathcal{R}\pi^* \;\dashv\; \pi_* \varpi^!.
\]
\end{lemma}
\begin{proof}
We consider the natural map
\[
\left[
\begin{tikzcd}[column sep=small, row sep=small]
\varpi_* \mathcal{R}\pi^* F \arrow[dr, "f"] & & G \arrow[dl, "t"'] \\
& G^- &
\end{tikzcd}
\right]
\quad \longmapsto \quad
\left[\begin{tikzcd}[column sep=small, row sep=small]
F \arrow[dr, "\pi_*\varpi^!f"] & & \pi_*\varpi^!G \arrow[dl, "\pi_*\varpi^!t"'] \\
& \pi_*\varpi^!G^- &
\end{tikzcd}\right],
\]
then use the following construction to show that the map is natural and surjective.
\[
\left[
\begin{tikzcd}[column sep=small, row sep=small]
& F^+ \arrow[dl, "s"'] \arrow[dr, "g"] & \\
F & & \pi_*\varpi^!G
\end{tikzcd}
\right]
\quad \longmapsto \quad
\left[
\begin{tikzcd}[column sep=small, row sep=small]
& \varpi_* \mathcal{R}^{s\gg0}\pi^*F^+ \arrow[dl, "\varpi_*\mathcal{R}\pi^*s"'] \arrow[dr, "\varpi_*\pi^*g"] & \\
\varpi_* \mathcal{R}\pi^* F & & G
\end{tikzcd}
\right].
\]
The proof of injectivity is also similar; in fact, the second map provides the inverse. Hence it induces the adjunction.
\end{proof}
Similarly, we obtain a well-defined triangulated functor
\[
\pi_! \varpi^* :
\mathrm{D}^b(\mathcal{Y})/\mathcal{K}_{-1}
\longrightarrow
\mathrm{D}^b(\mathcal{X})/\mathcal{L}_-,
\]
together with the following properties.
\begin{lemma}
The functor $\varpi_* \mathcal{R}\pi^*$ is admissible and fits into an adjoint triple
\[
\pi_! \varpi^* \;\dashv\; \varpi_* \mathcal{R}\pi^* \;\dashv\; \pi_* \varpi^!.
\]
\end{lemma}
Similarly, by a dévissage argument, the above results extend to the global situation. The natural replacement is the full saturated triangulated subcategory
\[
\mathcal{L}_- :=
\left\langle
\kappa_*\mathrm{D}^b(C)\otimes \nu^{e},\ldots,
\kappa_*\mathrm{D}^b(C)\otimes \nu^{r-1}
\right\rangle.
\]
Now we consider a global well-formed stack $\mathcal{X}$, and have the following:
\begin{lemma}
Consider the following collection in $\mathrm{D}^b(\mathcal{X})/\mathcal{L}_-$:
\[
\Omega :=
\left\{
\kappa_*\mathrm{D}^b(C),\,
\kappa_*\mathrm{D}^b(C)\otimes \nu^1,\,
\ldots,\,
\kappa_*\mathrm{D}^b(C)\otimes \nu^{e-1}
\right\}
\cup
\{\mathrm{k}_x\otimes\chi \mid x \text{ is a closed point not in } C,\ \chi\}.
\]
Then $\Omega$ forms a spanning class in $\mathrm{D}^b(\mathcal{X})/\mathcal{L}_-$. Moreover, we may replace $\mathrm{D}^b(C)$ here by any its strong generator.
\end{lemma}
\begin{proof}
By the definition of a spanning class, it suffices to show that for any
\[
F\in \mathrm{D}^b(\mathcal{X})/\mathcal{L}_-
\]
satisfying $\operatorname{Ext}_{ \mathrm{D}^b(\mathcal{X})/\mathcal{L}_-}^*(F,a)=0$ for all $a\in \Omega$, one has $F=0$ in $\mathrm{D}^b(\mathcal{X})/\mathcal{L}_-$.\\

First, since $\operatorname{Ext}_{ \mathrm{D}^b(\mathcal{X})/\mathcal{L}_-}^*(F,\mathrm{k}_x\otimes\chi)=0$ for every closed point $x\notin C$ and $\chi$, we have $F_x=0$ for all such $x$. Hence, by the Jacobian property of $\mathcal{X}$, the object $F$ is supported on $C$. Since $F$ comes from an object in the bounded derived category of coherent sheaves, it can be realized as a finite iterated extension of objects of the form
\[
\kappa_*\mathrm{D}^b(C)\otimes \nu^i,
\qquad
0\le i<r,
\]
(and if one uses a strong generator, one may additionally need to take direct summands). By the vanishing assumption $\operatorname{Ext}_{ \mathrm{D}^b(\mathcal{X})/\mathcal{L}_-}^*(F,a)=0$ for
\[
\kappa_*\mathrm{D}^b(C),\,
\kappa_*\mathrm{D}^b(C)\otimes \nu^1,\,
\ldots,\,
\kappa_*\mathrm{D}^b(C)\otimes \nu^{e-1},
\]
we obtain $\operatorname{Ext}_{ \mathrm{D}^b(\mathcal{X})/\mathcal{L}_-}^*(F,F)=0$. It follows that $F\in \mathcal{L}_-$, hence
$F=0$ in $\mathrm{D}^b(\mathcal{X})/\mathcal{L}_-$.
The verification in the opposite direction is similar.
\end{proof}
\begin{lemma}\label{adjucomp}
Assume $e < r$. Then we have
\[
\pi_* \varpi^!
\Bigl(
\varpi_* \mathcal{R}\pi^*
(\kappa_*F\otimes \nu^i)
\Bigr)
\simeq
\kappa_*F\otimes \nu^i
\]
for any $F\in \mathrm{D}^b(C)$ and any $0\le i\le e-1$.
\end{lemma}

\begin{proof}
By dévissage, we may reduce $F$ to the structure sheaf of $C$. Since $F$ is supported on the center of the blow-up, the computation can be carried out formally, and it suffices to consider the case $F$ is $\mathcal{O}_C$. We have
\[
\varpi_* \mathcal{R}\pi^*(\kappa_*\mathcal{O}_C\otimes \nu^i)=\iota_*\mathcal{D}_i,
\qquad 0\le i<e.
\]
Moreover, $\varpi^!\iota_*\mathcal{D}_i$ admits a filtration with successive factors $\iota^+_*\mathcal{D}_i(-l)\otimes \nu^l$ for $l=0,\ldots,r-1$ by Lemma \ref{derivedroot}. For $0\le l<e$, we have via Lemma \ref{orddualcomp}:
\[
\mathrm{Ext}_{E_\mathcal{Y}}^*(\mathcal{O}(l),\mathcal{D}_i)=\mathcal{O}_C\cdot\delta_{l,i},
\]
so the only nontrivial factor of $\pi_*\iota^+_*\mathcal{D}_i(-l)\otimes \nu^l$ occurs when $l=i$, in which case it is $\kappa_*\mathcal{O}_C\otimes \nu^i$. For $e\le l<r$, we have $\pi_*\iota^+_*\mathcal{D}_i(-l)\otimes \nu^l=\mathrm{V}\otimes \nu^l$, where $\mathrm{V}$ is a complex support on $C$ carries no representations, hence these terms lie in $\mathcal{L}_-$ and contribute zero in the quotient. Therefore
\[
\pi_* \varpi^!
\Bigl(
\varpi_* \mathcal{R}\pi^*(\kappa_*\mathcal{O}_C\otimes \nu^i)
\Bigr)
\simeq
\kappa_*\mathcal{O}_C\otimes \nu^i.
\]
for any $0\le i\le e-1$
\end{proof}
\begin{corollary}
Assume $e<r$. Then the functor $\varpi_* \mathcal{R}\pi^*$ is fully faithful.
\end{corollary}

\begin{proof}
We know that $\varpi_* \mathcal{R}\pi^*$ is admissible, and
\[
\Omega :=
\left\{
\kappa_*\mathrm{D}^b(C),\,
\kappa_*\mathrm{D}^b(C)\otimes \nu^1,\,
\ldots,\,
\kappa_*\mathrm{D}^b(C)\otimes \nu^{e-1}
\right\}
\cup
\{\mathrm{k}_x\otimes\chi \mid x \text{ is a closed point not in } C\}
\]
forms a spanning class for $\mathrm{D}^b(\mathcal{X})/\mathcal{L}_{-}$. The full faithfulness of $\varpi_* \mathcal{R}\pi^*$ is equivalent to
\[
\operatorname{Ext}^*_{\mathrm{D}^b(\mathcal{Y})/\mathcal{K}_{-1}}
\bigl(\varpi_* \mathcal{R}\pi^* a,\varpi_* \mathcal{R}\pi^* b\bigr)
\simeq
\operatorname{Ext}^*_{\mathrm{D}^b(\mathcal{X})/\mathcal{L}_{-}}(a,b)
\quad \text{for all } a,b\in \Omega.
\]
by \cite[Proposition 1.49]{Huybrechts}. Moreover by Lemma \ref{adjucomp},
\[
\operatorname{Ext}^*_{\mathrm{D}^b(\mathcal{Y})/\mathcal{K}_{-1}}
\bigl(\varpi_* \mathcal{R}\pi^* a,\varpi_* \mathcal{R}\pi^* b\bigr)
\simeq
\operatorname{Ext}^*_{\mathrm{D}^b(\mathcal{X})/\mathcal{L}_{-}}
\bigl(a,\pi_* \varpi^! \varpi_* \mathcal{R}\pi^* b\bigr)
\simeq
\operatorname{Ext}^*_{\mathrm{D}^b(\mathcal{X})/\mathcal{L}_{-}}(a,b),
\]
and note that this equality is trivial for skyscraper sheaves supported outside the exceptional locus, so the conclusion follows.
\end{proof}
\begin{proposition}
Assume $e<r$. Then the functor
\[
\varpi_* \mathcal{R}\pi^* :
\mathrm{D}^b(\mathcal{X})/\mathcal{L}_-
\longrightarrow
\mathrm{D}^b(\mathcal{Y})/\mathcal{K}_{-1}
\]
defines an equivalence of triangulated categories.
\end{proposition}
\begin{proof}
It remains to prove essential surjectivity. For any $G\in \mathrm{D}^b(\mathcal{Y})$, consider the distinguished triangle in $\mathrm{D}^b(\mathcal{Z})$ induced by the natural adjunction map
\[
\mathcal{R}\pi^*\pi_*\varpi^!G
\longrightarrow
\varpi^!G
\longrightarrow
\mathcal{Q}
\in \Delta.
\]
It is easy to see that $\pi_*\mathcal{Q}=0$, hence $\mathcal{Q}\in \mathcal{T}$. Pushing forward along $\varpi$, we obtain a distinguished triangle
\[
\varpi_*\mathcal{R}\pi^*\pi_*\varpi^!G
\longrightarrow
\varpi_*\varpi^!G \simeq G
\longrightarrow
\varpi_*\mathcal{Q}
\in \Delta.
\]
so it remains to show that $\varpi_*\mathcal{Q}$ lies in the image of our functor. By a dévissage argument, it suffices to show that for any $\mathcal{O}_{E_{\mathcal{Y}}}(-i)$ with $0<i<e$, these objects lie in the image of the functor. But we have
\[
\varpi_*\mathcal{R}\pi^*(\kappa_*\mathcal{O}_C\otimes \nu^i)
=
\iota_*\mathcal{D}_i,
\qquad 0\leq i<e,
\]
in (\ref{pullbackcomII}), so the conclusion follows from Lemma \ref{mutationlemma}.
\end{proof}
\begin{remark}
We expect the global equivalence to induce one on local, e.g. via approximation.
\end{remark}

\subsubsection{Coarse moduli}\label{seckblcm}

In this section, we consider the stacky Kawamata blow-up introduced above and study its coarse moduli spaces under the assumption that $\mathcal{Y}$ is also well-formed (or stronger, $Y$ is well-formed and $E_Y$ is a prime divisor).
\[
\begin{tikzcd}[column sep=2.5em, row sep=2.2em]
\mathcal{Z} := w\mathrm{Bl}^{(a_i)}_C \mathcal{X} \arrow[r,"\varpi"] \arrow[d,"\pi"']
& \mathcal{Y} := \widetilde{w\mathrm{Bl}^{(a_i/r)}_C X} \arrow[r,"\epsilon"] \arrow[d,"f"]
& Y := w\mathrm{Bl}^{(a_i/r)}_C X \arrow[dl,"o"] \\
\mathcal{X} \arrow[r,"\varepsilon"]
& X &
\end{tikzcd}
\]
Here $o$ is induced by the universal property of the coarse moduli space. The morphism $\varepsilon_*$ restricts to the exceptional divisor $E_{\mathcal{Y}}$, which is a fibration of complete intersections in a weighted projective stack, and maps it to its coarse moduli space $E_Y$. Here we relax the assumption to $e \ge 1$. In this case, the space $E_Y$ is always a Fano fibration over the smooth center $C$, and we could consider the semi-orthogonal decomposition
\[
\mathrm{D}^b(E_{\mathcal{Y}})\simeq
\left\langle
f_{E_{\mathcal{Y}}}^*\mathrm{D}^b(C)\otimes \mathcal{O}_{E_{\mathcal{Y}}}(-e+1),\ldots,
f_{E_{\mathcal{Y}}}^*\mathrm{D}^b(C)\otimes \mathcal{O}_{E_{\mathcal{Y}}}(-1),
\mathcal{A}_{-1},
f_{E_{\mathcal{Y}}}^*\mathrm{D}^b(C)
\right\rangle,
\]
together with $\varepsilon_* \mathcal{O}_{E_{\mathcal{Y}}}\simeq\mathcal{O}_{E_Y}$ and $f_{E_{\mathcal{Y}}*}\mathcal{O}_{E_{\mathcal{Y}}} \simeq f_{E_Y*}\mathcal{O}_{E_Y}\simeq\mathcal{O}_C$, induce a semi-orthogonal decomposition
\[
\mathrm{D}^b(E_Y)\simeq\left\langle \mathcal{B}_{-1},\, f_{E_Y}^*\mathrm{D}^b(C) \right\rangle.
\]
Moreover, the left orthogonal component $\langle f_{E_{\mathcal{Y}}}^*\mathrm{D}^b(C)\rangle^\perp \subset \mathrm{D}^b(E_{\mathcal{Y}})$ is mapped by $\varepsilon_*$ into $\mathcal{B}_{-1}$. By the theory of root stacks and Cohen--Macaulay approximation, this pushforward is essentially surjective and satisfies a corresponding localization formula. We therefore consider the full saturated triangulated subcategory $\mathcal{L}_- \subset \mathrm{D}^b(Y)$ generated by $\iota^-_* \mathcal{B}_{-1}$.
\begin{lemma}
If $F \in \mathrm{D}^b(Y)$ satisfies $o_* F = 0$, then $F \in \mathcal{L}_-$.
\end{lemma}
\begin{proof}
We first consider the local situation. We start from the distinguished triangle
\[
\varepsilon^*(\kappa_{X*}\mathcal{O}_C)\to \kappa_{*}\mathcal{O}_C \to \mathcal{Q},
\]
where $\mathcal{Q}$ admits a bounded above filtration. More precisely using a dual version of  Proposition \ref{CMappro},
\[
\varepsilon^*\kappa_{X*}\mathcal{O}_C \simeq \varprojlim_s \mathcal{H}^s,\qquad
\mathcal{H}^s := \tau^{\geq s}(\varepsilon^*\kappa_{X*}\mathcal{O}_C),
\]
whose graded pieces are generated by $\kappa_{*}\mathcal{O}_C\otimes \nu^i$ for $1\le i\le r-1$.
For each $s$, we consider $\mathcal{H}^s$ and compute its pullback along $\pi$, which induces a second filtration from previous sections
\[
\pi^*\mathcal{H}^s \simeq \varprojlim_t \mathcal{F}^{(s,t)}.
\]
In particular $\mathcal{F}^{(s,0)}$ fits into a bounded complex with terms in $\tau$ and $\iota_*\mathcal{O}_{E_{\mathcal{Y}}}$. The successive cones in the filtrations
\[
\mathcal{F}^{(s,t+1)}\to \mathcal{F}^{(s,t)}\to \mathcal{Q}^t,\qquad
\mathcal{F}^{(s,t)}\to \mathcal{F}^{(s+1,t)}\to \mathcal{Q}^s
\]
$\mathcal{Q}^t$ and  $\mathcal{Q}^s$ lie in $\mathcal{T}$.
Pushing forward along $\varpi$ and $\epsilon$, we obtain a filtration on
\[
\mathcal{G}^{(s,t)} := \epsilon_*\varpi_*\mathcal{F}^{(s,t)}.
\]
This yields
\[
\mathcal{Q}^0 \to \mathcal{G}^0([\mathcal{O}_{C}]^+) \to \iota^-_*\mathcal{O}_{E_Y},
\]
and the cones in
\[
\mathcal{G}^{(s,t+1)}\to \mathcal{G}^{(s,t)}\to \mathcal{Q}^t,\qquad
\mathcal{G}^{(s+1,t)}\to \mathcal{G}^{(s,t)}\to \mathcal{Q}^s
\]
all lie in $\mathcal{L}_-$. That is a corresponding filtration $\mathcal{G}^{(s,t)}$ on
\[
o^*\kappa_{X*}\mathcal{O}_C = \epsilon_*\varpi_*\pi^*\varepsilon^*\kappa_{X*}\mathcal{O}_C,
\]
such that $o^*\kappa_{X*}\mathcal{O}_C \simeq \varprojlim_{s,t} \mathcal{G}^{(s,t)}$. Once such a filtration is given, we may apply the local adjunction \cite[Lemma 3.7]{Hao2025b}  to construct a natural isomorphism
\[
\operatorname{Ext}^*_{\mathrm{D}^b(Y)/\mathcal{L}_-}([\mathcal{G}^{(s,t)}\simeq \iota^-_*\mathcal{O}_{E_Y}], F)
\simeq
\operatorname{Ext}^*_{\mathrm{D}^b(X)}(\kappa_{X*}\mathcal{O}_C, \pi_* F),
\]
for any object $F \in \mathrm{D}^b(Y)$. Moreover, since we are in the local situation, any $F$ satisfying $\pi_*F=0$ is supported on the exceptional divisor $E_Y$, and can be obtained by successive extensions of $\mathcal{O}_{E_Y}$ in the quotient category $\mathrm{D}^b(Y)/\mathcal{L}_-$. Hence we obtain $\operatorname{Ext}^*_{\mathrm{D}^b(Y)/\mathcal{L}_-}(F,F)=0$, and therefore $F \in \mathcal{L}_-$. In the global case, the same conclusion follows by a dévissage argument.
\end{proof}

We can now define a (quasi-)functor
\[
\mathcal{R}o^* : \mathrm{D}^b(X) \longrightarrow \mathrm{D}^b(Y)
\]
with a mild ambiguity by
\[
\mathcal{R}o^*(-) := \epsilon_* \varpi_* \mathcal{R}\pi^* \widetilde{(-)}.
\]
Similarly, we obtain a bifiltration on $o^*(F)$, written as
$o^*(F) \simeq\varprojlim_{s,t} \mathcal{R}^{(s,t)}o^*(F)$, such that all successive quotients (in either direction of the filtration) lie in $\mathcal{L}_-$.
\begin{lemma}\label{coasefullcomp}
We have $o_* \mathcal{R}o^* \simeq \mathrm{id}$.
\end{lemma}
\begin{proof}
To simplify the discussion, we use certain unbounded pushforward arguments. For any $F$ we consider $\epsilon_*\varpi_*\pi^*\varepsilon^*F$, which satisfies
$o_*\epsilon_*\varpi_*\pi^*\varepsilon^*F \simeq o_*o^*F \simeq F$. Moreover, for any $\mathcal{Q} \in \mathcal{L}_-$ we have $o_*\mathcal{Q}=0$, hence we obtain the conclusion.
\end{proof}
\begin{lemma}
We have a natural isomorphism
\[
\operatorname{Ext}^*_{\mathrm{D}^b(Y)/\mathcal{L}_-}(\mathcal{R}o^*F, G)
\simeq
\operatorname{Ext}^*_{\mathrm{D}^b(X)}(F, G).
\]
\end{lemma}
\begin{proof}
The argument is standard by global adjunction \cite[Lemma 3.9]{Hao2025b}, using the bifiltration $\mathcal{R}^{(s,t)}o^*(F)$ constructed above.
\end{proof}
\begin{proposition}
The functor $\mathcal{R}o^*$ induces an equivalence of triangulated categories
\[
\mathrm{D}^b(X) \simeq \mathrm{D}^b(Y)/\mathcal{L}_-.
\]
\end{proposition}

\begin{proof}
Well-definedness and the triangulated structure are standard, and we have already shown that $\mathcal{R}o^*$ is fully faithful by Lemma \ref{coasefullcomp}. It remains to prove essential surjectivity. For any $F \in \mathrm{D}^b(Y)$, consider the natural adjunction triangle
\[
\mathcal{R}o^*o_*F \longrightarrow F \longrightarrow \mathcal{Q},
\]
where $\mathcal{Q}$ satisfies $o_*\mathcal{Q}=0$. Hence $\tau \in \mathcal{L}_-$, and therefore $F \simeq \mathcal{R}o^*o_*F$ in $\mathrm{D}^b(Y)/\mathcal{L}_-$.
\end{proof}

\subsection{Weighted blow-up}\label{Weightedblowup}
We previously considered Kawamata blow-ups along certain special global centers $C$ carrying mild well-formed singularities, since globally one needs to define a constructible Rees-type derived pullback. It is natural, however, to try to extend the above results to arbitrary centers $C$. If both $C$ and $X$ are smooth, the classical blow-up formula suggests that this should be possible. However, for a general weighted blow-up, even in the case of a smooth variety blown up along a smooth center, one must take into account different embeddings. For instance, the weighted blow-up of $k[[x_1,x_2,x_3]]$ with weights $(0,1,1)$ and the weighted blow-up of $k[[x_1,x_2,x_3,x_4]]/(x_4+x_3x_2)$ with weights $(0,1,1,3)$ produce different outcomes.\\

We first consider a basic case. Let $X$ be a smooth variety. Locally, we embed it as a complete intersection into $\mathrm{Spec}\,R:=k[[x_1,\ldots,x_n]]$, and consider the stacky weighted blow-up of $\mathrm{Spec}\,R$ with weights $(a_1,\ldots,a_s)$ along a smooth center. Let $\mathcal{Z}$ be the strict transform of $X$ in this ambient weighted blow-up. We have the diagram
\[
\begin{tikzcd}[column sep=3.2em, row sep=2.2em]
E \arrow[r,"\iota_E"] \arrow[d,"f_E"']
& E_{\mathcal{Z}} \arrow[r,"\iota"] \arrow[d,"f"']
& \mathcal{Z} \arrow[d,"\pi"] \\
p \arrow[r,"\kappa_p"']
& C \arrow[r,"\kappa"']
& X
\end{tikzcd}
\]
As a generalization of the classical case, $E_{\mathcal{Z}}$ is no longer a weighted projective stack in general, but rather a hypersurface fibration. We assume it is a Fano fibration, i.e.
\[
e := \sum_{i=1}^{s} a_i - val(I) > 0.
\]
Moreover, we assume a semiorthogonal decomposition
\[
\mathrm{D}^b(E_{\mathcal{Z}})=
\left\langle
\mathcal{A}_{-1},\,
f^*\mathrm{D}^b(C),\,
f^*\mathrm{D}^b(C)\otimes \mathcal{O}_{E_{\mathcal{Z}}}(1),\,
\ldots,\,
f^*\mathrm{D}^b(C)\otimes \mathcal{O}_{E_{\mathcal{Z}}}(e-1)
\right\rangle.
\]
and we define a full saturated subcategory in $\mathrm{D}^b(\mathcal{Z})$ by $\mathcal{K}_{-1}:=\langle \iota_{*}\mathcal{A}_{-1}\rangle$. Fortunately, by the regularity of $X$, we have a well-defined derived pullback
\[
\pi^*:\mathrm{D}^b(X)\longrightarrow \mathrm{D}^b(\mathcal{Z}),
\]
together with an adjoint triple
\[
\pi^* \dashv \pi_* \dashv \pi^!.
\]
One may check that it agrees with the previously defined $\mathcal{R}\pi^*$ in \cite{Hao2025b}.
\begin{proposition}\label{smoothwblsod}
In any case, we have the following classical decomposition:
\[
\mathrm{D}^b(\mathcal{Z})
=
\left\langle
\operatorname{Im}\big(\iota_*\big(f^*\mathrm{D}^b(C)\otimes \mathcal{O}_{E_{\mathcal Z}}(-e+1)\big)\big),\,
\ldots,\,
\operatorname{Im}\big(\iota_*\big(f^*\mathrm{D}^b(C)\otimes \mathcal{O}_{E_{\mathcal Z}}(-1)\big)\big),\,
\mathcal{K}_{-1},\,
\operatorname{Im}\,\mathcal{R}\pi^*
\right\rangle.
\]
\end{proposition}

\begin{corollary}
In this case, $\mathcal{K}_{-1}$ is admissible.
\end{corollary}

Then we assume $X$ is a local well-formed variety, with singular center $C$ given first by a closed point $p$, and that the cyclic group is trivial, so that $\mathcal{X}=X$ has a local complete intersection singularity at $p$:
\[
\widehat{\mathcal O}_{X,p}
\simeq
\mathrm{k}(p)[[x_1,\ldots,x_n]]/I.
\]
Choose moreover a smooth center, viewed as the center of a weighted blow-up (or equivalently as the center of a valuation), passing through $p$, and denote this center again by $C$. In the local situation, after a coordinate change, one may assume that $Z(x_1,\ldots,x_s)$ defines $C$. Consider the valuation \footnote{In fact, more generally, consider any algebraic divisorial contraction to this center $C$ over an algebraically closed field of characteristic zero. If the discrepancy satisfies $a(E,X)>-1$ and $-E$ is relatively ample, then such a divisorial valuation can always be realized by restriction a valuation of the form above, after enlarging the coordinate system, i.e. \cite{CCC}.
} on $R:=k[[x_1,\ldots,x_n]]$ given by
\[
val(x_i)=a_i>0 \quad (1\le i\le s),
\qquad
val(x_i)=0 \quad (s<i\le n).
\]
This defines a filtration of ideals
\[
P_i := \{\, f \in R \mid val(f)\ge i \,\}.
\]
We then consider the stacky Proj of the $\mathbb{N}$-graded $R$-algebra $\bigoplus_{i\ge 0} P_i$, namely
\[
\mathcal{P}\mathrm{roj}_R\!\left(\bigoplus_{i\ge 0} P_i\right)
:=
\left[\operatorname{Spec}_R\!\left(\bigoplus_{i\ge 0} P_i\right)\setminus Z \,\big/\, \mathbb{G}_m\right],
\]
where $Z$ denotes the irrelevant locus. We define this construction to be the weighted blow-up along the valuation, and denote it by
\[
w\mathrm{Bl}^{(a_i)}_C \operatorname{Spec} R.
\]
For a general complete intersection $I$, we consider the embedding $\mathcal{X}\hookrightarrow \operatorname{Spec} R$. The birational transformation defined above induces the strict transform of $\mathcal{X}$, yielding the stack $\mathcal{Z}:=w\mathrm{Bl}^{(a_i)}_C \mathcal{X}$. We have a commutative diagram
\[
\begin{tikzcd}[column sep=3.2em, row sep=2.2em]
E \arrow[r,"\iota_E"] \arrow[d,"f_E"']
& E_{\mathcal{Z}} \arrow[r,"\iota"] \arrow[d,"f"']
& \mathcal{Z} \arrow[d,"\pi"] \\
p \arrow[r,"\kappa_p"']
& C \arrow[r,"\kappa"]
& \mathcal{X}
\end{tikzcd}
\]
and impose the following extra conditions:
\begin{enumerate}
\item $f$ is a flat (\ref{flatcond}) morphism.

\item $\mathcal{X}$ satisfies the measured condition  (\ref{cond:measured}): there exists a measured basis $\{f_i\}$ of a regular sequence generating $I$, such that the leading terms $\{L(f_i)\}$ also form a regular sequence generating $L(I)$.

\item The exceptional divisor is a Fano fibration.
\end{enumerate}
We note that we do not assume the coarse moduli $w\mathrm{Bl}^{(a_i)}_C X$ is well-formed. However, if $w\mathrm{Bl}^{(a_i)}_C X$ is well-formed and exceptional divisor is prime, then $\mathcal{Z}$ is the canonical stack of $w\mathrm{Bl}^{(a_i)}_C X$. We explain the flatness condition \label{flatcond}. Although more sophisticated definitions exist, we only use the following concrete interpretation: Locally, we may realize $E_{\mathcal{Z}}$ as a complete intersection of a partial moduli stack of the quotient stack $\left[\mathds{P}^s_C / (\mu_{a_1} \times \cdots \times \mu_{a_s})\right]$, i.e. we can take an iterated root construction which produces a substack
\[
E_{\mathcal{Z}}^+ := \left[ Z(J) / (\mu_{a_1} \times \cdots \times \mu_{a_s}) \right]
\subset
\left[ \mathds{P}^s_C / (\mu_{a_1} \times \cdots \times \mu_{a_s}) \right],
\]
together with the roots morphism
\[
\alpha: E_{\mathcal{Z}}^+ \longrightarrow E_{\mathcal{Z}}.
\]
We say that $f$ is flat if $Z(J)$ is flat over $C$. In this sense, we obtain the following flat base change :
\begin{lemma}\label{flatbasechange}
Consider the cartesian diagram
\[
\begin{tikzcd}[column sep=3.2em, row sep=2.2em]
E \ar[r, "\iota_E"] \ar[d, "f_E"'] & E_{\mathcal{Z}} \ar[d, "f"] \\
p \ar[r, "\kappa_p"'] & C
\end{tikzcd}
\]
with $f$ flat, there exists a canonical isomorphism
\[
\kappa_{p}^{*}\, f_* F \;\simeq\; {f_E}_*\,\iota_{E}^* F
\]
for any $F \in \mathrm{D}^b(E_{\mathcal{Z}})$.
\end{lemma}
\begin{proof}
Note that for any $F \in \mathrm{D}^b(E_{\mathcal Z})$ we have $\alpha_*\alpha^*F \simeq F$, hence the desired base change follows from  Flat base change of the cartesian diagram below
\[
\begin{tikzcd}[column sep=3.2em, row sep=2.2em]
E':=\left[ Z(J)|_p \big/ \left(\mu_{a_1}\times\cdots\times\mu_{a_s}\right) \right]
\arrow[r]
\arrow[d]
&
E'_{\mathcal Z}:=\left[ Z(J) \big/ \left(\mu_{a_1}\times\cdots\times\mu_{a_s}\right) \right]
\arrow[d]
\\
p \arrow[r] & C
\end{tikzcd}
\]
We note that $K_{E'_{\mathcal{Z}}/C}$ is locally free and $K_{E'_{\mathcal{Z}}/C}|_{E'}\simeq K_{E'/p}$. Hence this dual form base change follows from $G$-equivariant flat base change by \cite[Lemma~1.3]{CCC}, for instance by passing to the equivariant \v{C}ech complex computing the derived pushforward.
\end{proof}
We can naturally consider the following decomposition:
\[
\mathrm{D}^b(E_{\mathcal{Z}})=
\left\langle
\mathcal{A}_{-1},
f^*\mathrm{D}^b(C),\,
f^*\mathrm{D}^b(C)\otimes \mathcal{O}_{E_{\mathcal{Z}}}(1),\,
\ldots,\,
f^*\mathrm{D}^b(C)\otimes \mathcal{O}_{E_{\mathcal{Z}}}(e-1)
\right\rangle
\]
where we set $e := \sum_{i=1}^{s} a_i - val(I)$ as the Fano index.

\begin{lemma}
We have a semiorthogonal decomposition
\[
\mathrm{D}^b(E)=
\left\langle
\mathcal{B}_{-1},\,
\mathcal{O}_{E},\,
\mathcal{O}_{E}(1),\,
\ldots,\,
\mathcal{O}_{E}(e-1)
\right\rangle,
\]
where $E$ is a complete intersection in a weighted projective stack and $\mathcal{O}_{E}(i)$ are exceptional objects.
\end{lemma}
\begin{proof}
For $0\le i<e$, we have
\[
f_{E*}\mathcal{O}_E(-i)
=
f_{E*}\iota_E^*\mathcal{O}_{E_{\mathcal Z}}(-i)
\simeq
\kappa_p^*\,f_*\mathcal{O}_{E_{\mathcal Z}}(-i)
\]
by Lemma \ref{flatbasechange}. If $i=0$, then $f_*\mathcal{O}_{E_{\mathcal Z}}\simeq \mathcal{O}_C$, hence
\[
f_{E*}\mathcal{O}_E \simeq \kappa_p^*\mathcal{O}_C \simeq \mathrm{k}(p).
\]
If $0<i<e$, then $f_*\mathcal{O}_{E_{\mathcal Z}}(-i)=0$, hence $f_{E*}\mathcal{O}_E(-i)=0$.
\end{proof}
\begin{lemma}
If $G \in \mathcal{A}_{-1}$, then $\iota_E^*G \in \mathcal{B}_{-1}$.
\end{lemma}
\begin{proof}
If $G \in \mathcal{A}_{-1}$, then $f_*G(-i)=0$ for all $0\leq i<e$. By Lemma \ref{flatbasechange},
\[
\kappa_p^*f_*G(i)\simeq f_{E*}\iota_E^*G(i).
\]
Hence $f_{E*}\iota_E^*G(-i)=0$ for all $0\leq i<e$, which implies $\iota_E^*G \in \mathcal{B}_{-1}$.
\end{proof}
\begin{lemma}
If $F \in \mathcal{B}_{-1}$, then $\iota_{E*}F \in \mathcal{A}_{-1}$.
\end{lemma}
\begin{proof}
It suffices to show $f_*\iota_{E*}F(-i)=\kappa_{p*}f_{E*}F(-i)=0$ for all $0\leq i<e$, which follows immediately.
\end{proof}
Similarly we define a full saturated subcategory $\mathcal{K}_{-1}:=\langle \iota_{*}\mathcal{A}_{-1}\rangle \subset \mathrm{D}^b(\mathcal{Z})$.
\begin{lemma}\label{wbuporth}
There is a semiorthogonal decomposition
\[
\mathcal{T}:=
\left\langle
\iota_{*}\mathcal{O}_{E_{\mathcal Z}}(-e+1),\,
\ldots,\,
\iota_{*}\mathcal{O}_{E_{\mathcal Z}}(-1),\,
\mathcal{K}_{-1}
\right\rangle
\subset \mathrm{D}^b(\mathcal{Z}).
\]
\end{lemma}
\begin{proof} Note that $\mathcal{O}_{\mathcal{Z}}(E_{\mathcal{Z}})|_{E_{\mathcal{Z}}}\simeq \mathcal{O}_{E_{\mathcal{Z}}}(-1)$, and we have the semiorthogonal decompositions
\[
\mathrm{D}^b(E_{\mathcal{Z}})=
\left\langle
\mathcal{O}_{E_{\mathcal{Z}}}(-(e-1)),\,
\ldots,\,
\mathcal{O}_{E_{\mathcal{Z}}}(-1),\,
\mathcal{A}_{-1},\,
\mathcal{O}_{E_{\mathcal{Z}}}
\right\rangle,
\]
and
\[
\mathrm{D}^b(E_{\mathcal{Z}})=
\left\langle
\mathcal{O}_{E_{\mathcal{Z}}}(-(e-1)),\,
\ldots,\,
\mathcal{O}_{E_{\mathcal{Z}}}(-1),\,
\mathcal{O}_{E_{\mathcal{Z}}},\,
\mathcal{A}_{0}
\right\rangle.
\]
The claim follows from long exact sequence induced by excess distinguished triangle for sheaves \ref{excdis}.
\end{proof}
Similarly, locally $\pi^*\kappa_*\mathcal{O}_{C}$ admits a filtration
\[
\begin{tikzcd}[column sep=0.5em]
\pi^*\mathcal{O}_{C}\arrow{rr} && F^s([\mathcal{O}_{C}]^+)\arrow{dl}\arrow{rr}&& F^{s-1}\arrow{dl}&&  \\
&\tau_s\arrow[ul,dashed,"\Delta"]&& \iota_{*}\mathcal{D}_{s}[1]\arrow[ul,dashed,"\Delta"]&&
\end{tikzcd}
\quad\cdots\quad
\begin{tikzcd}[column sep=0.5em]
 & F^1\arrow{rr}&& F^0\arrow{rr}\arrow{dl}&& 0\arrow{dl} \\
 && \iota_{*}\mathcal{D}_{1}[1]\arrow[ul,dashed,"\Delta"] && \iota_{*}\mathcal{D}_{0}[1]\arrow[ul,dashed,"\Delta"]
\end{tikzcd}
\]
where $\mathcal{D}_{i}:=\mathds{L}_{\langle \mathcal{O}_{E_{\mathcal Z}},\ldots,\mathcal{O}_{E_{\mathcal Z}}(i-1)\rangle}\mathcal{O}_{E_{\mathcal Z}}(i)$ for $0\le i\le e-1$. Since we work formal locally, we may still apply the dévissage argument, so that $\pi^*\kappa_{p*}\mathrm{k}(p)$ admits a filtration:
\[
\begin{tikzcd}[column sep=0.5em]
\pi^*\kappa_{p*}\mathrm{k}(p) \arrow{rr} && F^{s}([\mathrm{k}(p)]^+) \arrow{dl}\arrow{rr} && F'^{\,s-1} \arrow{dl} && \\
& \tau'_{s} \arrow[ul,dashed,"\Delta"] && \iota_{E*}\iota_E^*\mathcal{D}_{s}[1] \arrow[ul,dashed,"\Delta"] &&
\end{tikzcd}
\quad\cdots\quad
\begin{tikzcd}[column sep=0.5em]
& F'^{\,1} \arrow{rr} && F'^{\,0} \arrow{rr}\arrow{dl} && 0 \arrow{dl} \\
&& \iota_{E*}\iota_E^*\mathcal{D}_{1}[1] \arrow[ul,dashed,"\Delta"] && \iota_{E*}\iota_E^*\mathcal{D}_{0}[1] \arrow[ul,dashed,"\Delta"]
\end{tikzcd}
\]
for any $s\geq e$.
\begin{lemma}
For any $0\le i<e$, we have $\iota_E^*\mathcal{D}_i \simeq \mathcal{D}_{p,i}$, where
\[
\mathcal{D}_{p,i}:=\mathds{L}_{\langle \mathcal{O}_{E},\ldots,\mathcal{O}_{E}(i-1)\rangle}\mathcal{O}_{E}(i)
\quad \text{for } 0\le i<e.
\]
\end{lemma}
\begin{proof}
Recall the definition of the left mutation along the component: $\mathcal{D}_1$ fits into a distinguished triangle
\[
f^*f_*\mathcal{O}_{E_{\mathcal Z}}(1)\longrightarrow \mathcal{O}_{E_{\mathcal Z}}(1)\longrightarrow \mathcal{D}_1 \in  \Delta,
\]
restricting to $E$, we obtain a distinguished triangle
\[
\iota_E^*f^*f_*\mathcal{O}_{E_{\mathcal Z}}(1)\longrightarrow \iota_E^*\mathcal{O}_{E_{\mathcal Z}}(1)\longrightarrow \iota_E^*\mathcal{D}_1 \in  \Delta.
\]
For the first term, by Lemma \ref{flatbasechange},
\[
\iota_E^*f^*f_*\mathcal{O}_{E_{\mathcal Z}}(1)
\simeq
f_E^*\kappa_p^*f_*\mathcal{O}_{E_{\mathcal Z}}(1)
\simeq
f_E^*f_{E*}\iota_E^*\mathcal{O}_{E_{\mathcal Z}}(1)
\simeq
f_E^*f_{E*}\mathcal{O}_E(1).
\]
while second term is simply $\mathcal{O}_E(1)$ and one checks that the induced morphism is the natural adjunction map. Hence $\iota_E^*\mathcal{D}_1 \simeq \mathcal{D}_{p,1}$. The general case is analogous.
\end{proof}
So we obtain a natural filtration
\[
\begin{tikzcd}[column sep=0.5em]
\pi^*\kappa_{p*}\mathrm{k}(p) \arrow{rr} && F^{s}([\mathrm{k}(p)]^+) \arrow{dl}\arrow{rr} && F'^{\,s-1} \arrow{dl} && \\
& \tau'_{s} \arrow[ul,dashed,"\Delta"] && \iota_{E*}\mathcal{D}_{p,s}[1] \arrow[ul,dashed,"\Delta"] &&
\end{tikzcd}
\quad\cdots\quad
\begin{tikzcd}[column sep=0.5em]
& F'^{\,1} \arrow{rr} && F'^{\,0} \arrow{rr}\arrow{dl} && 0 \arrow{dl} \\
&& \iota_{E*}\mathcal{D}_{p,1}[1] \arrow[ul,dashed,"\Delta"] && \iota_{E*}\mathcal{D}_{p,0}[1] \arrow[ul,dashed,"\Delta"]
\end{tikzcd}
\]
and we know that for $i\ge e$, $\mathcal{D}_{p,i}\in \mathcal{B}_{-1}$. Then we define $\mathcal{H}_{-1}$ to be the full saturated subcategory of $\mathrm{D}^b(\mathcal{Z})$ generated by $\iota_{E*}\mathcal{B}_{-1}$, and set
\[
\mathcal{S}:=
\left\langle
\iota_{E*}\mathcal{O}_{E}(-e+1),\,
\ldots,\,
\iota_{E*}\mathcal{O}_{E}(-1),\,
\mathcal{H}_{-1}
\right\rangle
\subset \mathrm{D}^b(\mathcal{Z}).
\]
\begin{lemma}\label{kernalElemma}
For any $F \in \mathrm{D}^b(\mathcal{Z})$, if $\pi_*F=0$ and support on E, then $F \in \mathcal{S}$.
\end{lemma}

\begin{proof}
We consider $F \in \mathrm{D}^b(\mathcal{Z})$ with $\pi_*F=0$. We may replace $F$ by its right mutation through the component
\[
\left\langle
\iota_{*}\mathcal{O}_{E_{\mathcal Z}}(-e+1),\,
\ldots,\,
\iota_{*}\mathcal{O}_{E_{\mathcal Z}}(-1)
\right\rangle,
\]
and assume that $F$ is left orthogonal to this component. Since $F \in \mathrm{D}^b(\mathcal{Z})$ is a bounded complex of coherent sheaves, is orthogonal to the above components, and is supported on $E$, we may regard its image in the quotient category $\mathrm{D}^b(\mathcal{Z})/\mathcal{H}_{-1}$, can obtained as an iterated extension of objects of the form $F^{s}([\mathrm{k}(p)]^+)$ for $s \ge e-1$. On the other hand, using \cite[Lemma 3.7]{Hao2025b}, we have the following local adjunction formula
\[
\operatorname{Ext}^*_{\mathrm{D}^b(\mathcal{Z})/\mathcal{H}_{-1}}\!\big(F^{s}([\mathrm{k}(p)]^+),\,F\big)
\simeq
\operatorname{Ext}^*_{\mathrm{D}^b(X)}\!\big(\mathrm{k}(p),\,\pi_*F\big),
\]
for $s \ge e-1$. By the long exact sequences of $\operatorname{Ext}$ associated to these extensions, we obtain
\[
\operatorname{Ext}^*_{\mathrm{D}^b(\mathcal{Z})/\mathcal{H}_{-1}}(F,F)=0,
\]
hence $F \in \mathcal{H}_{-1}$. The proof is complete.
\end{proof}
\begin{corollary}
The following full saturated subcategories of $\mathrm{D}^b(\mathcal{Z})$ are equivalent:
\[
\mathcal{H}_{-1}:=
\left\langle
\iota_{E*}\mathcal{B}_{-1}
\right\rangle
\quad\text{and}\quad
\left\langle
\iota_{*}\mathcal{A}_{-1}
\;\middle|\;
\operatorname{Supp}(-) \subset E
\right\rangle .
\]
\end{corollary}
\begin{proof}
It is clear that $\langle \iota_{*}\mathcal{A}_{-1}\rangle \subset \ker \pi_*$, and it is left orthogonal to
\[
\left\langle
\iota_{*}\mathcal{O}_{E_{\mathcal Z}}(-e+1),\,
\ldots,\,
\iota_{*}\mathcal{O}_{E_{\mathcal Z}}(-1)
\right\rangle.
\]
By Lemma \ref{wbuporth} and \ref{kernalElemma}, $\left\langle
\iota_{*}\mathcal{A}_{-1}
\;\middle|\;
\operatorname{Supp}(-) \subset E
\right\rangle \subset\mathcal{H}_{-1}\subset \left\langle
\iota_{*}\mathcal{A}_{-1}
\;\middle|\;
\operatorname{Supp}(-) \subset E
\right\rangle$.
\end{proof}
At the same time, for any maximal Cohen--Macaulay module $M$ on $\mathcal{X}$, we construct an object $\mathcal{R}(M)\in \mathrm{D}^b(\mathcal{Z})$ with a filtration obtained by iterating the above procedure:
\[
\begin{tikzcd}[column sep=0.5em]
 & \mathcal{F}^s(\mathcal{R}(M)) \arrow{rr}&& \mathcal{F}^{s-1}\arrow{dl}&& \\
&&  \iota_{*}\operatorname{gr}^{s}(M)\arrow[ul,dashed,"\Delta"]&&
\end{tikzcd}
\quad\cdots\quad
\begin{tikzcd}[column sep=0.5em]
 &   \mathcal{F}^2\arrow{rr}&&  \mathcal{F}^1\arrow{rr}\arrow{dl}&& \mathcal{R}(M)\arrow{dl} \\
 && \iota_{*}\operatorname{gr}^2(M)\arrow[ul,dashed,"\Delta"]  &&\iota_{*}\operatorname{gr}^1(M)\arrow[ul,dashed,"\Delta"]
\end{tikzcd}
\]
\begin{lemma}
For any $i$, we have $ \iota_{*}\operatorname{gr}^i(M)\in \mathcal{H}_{-1}$.
\end{lemma}

\begin{proof}
Assume $X=\mathrm{Spec}\,S$. The above filtration is stable under arbitrary localization of $S$. Suppose $\mathrm{Supp}(\operatorname{gr}^i(M))\not\subset E$. Then there exists a non-maximal prime ideal $\mathfrak{p}\subset S$ containing the prime ideal $\mathfrak{q}$ defining the center $C$, such that $(\operatorname{gr}^i(M))_{\mathfrak{p}}\neq 0$. Set $X_{(\mathfrak{p})}:=\mathrm{Spec}\,S_{\mathfrak{p}}$ and $C_{(\mathfrak{p})}:=\mathrm{Spec}\,(S/\mathfrak{q})_{\mathfrak{p}}$. Then both are regular, and we have a localized semiorthogonal decomposition
\[
\mathrm{D}^b(E_{\mathcal{Z},(\mathfrak{p})})=
\left\langle
\mathcal{A}_{-1}(\mathfrak{p}),\,
f^*\mathrm{D}^b(C_{(\mathfrak{p})}),\,
\ldots,\,
f^*\mathrm{D}^b(C_{(\mathfrak{p})})\otimes \mathcal{O}_{E_{\mathcal{Z}}}(e-1)
\right\rangle,
\]
since the weighted blow-up is locally a weighted projective space fibration with the same Fano index. Comparing $\operatorname{gr}^i(M)$ under localization with Proposition \ref{smoothwblsod}, we obtain a contradiction. Hence $\mathrm{Supp}(\operatorname{gr}^i(M))\subset E$. Since we already know that $\iota_{*}\operatorname{gr}^i(M)$ is orthogonal to
\[
\left\langle
\iota_{*}\mathcal{O}_{E_{\mathcal Z}}(-e+1),\,
\ldots,\,
\iota_{*}\mathcal{O}_{E_{\mathcal Z}}(-1)
\right\rangle,
\]
and moreover $\pi_*\iota_{*}\operatorname{gr}^i(M)=0$. By Lemma \ref{kernalElemma}, $\iota_{*}\operatorname{gr}^i(M)\in \mathcal{H}_{-1}$.
\end{proof}

In particular, this gives an approximation of $\pi^*M$:
\[
\pi^*M \simeq \varprojlim_{s\to\infty} \mathcal{F}^s(\mathcal{R}(M)).
\]
By the previous corollary, $\iota_{*}\operatorname{gr}^{s}(M)\in \mathcal{H}_{-1}$ for all $s$. Moreover, we define a functor
\[
\mathcal{R}\pi^*:\mathrm{D}^b(\mathcal{X})\longrightarrow \mathrm{D}^b(\mathcal{Z})/\mathcal{H}_{-1},
\]
and we have:
\begin{lemma}
If $e>0$, then
\[
\mathcal{R}\pi^*:\mathrm{D}^b(\mathcal{X})\longrightarrow \mathrm{D}^b(\mathcal{Z})/\mathcal{H}_{-1}
\]
is a well-defined admissible triangulated categories functor with adjoints $\pi_!\dashv \mathcal{R}\pi^* \dashv \pi_*$.
\end{lemma}
\begin{proof}
We should prove the global adjunction formula:
\[
\operatorname{Ext}^*_{\mathrm{D}^b(\mathcal{Z})/\mathcal{H}_{-1}}\!\big(\mathcal{R}\pi^*F,\,G\big)
\simeq
\operatorname{Ext}^*_{\mathrm{D}^b(\mathcal{X})}\!\big(F,\,\pi_*G\big),
\]
it follows from classical global adjunction e.g.\cite[Lemma 3.9]{Hao2025b}.
\end{proof}
\begin{corollary}
$\mathcal{R}\pi^*$ is a fully faithful embedding.
\end{corollary}

\begin{proof}
This follows immediately from $\pi_* \mathcal{R}\pi^* \simeq \mathrm{id}$.
\end{proof}
We now pass to the global situation. Let $\mathcal{X}$ be a well-formed stack with a well-formed isolated singularity $p$, and let $C$ be a smooth center passing through $p$. Consider a global weighted blow-up along $C$ with weights $(a_1,\ldots,a_s)$ such that, in the formal neighborhood of $p$, all the local assumptions above are satisfied. Denote
\[
\mathcal{Z}:= w\mathrm{Bl}^{(a_i)}_C \mathcal{X}.
\]
We have the diagram
\[
\begin{tikzcd}[column sep=3.2em, row sep=2.2em]
E \arrow[r,"\iota_E"] \arrow[d,"f_E"']
& E_{\mathcal{Z}} \arrow[r,"\iota"] \arrow[d,"f"']
& \mathcal{Z} \arrow[d,"\pi"] \\
p \arrow[r,"\kappa_p"']
& C \arrow[r,"\kappa"']
& \mathcal{X}
\end{tikzcd}
\]
\begin{lemma}\label{wblspanningclass}
Assume $e>0$. Consider the set
\[
\Omega :=
\bigcup_{\substack{q\in C\\ \text{closed}}}
\left\{
\mathcal{O}_{E_q}(-e+1),\ldots,\mathcal{O}_{E_q}(-1),\,\mathcal{B}_{E_q,-1},\,\mathcal{O}_{E_q}
\right\}
\;\cup\;
\left\{
\mathrm{k}_x\otimes \chi \mid x \notin E_{\mathcal Z}\ \text{closed},\ \chi
\right\}.
\]
where $E_q$ denotes the exceptional divisor over $q$. Then $\Omega$ forms a spanning class in $\mathrm{D}^b(\mathcal{Z})/\mathcal{H}_{-1}$.
\end{lemma}
\begin{proof}
Let $F\in \mathrm{D}^b(\mathcal{Z})$ be such that
\[
\operatorname{Ext}^*_{\mathrm{D}^b(\mathcal{Z})/\mathcal{H}_{-1}}(F,a)=0
\quad \text{for all } a\in \Omega.
\]
Consider the set
\[
\bigcup_{\substack{q\in C\setminus\{p\}\\ \text{closed}}}
\left\{
\mathcal{O}_{E_q}(-e+1),\ldots,\mathcal{B}_{E_q,-1},\,\mathcal{O}_{E_q}
\right\}
\;\cup\;
\left\{
\mathrm{k}_x\otimes \chi \mid x \notin E_{\mathcal Z} \text{ closed},\ \chi
\right\},
\]
this forms a stacky spanning class away from $E$, hence $F$ is supported on $E$. Therefore $F$ is generated by iterated extensions of $\mathcal{O}_{E}(-e+1),\ldots,\mathcal{B}_{E_q,-1},\,\mathcal{O}_{E}$, and by the assumption we obtain
\[
\operatorname{Ext}^*_{\mathrm{D}^b(\mathcal{Z})/\mathcal{H}_{-1}}(F,F)=0,
\]
so $F\in \mathcal{H}_{-1}$. The converse direction is analogous.
\end{proof}
\begin{proposition}
The stacky weighted blow-up admits a semiorthogonal decomposition
\[
\mathrm{D}^b(\mathcal{Z})/\mathcal{H}_{-1}
=
\left\langle
\operatorname{Im}\big(\iota_*\big(f^*\mathrm{D}^b(C)\otimes \mathcal{O}_{E_{\mathcal Z}}(-e+1)\big)\big),\,
\ldots,\,
\operatorname{Im}\big(\iota_*\big(f^*\mathrm{D}^b(C)\otimes \mathcal{O}_{E_{\mathcal Z}}(-1)\big)\big),\,
\mathcal{K}_{-1},\,
\operatorname{Im}\,\mathcal{R}\pi^*
\right\rangle,
\]
where all functors are fully faithful and $\mathcal{H}_{-1}:=\langle \iota_{E*}\mathcal{B}_{-1}\rangle$, assuming $e>0$.
\end{proposition}
\begin{proof}
Since $\mathcal{R}\pi^*$ is a well-defined admissible functor, and the above gives an admissible semiorthogonal decomposition of $\mathcal{D}\subset \mathrm{D}^b(\mathcal{Z})/\mathcal{H}_{-1}$, it remains to prove generation. It is clear that
\[
\bigcup_{q\in C\,|\,\text{closed}}
\left\{
\mathcal{O}_{E_q}(-e+1),\ldots,\mathcal{B}_{E_q,-1},\,\mathcal{O}_{E_q}(-1)
\right\}
\;\cup\;
\left\{
\mathrm{k}_x\otimes \chi \mid x \notin E_{\mathcal Z}\ \text{closed},\ \chi
\right\}
\subset \mathcal{D}.
\]
It remains to show that $\mathcal{O}_{E_q}\in \mathcal{D}$ for any closed point $q\in C$. If $q=p$, we use the filtration
\[
\begin{tikzcd}[column sep=0.5em]
\pi^*\kappa_{p*}\mathrm{k}(p) \arrow{rr} && F^{s}([\mathrm{k}(p)]^+) \arrow{dl}\arrow{rr} && F'^{\,s-1} \arrow{dl} && \\
& \tau'_{s} \arrow[ul,dashed,"\Delta"] && \iota_{E*}\mathcal{D}_{p,s}[1] \arrow[ul,dashed,"\Delta"] &&
\end{tikzcd}
\quad\cdots\quad
\begin{tikzcd}[column sep=0.5em]
& F'^{\,1} \arrow{rr} && F'^{\,0} \arrow{rr}\arrow{dl} && 0 \arrow{dl} \\
&& \iota_{E*}\mathcal{D}_{p,1}[1] \arrow[ul,dashed,"\Delta"] && \iota_{E*}\mathcal{D}_{p,0}[1] \arrow[ul,dashed,"\Delta"]
\end{tikzcd}
\]
and note that $\iota_{E*}\mathcal{D}_{p,i}\in \mathcal{D}$ for all $i$. Hence $\mathcal{O}_E\in \mathcal{D}$. If $q\neq p$, we are in the regular weighted blow-up case where $\mathcal{R}\pi^*=\pi^*$, and we have a similar filtration
\[
\begin{tikzcd}[column sep=0.5em]
\pi^*\kappa_{q*}k(q)\arrow{rr} && F'^{\,e-2} \arrow{dl} && \\
& \iota_{E_q*}\mathcal{D}_{q,e-1}[1] \arrow[ul,dashed,"\Delta"] &&
\end{tikzcd}
\quad\cdots\quad
\begin{tikzcd}[column sep=0.5em]
& F'^{\,1} \arrow{rr} && F'^{\,0} \arrow{rr}\arrow{dl} && 0 \arrow{dl} \\
&& \iota_{E_q*}\mathcal{D}_{q,1}[1] \arrow[ul,dashed,"\Delta"] && \iota_{E_q*}\mathcal{D}_{q,0}[1] \arrow[ul,dashed,"\Delta"]
\end{tikzcd}
\]
and similarly $\iota_{E_q*}\mathcal{D}_{q,i}\in \mathcal{D}$ for all $i$, hence $\mathcal{O}_{E_q}\in \mathcal{D}$. This completes the proof.
\end{proof}

We also need a natural and frequently used generalization. Assume $\mathcal{X}$ is a well-formed stack with a well-formed singularity satisfying the above assumptions, and suppose moreover that there is a $\mu_r$-action which is free outside the isolated singular point $p$. In this setting, we consider a weighted blow-up with weights $(a_1,\ldots,a_s)$ along a smooth orbifold center $\mathcal{C}$, which is smooth away from $p$ and has a $\mu_r$-orbifold point at $p$. Passing to a $\mu_r$-cover, we reduce to the local situation discussed above. Over the orbifold point on $C$, the exceptional divisor is
\[
[E/\mu_r]\longrightarrow [p/\mu_r],
\]
where the $\mu_r$-action on $E$ is not necessarily trivial. Nevertheless, we have a semiorthogonal decomposition
\[
\mathrm{D}^b([E/\mu_r])=
\left\langle
\mathcal{A}_{-1},\,
\mathcal{O}_{[E/\mu_r]}\otimes \Big(\bigoplus_{i=0}^{r-1}\nu^i\Big),\,
\mathcal{O}_{[E/\mu_r]}(1)\otimes \Big(\bigoplus_{i=0}^{r-1}\nu^i\Big),\,
\ldots,\,
\mathcal{O}_{[E/\mu_r]}(e-1)\otimes \Big(\bigoplus_{i=0}^{r-1}\nu^i\Big)
\right\rangle,
\]
and each direct summand is strong exceptional. Completely analogously, we obtain:

\begin{proposition}
The stacky weighted blow-up admits a semiorthogonal decomposition
\[
\mathrm{D}^b(\mathcal{Z})/\mathcal{H}_{-1}
=
\left\langle
\operatorname{Im}\big(\iota_*\big(f^*\mathrm{D}^b(\mathcal{C})\otimes \mathcal{O}_{E_{\mathcal Z}}(-e+1)\big)\big),\,
\ldots,\,
\operatorname{Im}\big(\iota_*\big(f^*\mathrm{D}^b(\mathcal{C})\otimes \mathcal{O}_{E_{\mathcal Z}}(-1)\big)\big),\,
\mathcal{K}_{-1},\,
\operatorname{Im}\,\mathcal{R}\pi^*
\right\rangle,
\]
where all functors are fully faithful and $\mathcal{H}_{-1}:=\langle \iota_{[E/\mu_r]*}\mathcal{B}_{-1}\rangle$, assuming $e\geq0$.
\end{proposition}
and a similar corollary:
\begin{corollary}
$\mathcal{K}_{-1}$ is admissible in $\mathrm{D}^b(\mathcal{Z})/\mathcal{H}_{-1}$,  which contains many examples with $\mathcal{K}_{-1}=0$.
\end{corollary}

Finally, we should point out that the above situation includes the case of a well-formed isolated singularity subset in type $(a_{1},\dots,a_{n})/\mu_{r}$ together with a weighted blow-up along a valuation of weight $(b_{1},\dots,b_{n})$. If $a_{i}\equiv b_{i}\pmod r$ for all $i$, then this reduces to the classical Kawamata blow-up case. Otherwise, we usually assume $\gcd(b_{1},\dots,b_{n})=1$ in order to avoid additional root structures appearing on the exceptional divisor. Under these assumptions, if $w\mathrm{Bl}^{(b_i)}X$ is well-formed and its exceptional divisor $E$ is prime, then
\[
w\mathrm{Bl}^{(b_i)}\mathcal{X}
\simeq
\widetilde{\,w\mathrm{Bl}^{(b_i)}X\,},
\]
where the right-hand side denotes the canonical stack associated with the weighted blow-up.
\subsubsection{Coarse moduli}
In this section, we consider the stacky weighted blow-up introduced above and study its coarse moduli spaces under the assumption that $\mathcal{Z}$ is also well-formed or $Z$ is well-formed and $E_Z$ is prime:
\[
\begin{tikzcd}[column sep=2.5em, row sep=2.2em]
\mathcal{Z} := w\mathrm{Bl}^{(a_i)}_{\mathcal C}\mathcal{X}=\widetilde{w\mathrm{Bl}^{(a_i)}_{C} X } \arrow[r,"\epsilon"] \arrow[d,"\pi"']
& Z := w\mathrm{Bl}^{(a_i)}_{C} X \arrow[d,"o"] \\
\mathcal{X} \arrow[r,"\varepsilon"]
& X
\end{tikzcd}
\]
Similarly, since we have a semiorthogonal decomposition of stacky exceptional divisor $E_\mathcal{Z}:=[E/\mu_r]$:
\[
\mathrm{D}^b([E/\mu_r])=
\left\langle
\mathcal{O}_{[E/\mu_r]}(-(e-1))\otimes \Big(\bigoplus_{i=0}^{r-1}\nu^i\Big),\,
\ldots,\,
\mathcal{O}_{[E/\mu_r]}(-1)\otimes \Big(\bigoplus_{i=0}^{r-1}\nu^i\Big),\,
\mathcal{B}_{-1},\,
\mathcal{O}_{[E/\mu_r]}\otimes \Big(\bigoplus_{i=0}^{r-1}\nu^i\Big)
\right\rangle,
\]
by the properties of the coarse moduli space, this induces a semiorthogonal decomposition on $E^{-}$:
\[
\mathrm{D}^b(E^{-})=
\left\langle
\mathcal{C}_{-1},\,
\mathcal{O}_{E^{-}}
\right\rangle.
\]
Similarly, we define a full saturated subcategory $\mathcal{L}_-:=\langle \iota^-_{*}\mathcal{C}_{-1}\rangle \subset \mathrm{D}^b(Z)$.
\begin{lemma}
If $F \in \mathrm{D}^b(Z)$ satisfies $o_*F=0$, then $F \in \mathcal{L}_-$.
\end{lemma}
Thus we can define a triangulated functor
\[
\mathcal{R}o^*:= \epsilon_*\mathcal{R}\pi^*\widetilde{(-)}:
\mathrm{D}^b(X)\longrightarrow \mathrm{D}^b(Z)/\mathcal{L}_-,
\]
which is well-defined. Using the same double filtration argument as before, the proofs are identical to the previous section, and we obtain:
\begin{lemma}
We have $o_* \mathcal{R}o^* \simeq \mathrm{id}$.
\end{lemma}

\begin{lemma}
There is a natural isomorphism
\[
\operatorname{Ext}^*_{\mathrm{D}^b(Z)/\mathcal{L}_-}(\mathcal{R}o^*F, G)
\simeq
\operatorname{Ext}^*_{\mathrm{D}^b(X)}(F, o_*G).
\]
\end{lemma}

\begin{proposition}
The functor $\mathcal{R}o^*$ induces an equivalence of triangulated categories
\[
\mathrm{D}^b(X) \simeq \mathrm{D}^b(Z)/\mathcal{L}_-.
\]
\end{proposition}
The above proofs are completely parallel to that in Section~\ref{seckblcm}, and we omit it.
\subsection{Standard type flip}\label{Standardtypeflip}

For the construction in this section we refer to \cite{Reid}. In this and following sections we assume $\mathcal{X}$ be an arbitrary stack over a field $\mathrm{k}$ of characteristic zero, containing a weighted projective stack
\[
\mathcal{P}(a_i):=\mathcal{P}(a_1,\ldots,a_k),
\]
and assume that in a formal neighborhood of $\mathcal{P}(a_i)$, it can be realized as the stacky total space of a vector bundle
\[
\mathcal{E} := \mathcal{O}_{\mathcal{P}(a_i)}(-b_1)\oplus \cdots \oplus \mathcal{O}_{\mathcal{P}(a_i)}(-b_\ell)\oplus \mathcal{O}_{\mathcal{P}(a_i)}^{\oplus j},
\]
for positive integers $k,\ell$ and a non-negative integer $j$, where all $a_i,b_i$ are positive integers. Formal locally, we may assume
\[
\mathcal{X} \simeq \mathds{A}_{\mathcal{P}(a_i)}[[\mathcal{E}]].
\]a more explicit GIT description, consider coordinates
\[
x_1,\ldots,x_k,\quad y_1,\ldots,y_\ell,\quad z_1,\ldots,z_j,
\]
with a $\mathbb{G}_m$-action of weights
\[
(a_1,\ldots,a_k,\,-b_1,\ldots,-b_\ell,\,0,\ldots,0),
\]
viewed as a weight vector. Then $\mathcal{X}$ can be written as the formal quotient stack
\[
\mathcal{X} = \big[\operatorname{Spec} k[[x_i,y_i,z_i]] \setminus V(x_i)\, /\, \mathbb{G}_m \big].
\]
In particular, the linear coordinates $y_i$ and $z_i$ determine a prime center
\[
\mathcal{C} := V(y_i).
\]
More generally, this defines a valuation
\[
val : K(\mathcal{X}) \simeq K^{\mathbb{G}_m}\big([R/\mathbb{G}_m]\big)
\subset K\big([R/\mathbb{G}_m]\big)
\longrightarrow \mathbb{Z}.
\]
given by
\[
val(x_i)=0,\qquad val(z_i)=0,\qquad val(y_i)=b_i.
\]
for any $i$ in their range. In fact, in equivariant geometry one should consider valuations on $K([R/\mathbb{G}_m])$ rather than on the coarse moduli space. Nevertheless, the above valuation induces a natural family of (equivarient) filtration ideals $P_i$, defined by
\[
P_i :=
\{\, f \in [R/\mathbb{G}_m] \mid val(f)\ge i \,\}.
\]
and we consider the weighted stacky blow-up along this valuation, denoted by
\[
\mathcal{W}:= w\mathrm{Bl}_{\mathcal{P}(a_i)\times \operatorname{Spec} k[[z_1,\ldots,z_j]]}^{(b_i)} \mathcal{X},
\]
in toric coordinates, this corresponds to the matrix
\[
\begin{pmatrix}
a_1 & \cdots & a_k & -b_1 & \cdots & -b_\ell & 0 & \cdots & 0 & 0 \\
0   & \cdots & 0   & b_1  & \cdots & b_\ell  & 0 & \cdots & 0 & -1
\end{pmatrix}.
\]  Similarly, we consider
\[
\mathcal{X}^+ \simeq \mathds{A}_{\mathcal{P}(b_i)}[[\mathcal{E}^+]],
\]
where $\mathcal{P}(b_i):=\mathcal{P}(b_1,\ldots,b_\ell)$ is a weighted projective stack and
\[
\mathcal{E}^+ := \mathcal{O}_{\mathcal{P}(b_i)}(-a_1)\oplus \cdots \oplus \mathcal{O}_{\mathcal{P}(b_i)}(-a_k)\oplus \mathcal{O}_{\mathcal{P}(b_i)}^{\oplus j}.
\]
We define the weighted blow-up
\[
w\mathrm{Bl}^{(a_i)}_{\mathcal{P}(b_i)\times \operatorname{Spec} k[[z_1,\ldots,z_j]]} \mathcal{X}^+,
\]
and in toric coordinates, this corresponds to the matrix
\[
\begin{pmatrix}
b_1 & \cdots & b_\ell & -a_1 & \cdots & -a_k & 0 & \cdots & 0 & 0 \\
0   & \cdots & 0      & a_1  & \cdots & a_k  & 0 & \cdots & 0 & -1
\end{pmatrix}.
\]
One can then observe that via the linear transformation of representations, this weighted blow-up is isomorphic to $\mathcal{W}$. We obtain a flip described by VGIT, admitting a common resolution given by a weighted blow-up:
\[
\begin{tikzcd}[column sep=small, row sep=small]
& \mathcal{W} \arrow[dl,"\phi"'] \arrow[dr,"\varphi"] & \\
\mathcal{X} && \mathcal{X}^+
\end{tikzcd}
\]
We may assume $a:=\sum_{i=1}^{k} a_i$, $b:=\sum_{i=1}^{\ell} b_i$, and $e:=a-b\ge 0$. In particular, setting $C:=\operatorname{Spec} k[[z_1,\ldots,z_j]]$, we have the embedding diagrams:
\[
\begin{tikzcd}[column sep=3.5em, row sep=2.2em]
\mathcal{P}(b_i)\times C
\arrow[r, hookrightarrow, "\kappa^{(b)}"]
\arrow[d]
& \mathcal{X}^+
\arrow[d] \\
C \arrow[r, equal]
& C
\end{tikzcd}
\quad\quad\text{and}\quad
\begin{tikzcd}[column sep=4em, row sep=3em]
E \simeq \mathcal{P}(a_i)\times \mathcal{P}(b_i)\times C
\arrow[r, "\iota^{(b)}"]
\arrow[d, "\mathrm{pr}^{(b,c)}"']
& \mathcal{W}
\arrow[d,"\varphi"] \\
\mathcal{P}(b_i)\times C
\arrow[r, "\kappa^{(b)}"]
& \mathcal{X}^+
\end{tikzcd}
\]
similar for the other side.
\begin{proposition}\label{BOflopprop}
The classical Bondal--Orlov formula gives a semiorthogonal decomposition analogous to the standard flip:
\[
\mathrm{D}^b(\mathcal{X})=
\left\langle
\operatorname{Im}\big(\kappa^{(a)}_*\mathrm{D}^b(C)\boxtimes \mathcal{O}_{\mathcal{P}(a_i)}(-e)\big),\,
\ldots,\,
\operatorname{Im}\big(\kappa^{(a)}_*\mathrm{D}^b(C)\boxtimes \mathcal{O}_{\mathcal{P}(a_i)}(-1)\big),\,
\operatorname{Im}\,\phi_*\varphi^*
\right\rangle,
\]
where all functors are admissible and fully faithful.
\end{proposition}
Considering the derived categories of $\mathcal{P}(b_i)$, we have the standard semiorthogonal decomposition
\[
\mathrm{D}^b(\mathcal{P}(b_i))=\left\langle \mathcal{O}_{\mathcal{P}},\, \mathcal{O}_{\mathcal{P}}(1),\,\ldots,\,\mathcal{O}_{\mathcal{P}}(b-1)\right\rangle,
\]
and define its dual exceptional collection:
\[
\mathrm{D}^b(\mathcal{P}(b_i))=\left\langle \mathcal{D}^{(b)}_{b-1},\ldots,\mathcal{D}^{(b)}_1,\mathcal{D}^{(b)}_0 \right\rangle,
\]
where $\mathcal{D}^{(b)}_i := \mathbb{L}_{\langle \mathcal{O}_{\mathcal{P}},\ldots,\mathcal{O}_{\mathcal{P}}(i-1)\rangle}\mathcal{O}_{\mathcal{P}}(i)$ for any $i$ in the range, similarly for $\mathcal{P}(a_i)$. In the absence of ambiguity, we use $\mathcal{D}^{(a)}_i$ to denote $\mathcal{D}^{(a)}_i \boxtimes \mathcal{O}_C$ on $\mathcal{P}(a_i)\times C$, etc.
\begin{lemma}\label{flipullbackcom}
We have that $\varphi^{*}\kappa^{(b)}_{*}\mathcal{O}_{\mathcal{P}(b_i)\times C}$ admits a filtration:
\[
\begin{tikzcd}[column sep=0.5em]
\varphi^{*}\kappa^{(b)}_{*}\mathcal{O}_{\mathcal{P}(b_i)\times C}\arrow{rr}&& F^{a-2}\arrow{dl}&&  \\
& \iota_{*}\mathcal{D}^{(a)}_{a-1}\boxtimes\mathcal{O}_{\mathcal{P}(b_i)}(a-1)[1]\arrow[ul,dashed,"\Delta"]&&
\end{tikzcd}
\quad\cdots\]\[\cdots\quad
\begin{tikzcd}[column sep=0.5em]
& F^1\arrow{rr}&& F^0\arrow{rr}\arrow{dl}&& 0\arrow{dl} \\
&& \iota_{*}\mathcal{D}^{(a)}_{1}\boxtimes\mathcal{O}_{\mathcal{P}(b_i)}(1)[1]\arrow[ul,dashed,"\Delta"] && \iota_{*}\mathcal{D}^{(a)}_{0}\boxtimes\mathcal{O}_{\mathcal{P}(b_i)}[1]\arrow[ul,dashed,"\Delta"]
\end{tikzcd}
\]
where we use the convention $\mathcal{F}\boxtimes\mathcal{G}:=\operatorname{pr}^{(a)*}\mathcal{F}\otimes\operatorname{pr}^{(b)*}\mathcal{G}\otimes\operatorname{pr}^{(c)*}\mathcal{O}_C$.
\end{lemma}

\begin{proof}
This is equivalent to constructing an equivariant standard resolution of $\kappa^{(b)}_{*}\mathcal{O}_{\mathcal{P}(b_i)\times C}$ in $\mathcal{X}^{+}$ as a  generalization of \cite[Section~2]{Hao2025b}. Consider the defining equations $I$ in $R:=k[[x_i,y_i,z_i]]$, together with the induced $\mathbb{G}_{m}$-equivariant Koszul resolution
\[
0
\to
R\otimes \wedge^{k}V_{k}
\xrightarrow{F_{k-1}}
\cdots
\xrightarrow{}
R\otimes \wedge^{2}V_{k}
\xrightarrow{F_{1}:=\wedge^{2}F_{0}}
R\otimes V_{k}
\xrightarrow{F_{0}:=(x_{1},\ldots,x_{k})}
R
\to
R/I
\to
0 .
\]
The associated equivariant exact sequence is
\[
0
\to
R\otimes \chi^{\sum_{i=1}^{k}a_i}
\xrightarrow{F_{k-1}}
\cdots
\xrightarrow{}
\bigoplus_{i<j}
R\otimes \chi^{a_i+a_j}
\xrightarrow{F_{1}:=\wedge^{2}F_{0}}
\bigoplus_{i=1}^{k}
R\otimes \chi^{a_i}
\xrightarrow{F_{0}:=(x_{1},\ldots,x_{k})}
R
\to
R/I
\to
0 .
\]
where in our weighted blow-up convention, the variable $x_i$ carries weight $-a_i$.  We note that the Koszul resolution is simultaneously the standard resolution associated to the standard basis induced by the valuation $val$. Since the weighted blow-up weights satisfy $val(x_i)=a_i$, which coincide with the $\mathbb{G}_m$-weights of the corresponding representations, for every positive integer $t$ we obtain exact sequence:
\[
0
\to
P_{t-\sum_{i=1}^{k}a_i}\otimes \chi^{\sum_{i=1}^{k}a_i}
\xrightarrow{F_{k-1}}
\cdots
\to
\bigoplus_{i<j}
P_{t-a_i-a_j}\otimes \chi^{a_i+a_j}
\xrightarrow{F_{1}}
\bigoplus_{i=1}^{k}
P_{t-a_i}\otimes \chi^{a_i}
\xrightarrow{F_{0}}
P_t
\to
0 .
\]Using Serre’s theorem, we project the above exact sequence to $\mathcal{X}^{+}$ and obtain
\[
0
\to
\mathcal{P}_{t-\sum_{i=1}^{k}a_i}\otimes \mathcal{O}_{\mathcal{X}^{+}}\!\big(\sum_{i=1}^{k}a_i\big)
\xrightarrow{F_{k-1}}
\cdots
\to
\bigoplus_{i<j}
\mathcal{P}_{t-a_i-a_j}\otimes \mathcal{O}_{\mathcal{X}^{+}}(a_i+a_j)
\xrightarrow{F_{1}}
\bigoplus_{i=1}^{k}
\mathcal{P}_{t-a_i}\otimes \mathcal{O}_{\mathcal{X}^{+}}(a_i)
\xrightarrow{F_{0}}
\mathcal{P}_{t}
\to
0 ,
\]
where $\mathcal{O}_{\mathcal{X}^{+}}(-)$ is the line bundle induced by pulling back $\mathcal{O}_{\mathcal{P}(b_i)}(-)$. We sum over $t$, and by Serre's theorem again, we obtain the exact sequence in $\mathcal{W}$:
\[
0
\to
\mathcal{O}_{\mathcal{W}}(aE)\otimes \varphi^*\mathcal{O}_{\mathcal{X}^{+}}(a)
\xrightarrow{F_{k-1}}
\cdots
\to
\bigoplus_{i=1}^{k}
\mathcal{O}_{\mathcal{W}}(a_iE)\otimes \varphi^*\mathcal{O}_{\mathcal{X}^{+}}(a_i)
\xrightarrow{F_{0}}
\mathcal{O}_{\mathcal{W}}
\to
0 .
\]
Then compare  with the following complex which quasi-isomorphic to
$\varphi^{*}\kappa^{(b)}_{*}\mathcal{O}_{\mathcal{P}(b_i)\times C}$:
\[
0\to \mathcal{O}_{\mathcal{W}}\otimes \varphi^*\mathcal{O}_{\mathcal{X}^{+}}(a)\xrightarrow{F_{k-1}}\cdots\to \bigoplus_{i<j}\mathcal{O}_{\mathcal{W}}\otimes \varphi^*\mathcal{O}_{\mathcal{X}^{+}}(a_i+a_j)\to \bigoplus_{i=1}^{k}\mathcal{O}_{\mathcal{W}}\otimes \varphi^*\mathcal{O}_{\mathcal{X}^{+}}(a_i)\xrightarrow{F_{0}}\mathcal{O}_{\mathcal{W}}\to 0 .
\]
Using the short exact sequence $0 \to \mathcal{O}_{\mathcal{W}} \to \mathcal{O}_{\mathcal{W}}(E) \to \mathcal{O}_{E}(E) \to 0$, and the restriction formula
\[
\mathcal{O}_{\mathcal{W}}(tE)\big|_{E}
\simeq
\mathcal{O}_{\mathcal{P}(a_i)}(-t)\boxtimes \mathcal{O}_{\mathcal{P}(b_i)}(-t),
\quad \forall t\in \mathbb{Z},
\]
we obtain, by iterating successive shifts of the above exact sequence, that each $(s+1)$-fold shift produces a cone complex whose first term has two contributions: one coming from $\mathcal{O}_{\mathcal{W}}(E)$ whose restriction to $E$ gives $\mathcal{O}_{E}(E)$, and the other coming from the restriction of $\varphi^{*}\mathcal{O}_{\mathcal{X}^{+}}(s)$ to $E$, which is given by $\operatorname{pr}^{(b)*}\mathcal{O}_{\mathcal{P}(b_i)}(s)$. By a straightforward bookkeeping of these contributions, we leave the details to the reader and obtain that the $s$-th cone complex is isomorphic to
\[
\iota_{*}\mathcal{D}^{(a)}_{s}\boxtimes \mathcal{O}_{\mathcal{P}(b_i)}(s)[1].
\]
for any $0 \le s < a$. This completes the proof.
\end{proof}
\begin{corollary}
For any dual pre-tilting object $\mathcal{D}_s^{(b)}$ on $\mathcal{P}(b_i)\times C$, we have that
$\varphi^{*}\kappa^{(b)}_{*}\mathcal{D}_s^{(b)}(-(b-1))$ admits a filtration:
\[
\begin{tikzcd}[column sep=0em]
\varphi^{*}\kappa^{(b)}_{*}\mathcal{D}_s^{(b)}(-b)\arrow{rr}&& F^{a-2}\arrow{dl}&&  \\
& \iota_{*}\mathcal{D}^{(a)}_{a-1}\boxtimes\mathcal{D}_s^{(b)}(a-b)[1]\arrow[ul,dashed,"\Delta"]&&
\end{tikzcd}
\quad\cdots\]\[\quad\cdots
\begin{tikzcd}[column sep=0.5em]
& F^1\arrow{rr}&& F^0\arrow{rr}\arrow{dl}&& 0\arrow{dl} \\
&& \iota_{*}\mathcal{D}^{(a)}_{1}\boxtimes\mathcal{D}_s^{(b)}(1-(b-1))[1]\arrow[ul,dashed,"\Delta"] && \iota_{*}\mathcal{D}^{(a)}_{0}\boxtimes\mathcal{D}_s^{(b)}(-(b-1))[1]\arrow[ul,dashed,"\Delta"]
\end{tikzcd}
\]
where $0 \le s < b$ and we use the convention $\mathcal{F}\boxtimes\mathcal{G}:=\operatorname{pr}^{(a)*}\mathcal{F}\otimes\operatorname{pr}^{(b)*}\mathcal{G}\otimes\operatorname{pr}^{(c)*}\mathcal{O}_C$.
\end{corollary}
\begin{proof}
Since the weighted blow-up $\varphi$ is formal over $\mathcal{P}(b_i)\times C$ and admits a section in $\mathcal{X}^+$, by a dévissage argument we may move the tensor factor $\mathcal{D}^{(b)}(-(b-1))$ outside the pullback functor $\varphi^{*}$. Restricting to each factor and applying the previous filtration result, we obtain the stated conclusion.
\end{proof}
\begin{corollary}\label{dualcompuprop}
For any dual pre-tilting object $\mathcal{D}_s^{(b)}$ on $\mathcal{P}(b_i)\times C$, we have that
$\phi_{*}\varphi^{*}\kappa^{(b)}_{*}\mathcal{D}_s^{(b)}(-(b-1))$ admits a filtration:
\[
\begin{tikzcd}[column sep=0.5em]
\phi_{*}\varphi^{*}\kappa^{(b)}_{*}\mathcal{D}_s^{(b)}(-b)\arrow{rr}&& G^{a-2}\arrow{dl}&&  \\
& \kappa^{(a)}_{*}\mathcal{D}^{(a)}_{a-1}\otimes V_{a-1}[1]\arrow[ul,dashed,"\Delta"]&&
\end{tikzcd}
\cdots
\begin{tikzcd}[column sep=0.5em]
& G^1\arrow{rr}&& G^0\arrow{rr}\arrow{dl}&& 0\arrow{dl} \\
&& \kappa^{(a)}_{*}\mathcal{D}^{(a)}_{1}\otimes V_{1}[1]\arrow[ul,dashed,"\Delta"]
&&
\kappa^{(a)}_{*}\mathcal{D}^{(a)}_{0}\otimes V_{0}[1]\arrow[ul,dashed,"\Delta"]
\end{tikzcd}
\]
where the tensor factor is defined by $V_t := \operatorname{Ext}^{*}_{\mathcal{P}(b_i)\times C}\big(\mathcal{O}_{\mathcal{P}(b_i)}(b-1-t),\, \mathcal{D}_s^{(b)}\big)$. In particular, when $0 \le t < b$, we have
\[
V_t =
\begin{cases}
\mathcal{O}_{C}, & t = b-1-s,\\[4pt]
0, & t \neq b-1-s.
\end{cases}
\]
Perhaps we should note that if $a=b$, then
\[
\phi_{*}\varphi^{*}\kappa^{(b)}_{*}\mathcal{D}_s^{(b)}(-(b-1))
\simeq
\kappa^{(a)}_{*}\mathcal{D}^{(a)}_{b-1-s},
\quad \forall\, 0 \le s < b .
\]
\end{corollary}
\begin{proof}
We push forward the filtration along $\phi$, obtaining the stated filtraion. The vanishing of cohomology is given by the properties of the dual exceptional collection $\{\mathcal{D}_s^{(b)}\}$ e.g. Lemma \ref{orddualcomp}, noting our flip is over $C$.
\end{proof}
We now consider the right adjoint of $\phi_{*}\varphi^{*}$:
\[
\phi_{*}\varphi^{*}\dashv \varphi_{*}\phi^{!}.
\]
Note that $\phi^{!}(-)=\phi^{*}(-)\otimes \mathcal{O}_{\mathcal{W}}((b-1)E)$. Then, the corresponding computation in the opposite direction is straightforward, and the proof is completely analogous to Lemma \ref{flipullbackcom}.
\begin{lemma}\label{phidualcom}
We have that $\phi^{!}\kappa^{(a)}_{*}\mathcal{O}_{\mathcal{P}(a_i)\times C}$ admits a filtration:
\[
\begin{tikzcd}[column sep=0em]
\phi^{!}\kappa^{(a)}_{*}\mathcal{O}_{\mathcal{P}(a_i)\times C}\arrow{rr}&& F^{b-2}\arrow{dl}&&  \\
& \iota_{*}\mathcal{O}_{\mathcal{P}(a_i)}\boxtimes\mathcal{D}^{(b)}_{b-1}(-(b-1))[1]\arrow[ul,dashed,"\Delta"]&&
\end{tikzcd}
\quad\cdots\]
\[\cdots\quad
\begin{tikzcd}[column sep=0.5em]
& F^1\arrow{rr}&& F^0\arrow{rr}\arrow{dl}&& 0\arrow{dl} \\
&& \iota_{*}\mathcal{O}_{\mathcal{P}(a_i)}(1-(b-1))\boxtimes\mathcal{D}^{(b)}_{1}(-(b-1))[1]\arrow[ul,dashed,"\Delta"]
&&
\iota_{*}\mathcal{O}_{\mathcal{P}(a_i)}(-(b-1))\boxtimes\mathcal{D}^{(b)}_{0}(-(b-1))[1]\arrow[ul,dashed,"\Delta"]
\end{tikzcd}
\]
where we use the convention $\mathcal{F}\boxtimes\mathcal{G}:=\operatorname{pr}^{(a)*}\mathcal{F}\otimes\operatorname{pr}^{(b)*}\mathcal{G}\otimes\operatorname{pr}^{(c)*}\mathcal{O}_C$.
\end{lemma}
\begin{corollary}\label{dualcomp2}
For any dual pre-tilting object $\mathcal{D}_s^{(a)}$ on $\mathcal{P}(a_i)\times C$, we have
\[
\varphi_{*}\phi^{!}\kappa^{(a)}_{*}\mathcal{D}_s^{(a)}
=
\begin{cases}
0, & b \le s < a,\\[4pt]
\kappa^{(b)}_{*}\mathcal{D}^{(b)}_{\,b-1-s}(-(b-1)), & 0 \le s < b .
\end{cases}
\]
\end{corollary}
We obtain the lemma needed for the sequel.
\begin{lemma}\label{trandualprop}
For any dual pre-tilting object $\mathcal{D}_s^{(b)}$ on $\mathcal{P}(b_i)\times C$, we have
\[
\varphi_{*}\phi^{!}\phi_{*}\varphi^{*}\kappa^{(b)}_{*}\mathcal{D}_s^{(b)}(-(b-1))
\simeq
\kappa^{(b)}_{*}\mathcal{D}_s^{(b)}(-(b-1)),
\]
where $0 \le s < b$.
\end{lemma}
\begin{proof}
The object $\phi_{*}\varphi^{*}\kappa^{(b)}_{*}\mathcal{D}^{(b)}_{s}(-(b-1))$ admits a filtration with factors $V_t\otimes\kappa^{(a)}_{*}\mathcal{D}^{(a)}_{t}$ for $b\le t<a$, together with the factor $\kappa^{(a)}_{*}\mathcal{D}^{(a)}_{b-1-s}$ by Corollary \ref{dualcompuprop}. Applying $\varphi_{*}\phi^{!}$ to these factors, Corollary  \ref{dualcomp2} implies that only $\varphi_{*}\phi^{!}\kappa^{(a)}_{*}\mathcal{D}^{(a)}_{b-1-s}$ contributes, which gives precisely $\kappa^{(b)}_{*}\mathcal{D}^{(b)}_{s}(-(b-1))$.
\end{proof}

\begin{lemma}\label{equalcollection}
In $\mathrm{D}^b(\mathcal{P}(a_i)\times C)$, the following two decompositions are equivalent:
\[
\langle \mathcal{D}^{(a)}_{a-1},\ldots,\mathcal{D}^{(a)}_{b}\rangle
\simeq
\langle \mathcal{O}_{\mathcal{P}(a_i)\times C}(-e),\ldots,\mathcal{O}_{\mathcal{P}(a_i)\times C}(-1)\rangle,
\]
where $e:=a-b>0$.
\end{lemma}
\begin{proof}
Since
\[
\langle \mathcal{D}^{(a)}_{a-1},\ldots,\mathcal{D}^{(a)}_{b},\mathcal{D}^{(a)}_{b-1},\ldots,\mathcal{D}^{(a)}_{0}\rangle
\simeq
\langle \mathcal{D}^{(a)}_{a-1},\ldots,\mathcal{D}^{(a)}_{b},\mathcal{O}_{\mathcal{P}(a_i)\times C},\ldots,\mathcal{O}_{\mathcal{P}(a_i)\times C}(b-1)\rangle
\]
and
\[
\simeq
\langle \mathcal{O}_{\mathcal{P}(a_i)}(-e),\ldots,\mathcal{O}_{\mathcal{P}(a_i)}(-1),\mathcal{O}_{\mathcal{P}(a_i)\times C},\ldots,\mathcal{O}_{\mathcal{P}(a_i)\times C}(b-1)\rangle,
\]
we obtain the claimed equivalence of decompositions.
\end{proof}
\begin{corollary}\label{standardlemma1}
By Corollary \ref{dualcomp2} and Lemma \ref{equalcollection} we have
\[
\varphi_{*}\phi^{!}\kappa^{(a)}_{*}\mathcal{O}_{\mathcal{P}(a_i)\times C}(-t)=0,
\qquad
0<t\le e .
\]
\end{corollary}
Like the situation in Lemma \ref{ptspanningclass} or \ref{wblspanningclass}, we have:
\begin{lemma}
For any $0 \le s < b$, and any dual pre-tilting object $\mathcal{D}_s^{(b)}$ on $\mathcal{P}(b_i)\times C$, the objects $\kappa^{(b)}_{*}\mathcal{D}_s^{(b)}(-(b-1))$ form a spanning class in $\mathrm{D}^b(\mathcal{X}^{+})$. An analogous statement holds for $\mathrm{D}^b(\mathcal{X})$.
\end{lemma}
\begin{corollary} \label{standardlemma2}
The functor $\phi_{*}\varphi^{*}$ is admissible and fully faithful.
\end{corollary}

\begin{proof}
By \cite[Proposition 1.49]{Huybrechts}, it suffices to show
\[
\operatorname{Ext}^{*}_{\mathcal{X}}
\big(
\phi_{*}\varphi^{*}\mathcal{F},
\phi_{*}\varphi^{*}\mathcal{G}
\big)
\simeq
\operatorname{Ext}^{*}_{\mathcal{X}^{+}}(\mathcal{F},\mathcal{G})
\]
for any $\mathcal{F},\mathcal{G}\in \kappa^{(b)}_{*}\mathcal{D}^{(b)}_{s}(-(b-1))$, where $0\le s<b$. However,
\[
\operatorname{Ext}^{*}_{\mathcal{X}}
\big(
\phi_{*}\varphi^{*}\mathcal{F},
\phi_{*}\varphi^{*}\mathcal{G}
\big)
\simeq
\operatorname{Ext}^{*}_{\mathcal{X}^{+}}
\big(
\mathcal{F},
\varphi_{*}\phi^{!}\phi_{*}\varphi^{*}\mathcal{G}
\big)
\simeq
\operatorname{Ext}^{*}_{\mathcal{X}^{+}}(\mathcal{F},\mathcal{G}),
\]
which follows immediately from the Lemma \ref{trandualprop}.
\end{proof}
\begin{lemma} \label{standardlemma3}
We have the cohomological computation
\[
\operatorname{Ext}^{*}_{\mathcal{X}}
\big(
\kappa^{(a)}_{*}\mathcal{O}_{\mathcal{P}(a_i)\times C},
\kappa^{(a)}_{*}\mathcal{O}_{\mathcal{P}(a_i)\times C}(-t)
\big)
=
\begin{cases}
\mathcal{O}_{C}, & t=0,\\[4pt]
0, & 0<t < e .
\end{cases}
\]
\end{lemma}

\begin{proof}
Considering that $\mathcal{P}(a_i)\times C$ admits a Koszul resolution in $\mathcal{X}$, we have
\[
0
\to
\mathcal{O}_{\mathcal{X}}\!\big(\sum_{i=1}^{k}b_i\big)
\xrightarrow{F_{k-1}}
\cdots
\to
\bigoplus_{i<j}
\mathcal{O}_{\mathcal{X}}(b_i+b_j)
\xrightarrow{F_{1}}
\bigoplus_{i=1}^{k}
\mathcal{O}_{\mathcal{X}}(b_i)
\xrightarrow{F_{0}}
\mathcal{O}_{\mathcal{X}}
\to
\kappa^{(a)}_{*}\mathcal{O}_{\mathcal{P}(a_i)\times C}
\to
0 .
\]
Restricting the above complex to $\mathcal{P}(a_i)\times C$, and noting that
\[
\mathcal{O}_{\mathcal{X}}(t)\big|_{\mathcal{P}(a_i)\times C}
=
\mathcal{O}_{\mathcal{P}(a_i)\times C}(t),
\]
together with
\[
\mathrm{H}^{*}
\big(\mathcal{P}(a_i)\times C,
\mathcal{O}_{\mathcal{P}(a_i)\times C}(-t)
\big)
=
\begin{cases}
\mathcal{O}_{C}, & t=0,\\[4pt]
0, & 0<t<a,
\end{cases}
\]
we obtain the claim.
\end{proof}
\begin{proof}[Proof of Proposition~\ref{BOflopprop}]
By the previous results \ref{standardlemma1}, \ref{standardlemma2} and \ref{standardlemma3}, it remains to prove the generation of
\[
\mathcal{D}
:=
\left\langle
\operatorname{Im}\big(\kappa^{(a)}_{*}\mathrm{D}^{b}(C)\otimes \mathcal{O}_{\mathcal{P}}(-e)\big),
\ldots,
\operatorname{Im}\big(\kappa^{(a)}_{*}\mathrm{D}^{b}(C)\otimes \mathcal{O}_{\mathcal{P}}(-1)\big),
\operatorname{Im}\,\phi_{*}\varphi^{*}
\right\rangle .
\]
Since all the functors involved are admissible, it suffices to show that $\mathcal{D}$ contains a spanning class of $\mathrm{D}^{b}(\mathcal{X})$.

By Proposition~\ref{dualcompuprop}, $\phi_{*}\varphi^{*}\kappa^{(b)}_{*}\mathcal{D}^{(b)}_{s}(-(b-1))$ admits a filtration with factors $V_t\otimes\kappa^{(a)}_{*}\mathcal{D}^{(a)}_{t}$ for $b\le t<a$, together with the factor $\kappa^{(a)}_{*}\mathcal{D}^{(a)}_{b-1-s}$. By Lemma~\ref{trandualprop}, the former objects already belong to $\mathcal{D}$, hence $\kappa^{(a)}_{*}\mathcal{D}^{(a)}_{b-1-s}$ also belongs to $\mathcal{D}$. Applying Lemma~\ref{trandualprop} once again, we conclude that $\kappa^{(a)}_{*}\mathcal{D}^{(a)}_{t}$ belongs to $\mathcal{D}$ for all $0\le t<a$. Therefore $\mathcal{D}$ contains a spanning class of $\mathrm{D}^{b}(\mathcal{X})$, completing the proof.
\end{proof}
Clearly, the above argument extends suitably to the global setting of $\mathcal{X}$, provided it is analytically locally isomorphic to the situation considered above. However, in the local setting, we still need to consider more general cases.

\subsubsection{Type I section}

We now consider the weighted blow-up and weighted blow-down realization of the VGIT construction described in the previous section:
\[
\begin{tikzcd}[column sep=small, row sep=small]
& \bar{\mathcal{W}} \arrow[dl,"\bar{\phi}"'] \arrow[dr,"\bar{\varphi}"] & \\
\bar{\mathcal{X}} && \bar{\mathcal{X}}^{+}
\end{tikzcd}
\]
where $\bar{\mathcal{X}}=\big[\operatorname{Spec} k[[x_i,y_i,z_i]] \setminus V(x_i)\, /\, \mathbb{G}_m \big]$. We now consider a $\mathbb{G}_{m}$-equivariant element $f$ of $R:=k[[x_i,y_i,z_i]]$. In this section, we restrict to the basic situation where $f$ satisfies the following conditions:
\begin{enumerate}
\item $f$ is equivariant and has positive $\mathbb{G}_{m}$-degree, namely $f\in R_d$ for some $d>0$;
\item $f\in \langle x_i,y_i\rangle$, and moreover $f(x_i,y_i=0)\neq 0$ while $f(x_i=0,y_i)=0$;
\item $e:=\sum_{i=1}^{k}a_i-\sum_{i=1}^{l}b_i-d\ge0$.
\end{enumerate}
Such an $f$ defines a hypersurface $\mathcal{X}\subset \bar{\mathcal{X}}$. By taking the strict transform along the weighted blow-up and weighted blow-down, we obtain
\[
\begin{tikzcd}[column sep=small, row sep=small]
& \mathcal{W} \arrow[dl,"\phi"'] \arrow[dr,"\varphi"] & \\
\mathcal{X} && \mathcal{X}^{+}
\end{tikzcd}
\]
It is straightforward to see that $\mathcal{X}^{+}$ is also the hypersurface in $\bar{\mathcal{X}}^{+}$ defined by the same equation $f$ under the VGIT description. Hence, we may adopt the notation of the previous section to describe such hypersurface flips. We impose the following assumption:

\begin{itemize}
\item[4.] $\mathcal{X}$ and $\mathcal{X}^{+}$ are well-formed stacks, or at least have trivial stabilizers in codimension one.
\end{itemize}

In this situation, we have the stacky flipping loci $L\times C$ and $L^{+}\times C$, where
\[
L:=Z\big(f(x_i,y_i=0)\big)\subset \mathcal{P}(a_i),
\qquad
L^{+}:=Z\big(f(x_i=0,y_i)\big)=\mathcal{P}(b_i).
\]
Moreover, consider $L\times C\subset \mathcal{X}$ defined by the equations $\langle y_1,\ldots,y_k\rangle$. In particular, $\{y_1,\ldots,y_k,f\}$ forms a $\mathbb{G}_{m}$-equivariant regular sequence in $R$, hence by standard properties of local rings, $L\times C$ is a complete intersection in $\mathcal{X}$. On the other hand, although $\mathcal{P}(b_i)\times C\subset \mathcal{X}^{+}$, the locus $L^{+}\times C$ is not necessarily a complete intersection in $\mathcal{X}^{+}$. The exceptional divisor $E$ of $\mathcal{W}$ is naturally identified with $L\times L^{+}\times C$:
\[
\begin{tikzcd}[column sep=3.5em, row sep=2.2em]
L^{+}\times C
\arrow[r, hookrightarrow, "\kappa^{(b)}"]
\arrow[d]
& \mathcal{X}^{+}
\arrow[d] \\
C \arrow[r, equal]
& C
\end{tikzcd}
\quad\quad\text{and}\quad
\begin{tikzcd}[column sep=4em, row sep=3em]
E \simeq L\times L^{+}\times C
\arrow[r, "\iota^{(b)}"]
\arrow[d, "\mathrm{pr}^{(b,c)}"']
& \mathcal{W}
\arrow[d] \\
L^{+}\times C
\arrow[r, "\kappa^{(b)}"]
& \mathcal{X}^{+}
\end{tikzcd}
\]
We consider the following semi-orthogonal decomposition
\[
\mathrm{D}^{b}(L\times C)
=
\left\langle
\mathcal{A}^{(a)}_{-1},
\mathcal{O}_{L\times C},\ldots,\mathcal{O}_{L\times C}(a-d-1)
\right\rangle,
\]
and its associated dual pre-tilting collection
\[
\mathrm{D}^{b}(L\times C)
=
\left\langle
\mathcal{A}^{(a)}_{-1},
\mathcal{D}^{(a)}_{a-d-1},\ldots,\mathcal{D}^{(a)}_{0}
\right\rangle,
\]
where $\mathcal{D}^{(a)}_i := \mathbb{L}_{\langle \mathcal{O}_{L},\ldots,\mathcal{O}_{L}(i-1)\rangle}\mathcal{O}_{L}(i)$ for $0\le i<a-d$, and similarly define the corresponding constructions for $L^{+}\times C$. We also need the full saturated subcategory $\mathcal{K}^{(a)}_{-1}:=\left\langle \kappa^{(a)}_{*}\mathcal{A}^{(a)}_{-1}\right\rangle \subset \mathrm{D}^{b}(\mathcal{X})$.
\begin{proposition}\label{BOfloppropII}
The classical Bondal--Orlov formula gives a semiorthogonal decomposition analogous to the standard flip:
\[
\mathrm{D}^b(\mathcal{X})/\mathcal{K}^{(a)}_{-1}=
\left\langle
\operatorname{Im}\big(\kappa^{(a)}_* \mathcal{O}_{L\times C}(-e)\big),\,
\ldots,\,
\operatorname{Im}\big(\kappa^{(a)}_* \mathcal{O}_{L\times C}(-1)\big),\,
\operatorname{Im}\,\phi_*\mathcal{R}\varphi^*
\right\rangle,
\]
In particular, $\phi_{*}\mathcal{R}\varphi^{*}:\mathrm{D}^{b}(\mathcal{X}^{+})\to \mathrm{D}^{b}(\mathcal{X})/\mathcal{K}^{(a)}_{-1}$, and all functors involved are admissible and fully faithful.
\end{proposition}
\begin{lemma}
Like Lemma \ref{flipullbackcom}, we have that $\varphi^{*}\kappa^{(b)}_{*}O_{L^+\times C}$ admits a filtration:
\[
\begin{tikzcd}[column sep=0.5em]
\varphi^{*}\kappa^{(b)}_{*}\mathcal{O}_{L^+\times C}\arrow{rr}&&F^{s-1}\arrow{rr}\arrow{dl}&& F^{s-2}\arrow{dl}&&  \\
&\tau_s\arrow[ul,dashed,"\Delta"]&& \iota_{*}\mathcal{D}^{(a)}_{s-1}\boxtimes\mathcal{O}_{L^+}(s-1)[1]\arrow[ul,dashed,"\Delta"]&&
\end{tikzcd}
\quad\cdots\]
\[\cdots\quad
\begin{tikzcd}[column sep=0.5em]
& F^1\arrow{rr}&& F^0\arrow{rr}\arrow{dl}&& 0\arrow{dl} \\
&& \iota_{*}\mathcal{D}^{(a)}_{1}\boxtimes\mathcal{O}_{L^+}(1)[1]\arrow[ul,dashed,"\Delta"] && \iota_{*}\mathcal{D}^{(a)}_{0}\boxtimes\mathcal{O}_{L^+}[1]\arrow[ul,dashed,"\Delta"]
\end{tikzcd}
\]
for any integer $s$, and moreover $\mathcal{D}^{(a)}_{s}\in \mathcal{A}^{(a)}_{-1}$ for all $s\ge a-d$.
\end{lemma}
\begin{proof}
Recall that over $R$ we have the $\mathbb{G}_m$-equivariant exact resolution which induces on $\bar{\mathcal{X}}^{+}$:
\[
0
\to
\mathcal{P}_{t-\sum_{i=1}^{k}a_i}\otimes \mathcal{O}_{\bar{\mathcal{X}}^{+}}\!\big(\sum_{i=1}^{k}a_i\big)
\xrightarrow{F_{k-1}}
\cdots
\to
\bigoplus_{i<j}
\mathcal{P}_{t-a_i-a_j}\otimes \mathcal{O}_{\bar{\mathcal{X}}^{+}}(a_i+a_j)
\xrightarrow{F_{1}}
\bigoplus_{i=1}^{k}
\mathcal{P}_{t-a_i}\otimes \mathcal{O}_{\bar{\mathcal{X}}^{+}}(a_i)
\xrightarrow{F_{0}}
\mathcal{P}_{t}
\to
0 .
\]
By the division algorithm, we obtain the equivariant standard resolution over $R/f$ which induces:
\[\begin{tikzcd}[column sep=small, row sep=small]
\vdots\arrow[d]& \vdots\arrow[d] & \vdots\arrow[d]  &  \\
\bigoplus_{i=1}^{k}\mathcal{P}_{t-a_i-a_j-d}\otimes \mathcal{O}_{\mathcal{X}^{+}}(a_i+a_j+d)
\arrow[r,"F_{1}"]
\arrow[d,"K_{2}"]
&
\bigoplus_{i<j}
\mathcal{P}_{t-a_i-d}\otimes \mathcal{O}_{\mathcal{X}^{+}}(a_i+d)
\arrow[r,"F_{0}"]
\arrow[d,"-K_{1}"]
&
\mathcal{P}_{t-d}\otimes \mathcal{O}_{\mathcal{X}^{+}}(d)
\arrow[d,"K_{0}"]
& {}
\\
\cdots
\arrow[r,"F_{2}"]
&
\bigoplus_{i<j}
\mathcal{P}_{t-a_i-a_j}\otimes \mathcal{O}_{\mathcal{X}^{+}}(a_i+a_j)
\arrow[r,"F_{1}"]
&
\bigoplus_{i=1}^{k}
\mathcal{P}_{t-a_i}\otimes \mathcal{O}_{\mathcal{X}^{+}}(a_i)
\arrow[r,"F_{0}"]
&
\mathcal{P}_{t}
\end{tikzcd}\]We take the sum over all nonnegative integers $t$ of the above exact sequences, then we obtain the following exact sequence on $\mathcal{W}$: \[\begin{tikzcd}[column sep=small, row sep=small]
\vdots\arrow[d] & \vdots\arrow[d] & \vdots\arrow[d] & {} \\
\mathop{\bigoplus}\limits_{i=1}^{k}
\mathcal{O}_{\mathcal{W}}((a_i+a_j+d)E) \otimes \varphi^*\mathcal{O}_{\mathcal{X}^{+}}(\cdot)
\arrow[r,"F_{1}"]
\arrow[d,"K_{2}"]
&
\mathop{\bigoplus}\limits_{i<j}
\mathcal{O}_{\mathcal{W}}((a_i+d)E) \otimes \varphi^*\mathcal{O}_{\mathcal{X}^{+}}(\cdot)
\arrow[r,"F_{0}"]
\arrow[d,"-K_{1}"]
&
\mathcal{O}_{\mathcal{W}}(dE) \otimes \varphi^*\mathcal{O}_{\mathcal{X}^{+}}(\cdot)
\arrow[d,"K_{0}"]
& {}
\\
\cdots
\arrow[r,"F_{2}"]
&
\mathop{\bigoplus}\limits_{i<j}
\mathcal{O}_{\mathcal{W}}((a_i+a_j)E) \otimes \varphi^*\mathcal{O}_{\mathcal{X}^{+}}(\cdot)
\arrow[r,"F_{1}"]
&
\mathop{\bigoplus}\limits_{i=1}^{k}
\mathcal{O}_{\mathcal{W}}(a_iE) \otimes \varphi^*\mathcal{O}_{\mathcal{X}^{+}}(\cdot)
\arrow[r,"F_{0}"]
&
\mathcal{O}_{\mathcal{W}}
\end{tikzcd}\]
Using the short exact sequence $0 \to \mathcal{O}_{\mathcal{W}} \to \mathcal{O}_{\mathcal{W}}(E) \to \mathcal{O}_{E}(E) \to 0$, and the restriction formula
\[
\mathcal{O}_{\mathcal{W}}(tE)\big|_{E}
\simeq
\mathcal{O}_{L}(-t)\boxtimes \mathcal{O}_{L^+}(-t),
\quad \forall t\in \mathbb{Z},
\]
and by comparing the iterated shifts of the above complex with the derived pullback complex
\[
\varphi^{*}\kappa^{(b)}_{*}\mathcal{O}_{L^+\times C},
\]
we obtain the desired conclusion.

\end{proof}
\begin{corollary}
For any dual pre-tilting object $\mathcal{D}_t^{(b)}$ on $L^+\times C$, we have $\varphi^{*}\kappa^{(b)}_{*}\mathcal{D}_s^{(b)}(-(b-1))$ admits a filtration:
\[
\begin{tikzcd}[column sep=0.5em]
\varphi^{*}\kappa^{(b)}_{*}\mathcal{D}_t^{(b)}(-(b-1))
\arrow{rr}
&& F^{s-1}
\arrow{rr}
\arrow{dl}
&& F^{s-2}
\arrow{dl}
&& {} \\
& \tau_s
\arrow[ul,dashed,"\Delta"]
&&
\iota_{*}\mathcal{D}^{(a)}_{s-1}\boxtimes\mathcal{D}_t^{(b)}(-(b-1))(s-1)[1]
\arrow[ul,dashed,"\Delta"]
&&
\end{tikzcd}
\quad\cdots
\]
\[
\cdots\quad
\begin{tikzcd}[column sep=0.5em]
& F^1 \arrow{rr} && F^0 \arrow{rr} \arrow{dl} && 0 \arrow{dl} \\
&&
\iota_{*}\mathcal{D}^{(a)}_{1}\boxtimes\mathcal{D}_t^{(b)}(-(b-1))(1)[1]
\arrow[ul,dashed,"\Delta"]
&&
\iota_{*}\mathcal{D}^{(a)}_{0}\boxtimes\mathcal{D}_t^{(b)}(-(b-1))[1]
\arrow[ul,dashed,"\Delta"]
\end{tikzcd}
\]
for any non-negative integer $s$ and $0 \le t < b$.
\end{corollary}
\begin{corollary}\label{dualcompuprop}
For any dual pre-tilting object $\mathcal{D}_t^{(b)}$ on $L^+\times C$, we have that
$\phi_{*}\varphi^{*}\kappa^{(b)}_{*}\mathcal{D}_t^{(b)}(-(b-1))$ admits a filtration:
\[
\begin{tikzcd}[column sep=0.5em]
\phi_{*}\varphi^{*}\kappa^{(b)}_{*}\mathcal{D}_t^{(b)}(-(b-1))
\arrow{rr}
&& G^{s-1}
\arrow{rr}
\arrow{dl}
&& G^{s-2}
\arrow{dl}
&& {} \\
&
\tau_s
\arrow[ul,dashed,"\Delta"]
&&
\iota_{*}\mathcal{D}^{(a)}_{s-1}\otimes V_{s-1}[1]
\arrow[ul,dashed,"\Delta"]
&&
\end{tikzcd}
\quad\cdots
\]
\[
\cdots\quad
\begin{tikzcd}[column sep=0.5em]
& G^1 \arrow{rr} && G^0 \arrow{rr} \arrow{dl} && 0 \arrow{dl} \\
&&
\iota_{*}\mathcal{D}^{(a)}_{1}\otimes V_1[1]
\arrow[ul,dashed,"\Delta"]
&&
\iota_{*}\mathcal{D}^{(a)}_{0}\otimes V_0[1]
\arrow[ul,dashed,"\Delta"]
\end{tikzcd}
\]
where the tensor factor is defined by
\[
V_s := \operatorname{Ext}^{*}_{L^+\times C}\big(\mathcal{O}_{L^+}(b-1-s),\, \mathcal{D}_t^{(b)}\big).
\]
In particular, when $0 \le s < b$, we have
\[
V_s =
\begin{cases}
\mathcal{O}_C, & s = b-1-t,\\[4pt]
0, & s \neq b-1-t.
\end{cases}
\]
where $\mathcal{D}^{(a)}_{t}\in \mathcal{A}^{(a)}_{-1}$ for all $t\ge a-d$, and $\mathcal{D}^{(a)}_{t}$ is the $t$-th dual pre-tilting object if $t<a-d$.
\end{corollary}
Let $\mathcal{M}$ be any Cohen--Macaulay sheaf on $\mathcal{X}^{+}$. Then $\mathcal{M}$ is torsion-free, hence admits a Bourbaki sequence over $\mathcal{X}^+$ by Corollary \ref{Bexactwps}:
\[
0 \longrightarrow \bigoplus_{i=1}^{\operatorname{rank}(\mathcal{M})-1} \mathcal{O}_{\mathcal{X}^+}(m_i)
\longrightarrow \mathcal{M}
\longrightarrow \mathcal{I} \otimes \mathcal{O}_{\mathcal{X}^+}(m_0)
\longrightarrow 0.
\]
Fix a standard basis of $R$ refining the $\mathbb{G}_m$-equivariant weighted valuation defined above, we can construct by a classical argument in \cite[Section 2]{Hao2025b} a standard resolution of $\mathcal{I}$ and hence of $\mathcal{M}$. Then we can construct an object $\mathcal{R}(\mathcal{M})\in \mathrm{D}^b(\mathcal{W})$ together with a filtration obtained by iteratively applying the above procedure:
\[
\begin{tikzcd}[column sep=0.5em]
& \mathcal{F}^s(\mathcal{R}(\mathcal{M})) \arrow{rr} && \mathcal{F}^{s-1} \arrow{dl} && {} \\
&& \iota_{*}\operatorname{gr}^{s}(\mathcal{M})\arrow[ul,dashed,"\Delta"] &&
\end{tikzcd}
\quad\cdots\quad
\begin{tikzcd}[column sep=0.5em]
& \mathcal{F}^2 \arrow{rr} && \mathcal{F}^1 \arrow{rr} \arrow{dl} && \mathcal{R}(\mathcal{M})\arrow{dl} \\
&& \iota_{*}\operatorname{gr}^{2}(\mathcal{M})\arrow[ul,dashed,"\Delta"]
&& \iota_{*}\operatorname{gr}^{1}(\mathcal{M})\arrow[ul,dashed,"\Delta"]
\end{tikzcd}
\]
for which we have $\operatorname{gr}^{i}(\mathcal{M})\in \mathcal{A}^{(a)}_{-1}\boxtimes \mathrm{D}^b(L^+)\quad \text{for every } i.$ Moreover we have
\[
\varphi^*\mathcal{M} = \varprojlim_{s \to \infty} \mathcal{F}^s(\mathcal{R}(\mathcal{M})).
\]
Thus, for any $\mathcal{F} \in \mathrm{D}^b(\mathcal{X}^+)$, take a sheaf approximation $\mathcal{M}[k] \to \mathcal{F} \to \mathcal{P} \to \mathcal{M}[k+1]$ with $\mathcal{M}$ maximal Cohen--Macaulay and $\mathcal{P}$ a perfect complex, we define $\mathcal{R}(\mathcal{M}) \to \mathcal{R}\varphi^*\mathcal{F} \to \varphi^*\mathcal{P}$, and hence obtain a (quasi-)functor
\[
\mathcal{R}\varphi^*: \mathrm{D}^b(\mathcal{X}^+) \to \mathrm{D}^b(\mathcal{W}),
\]
defined up to some ambiguity. Further, by taking pushforward along $\phi$ and projection, we obtain the following (quasi-)functor with some ambiguity:
\[
\phi_*\mathcal{R}\varphi^*: \mathrm{D}^b(\mathcal{X}^+) \to \mathrm{D}^b(\mathcal{X})/\mathcal{K}_{-1}.
\]
\begin{lemma}
This defines a well-defined triangulated functor
\[
\phi_*\mathcal{R}\varphi^* : \mathrm{D}^b(\mathcal{X}^+) \to \mathrm{D}^b(\mathcal{X})/\mathcal{K}_{-1}.
\]
\end{lemma}

\begin{proof}
By the classical argument as in \cite[Proposition 3.3]{Hao2025b}, we have that
\[
\mathcal{R}\varphi^* : \mathrm{D}^b(\mathcal{X}^+) \longrightarrow
\mathrm{D}^b(\mathcal{W})/\{\mathcal{K}_{-1}^+:=\langle \iota_*\mathcal{A}^{(a)}_{-1}\boxtimes \mathrm{D}^b(L^+\times C)\rangle\}
\]
is a well-defined triangulated functor. Moreover, since
\[
\phi_*: \mathrm{D}^b(\mathcal{W}) \longrightarrow \mathrm{D}^b(\mathcal{X})/\mathcal{K}_{-1}
\]
satisfies
\[
\phi_*\big\langle \iota_*\mathcal{A}^{(a)}_{-1}\boxtimes \mathrm{D}^b(L^+\times C)\big\rangle = 0,
\]
it induces a triangulated functor
\[
\phi_* : \mathrm{D}^b(\mathcal{W})/\langle \iota_*\mathcal{A}^{(a)}_{-1}\boxtimes \mathrm{D}^b(L^+\times C)\rangle
\longrightarrow \mathrm{D}^b(\mathcal{X})/\mathcal{K}_{-1}.
\]
The composition of these two functors gives the desired result.
\end{proof}
On the other hand, for the weighted blow-up $\phi$ whose center is a complete intersection, we have the pullback functor $\phi^* : \mathrm{D}^b(\mathcal{X}) \to \mathrm{D}^b(\mathcal{W})$, as well as the natural functors $\varphi_!\phi^*$ and $\varphi_*\phi^!$ from $\mathrm{D}^b(\mathcal{X})$ to $\mathrm{D}^b(\mathcal{X}^+)$. One can directly verify that both satisfy $\varphi_!\phi^*\mathcal{K}_{-1}=0$ and $\varphi_*\phi^!\mathcal{K}_{-1}=0$. So we have:
\begin{lemma}
The functors
\[
\varphi_!\phi^* : \mathrm{D}^b(\mathcal{X})/\mathcal{K}_{-1} \longrightarrow \mathrm{D}^b(\mathcal{X}^+)
\quad \text{and} \quad
\varphi_*\phi^! : \mathrm{D}^b(\mathcal{X})/\mathcal{K}_{-1} \longrightarrow \mathrm{D}^b(\mathcal{X}^+)
\]
are well-defined triangulated functors.
\end{lemma}
\begin{proof}
If $F\in \mathcal{A}^{(a)}_{-1}$, then by Lemma~\ref{phidualcom} we know that $\varphi^{!}\kappa^{(a)}_{*}F$ admits a filtration whose successive factors are of the form with the first factor lies in $\mathcal{A}^{(a)}_{-1},\ldots,\mathcal{A}^{(a)}_{-b}$. Since $a-d-b\ge 0$, for any object $G$ appearing in these components we have
\[
\mathrm{H}^*(L,G)=0.
\]
Therefore the desired vanishing follows.
\end{proof}

\begin{lemma}
The functor $\phi_* \mathcal{R}\varphi^*$ is admissible and fits into an adjoint triple
\[
\varphi_! \phi^* \;\dashv\; \phi_* \mathcal{R}\varphi^* \;\dashv\; \varphi_* \phi^!.
\]
\end{lemma}
\begin{proof}
It suffices to prove that for any $F \in \mathrm{D}^b(\mathcal{X}^+)$ and $G \in \mathrm{D}^b(\mathcal{X})/\mathcal{K}_{-1}$, we have
\[
\operatorname{Ext}^*_{\mathrm{D}^b(\mathcal{X})/\mathcal{K}_{-1}}
(\phi_* \mathcal{R}\varphi^* F,\, G)
\simeq
\operatorname{Ext}^*_{\mathrm{D}^b(\mathcal{X}^+)}
(F,\, \varphi_* \phi^! G).
\]
Consider natural map
\[
\left[
F \xrightarrow{\, g \,} \varphi_* \phi^! G
\right]
\;\longmapsto\;
\left[
\begin{tikzcd}[column sep=0.6em, row sep=small]
& \mathcal{R}^{s\gg 0}\varphi^* F \arrow[dl, "s"] \arrow[dr, "\varphi^*g"] & \\
\mathcal{R}\varphi^* F \arrow[dr, dashed, "f"] & & \phi^! G \arrow[dl, dashed, "t"] \\
& G^- &
\end{tikzcd}
\right]
\;\longmapsto\;\left[
\begin{tikzcd}[column sep=small, row sep=small]
& \phi_* \mathcal{R}^{s\gg 0}\varphi^* F \arrow[dl, "\phi_* s"'] \arrow[dr, "\phi_* \varphi^*g"] & \\
\phi_* \mathcal{R}\varphi^* F & & G
\end{tikzcd}
\right].
\]
where $\mathrm{Cone}(s) \in \mathcal{K}_{-1}^+$ and $\phi_* \mathrm{Cone}(s) \in \mathcal{K}_{-1}$. It follows that the above correspondence is both surjective and injective, here one may consider inside $\mathrm{D}^b(\mathcal{W})$ a quotient by a full saturated subcategory generated by $\phi^!\mathcal{K}_{-1}$ and $\mathcal{K}^+_{-1}$, and then take the corresponding pushout. Its natural inverse map is given by
\[
\left[
\begin{tikzcd}[column sep=small, row sep=small]
\phi_* \mathcal{R}\varphi^* F \arrow[dr, "f"] & & G \arrow[dl, "t"'] \\
& G^- &
\end{tikzcd}
\right]
\quad \longmapsto \quad
\left[
F \xrightarrow{\ \varphi_* \phi^! f\ } \varphi_* \phi^! G^- = \varphi_* \phi^! G
\right],
\]
In the last step, since $\tau := \mathrm{Cone}(t) \in \mathcal{K}_{-1}$, we have $\varphi_* \phi^! \tau = 0$ which is a well-defined natural map.
\end{proof}
Once this functorial construction is established, we are essentially reduced to the situation of the previous section, provided we pass to suitable quotient categories, we have:
\begin{corollary}
For any dual pre-tilting object $\mathcal{D}_s^{(a)}$ on $L\times C$, we have
\[
\varphi_{*}\phi^{!}\kappa^{(a)}_{*}\mathcal{D}_s^{(a)}
=
\begin{cases}
0, & b \le s < a-d,\\[4pt]
\kappa^{(b)}_{*}\mathcal{D}^{(b)}_{\,b-1-s}(-(b-1)), & 0 \le s < b .
\end{cases}
\]
\end{corollary}

\begin{lemma}\label{trandualprop}
For any dual pre-tilting object $\mathcal{D}_s^{(b)}$ on $L^{+}\times C$, we have
\[
\varphi_{*}\phi^{!}\phi_{*}\mathcal{R}\varphi^{*}\kappa^{(b)}_{*}\mathcal{D}_s^{(b)}(-(b-1))
\simeq
\kappa^{(b)}_{*}\mathcal{D}_s^{(b)}(-(b-1)),
\]
where $0 \le s < b$.
\end{lemma}

\begin{lemma}
In $\mathrm{D}^b(L\times C)/\mathcal{A}^{(a)}_{-1}$, the following two decompositions are equivalent:
\[
\langle \mathcal{D}^{(a)}_{a-d-1},\ldots,\mathcal{D}^{(a)}_{b}\rangle
\simeq
\langle \mathcal{O}_{L\times C}(-e),\ldots,\mathcal{O}_{L\times C}(-1)\rangle,
\]
where $e:=a-b-d>0$.
\end{lemma}
\begin{corollary}
We have
\[
\varphi_{*}\phi^{!}\kappa^{(a)}_{*}\mathcal{O}_{L\times C}(-t)=0,
\qquad
0<t\le e .
\]
\end{corollary}
\begin{lemma}
For any $0 \le s < b$, and any dual pre-tilting object $\mathcal{D}_s^{(b)}$ on $L^{+}\times C$, the objects $\kappa^{(b)}_{*}\mathcal{D}_s^{(b)}(-(b-1))$ form a spanning class in $\mathrm{D}^b(\mathcal{X}^{+})$.
\end{lemma}

\begin{corollary}\label{fftypeI}
The functor $\phi_{*}\mathcal{R}\varphi^{*}:\mathrm{D}^{b}(\mathcal{X}^{+})\to \mathrm{D}^{b}(\mathcal{X})/\mathcal{K}^{(a)}_{-1}$ is admissible and fully faithful.
\end{corollary}

\begin{lemma}
We have the cohomological computation
\[
\operatorname{Ext}^{*}_{\mathcal{X}}
\big(
\kappa^{(a)}_{*}\mathcal{O}_{L\times C},
\kappa^{(a)}_{*}\mathcal{O}_{L\times C}(-t)
\big)
=
\begin{cases}
\mathcal{O}_C, & t=0,\\[4pt]
0, & 0<t<e.
\end{cases}
\]
\end{lemma}

\begin{proof}
Since $L\times C\subset \mathcal{X}$ is still complete intersection and admits a Koszul resolution in $\mathcal{X}$:
\[
0
\to
\mathcal{O}_{\mathcal{X}}\!\big(\sum_{i=1}^{k}b_i\big)
\xrightarrow{F_{k-1}}
\cdots
\to
\bigoplus_{i<j}
\mathcal{O}_{\mathcal{X}}(b_i+b_j)
\xrightarrow{F_{1}}
\bigoplus_{i=1}^{k}
\mathcal{O}_{\mathcal{X}}(b_i)
\xrightarrow{F_{0}}
\mathcal{O}_{\mathcal{X}}
\to
\kappa^{(a)}_{*}\mathcal{O}_{L\times C}
\to
0 ,
\]
proof is the same as before.
\end{proof}

\begin{proof}[Proof of Proposition~\ref{BOfloppropII}]
It remains to prove generation. Let
\[
\mathcal{D}
:=
\left\langle
\operatorname{Im}\big(\kappa^{(a)}_*\mathrm{D}^b(C)\otimes \mathcal{O}_{L}(-e)\big),\,
\ldots,\,
\operatorname{Im}\big(\kappa^{(a)}_*\mathrm{D}^b(C)\otimes \mathcal{O}_{L}(-1)\big),\,
\operatorname{Im}\,\phi_*\mathcal{R}\varphi^*
\right\rangle .
\]
For any $b\in \mathrm{D}^{b}(\mathcal{X})$, consider the distinguished triangle
\[
\phi_{*}\mathcal{R}\varphi^{*}\varphi_{*}\phi^{!}b
\longrightarrow
b
\longrightarrow
c
\in\Delta
\qquad
\text{in }
\mathrm{D}^{b}(\mathcal{X})/\mathcal{K}^{(a)}_{-1}.
\]
Applying $\varphi_{*}\phi^{!}$ and using the isomorphism form Corollary \ref{fftypeI}
\[
\varphi_{*}\phi^{!}\phi_{*}\mathcal{R}\varphi^{*}\simeq \mathrm{id},
\]
we obtain $\varphi_{*}\phi^{!}c=0$. Hence, by the description of $\ker(\varphi_{*})$, the object $c$ belongs to the triangulated subcategory of $\mathrm{D}^{b}(\mathcal{X})/\mathcal{K}^{(a)}_{-1}$ generated by
\[
\iota_{*}\mathcal{O}_{L\times C}(-a+d+1),\ldots,\iota_{*}\mathcal{O}_{L\times C}(-1).
\]
We already know that
\[
\iota_{*}\mathcal{O}_{L\times C}(-e),\ldots,\iota_{*}\mathcal{O}_{L\times C}(-1)
\]
belong to $\mathcal{D}$. Moreover, from the computation of
\[
\phi_{*}\mathcal{R}\varphi^{*}\kappa^{(b)}_{*}\mathcal{D}^{(b)}_{t}(-(b-1)),
\]
from Corollary \ref{dualcompuprop}, we know that
\[
\kappa^{(a)}_{*}\mathcal{D}^{(a)}_{b-1},\ldots,\kappa^{(a)}_{*}\mathcal{D}^{(a)}_{0}
\]
also belong to $\mathcal{D}$. Therefore $c$ belongs to $\mathcal{D}$, completing the proof.
\end{proof}

\subsubsection{Type II section}
In this section, we keep the setup of the previous section, but consider a different class of equations $f$. More precisely, we restrict to the situation where $f$ satisfies the following conditions:
\begin{enumerate}
\item $f$ is equivariant and has negative $\mathbb{G}_{m}$-degree, namely $f\in R_{-d}$ for some $d>0$;
\item $f\in \langle x_i,y_i\rangle$, and moreover $f(x_i,y_i=0)=0$ while $f(x_i=0,y_i)\neq0$;
\item
$e:=\sum_{i=1}^{k}a_i-(\sum_{i=1}^{l}b_i-d)\ge0$ .
\end{enumerate}
Such an $f$ defines a hypersurface $\mathcal{X}\subset \bar{\mathcal{X}}$. By taking the strict transform along the weighted blow-up and weighted blow-down, we obtain
\[
\begin{tikzcd}[column sep=small, row sep=small]
& \mathcal{W} \arrow[dl,"\phi"'] \arrow[dr,"\varphi"] & \\
\mathcal{X} && \mathcal{X}^{+}
\end{tikzcd}
\]
and impose the following assumption:
\begin{itemize}
\item[4.] $\mathcal{X}$ and $\mathcal{X}^{+}$ are well-formed stacks, or at least have trivial stabilizers in codimension one.
\end{itemize}
In this situation, we have the stacky flipping loci $L\times C$ and $L^{+}\times C$, where
\[
L:=Z\big(f(x_i,y_i=0)\big)=\mathcal{P}(a_i),
\qquad
L^{+}:=Z\big(f(x_i=0,y_i)\big)\subset \mathcal{P}(b_i).
\]
Moreover, consider $L^{+}\times C\subset \mathcal{X}^{+}$ defined by the equations $\langle x_1,\ldots,x_l\rangle$. In particular, $\{x_1,\ldots,x_l,f\}$ forms a $\mathbb{G}_{m}$-equivariant regular sequence in $R$, hence by standard properties of local rings, $L^{+}\times C$ is a complete intersection in $\mathcal{X}^{+}$. On the other hand, the locus $L\times C$ is not necessarily a complete intersection in $\mathcal{X}$. The exceptional divisor $E$ of $\mathcal{W}$ is naturally identified with $L\times L^{+}\times C$. We consider the following semi-orthogonal decomposition
\[
\mathrm{D}^{b}(L^{+}\times C)
=
\left\langle
\mathcal{A}^{(b)}_{-(b-d)},
\mathcal{O}_{L^{+}\times C}{(-(b-d)+1)},\ldots,\mathcal{O}_{L^{+}\times C}
\right\rangle .
\]\label{SODtypeII}
and define the corresponding constructions for $L\times C$. We also define the full saturated subcategory $\mathcal{K}^{(b)}_{-}:=\left\langle \kappa^{(b)}_{*}\mathcal{A}^{(b)}_{-(b-d)}\right\rangle \subset \mathrm{D}^{b}(\mathcal{X}^{+})$. Usually, in order to introduce a spanning class in the quotient category, we use properties of the Jacobian ring, and therefore work in a suitable algebraic neighbourhood of $\mathcal{X}$ and $\mathcal{X}^{+}$. In particular, we may replace $\mathcal{X}$ and $\mathcal{X}^{+}$ by their algebraic models.
\begin{proposition}\label{BOfloppropIII}
The classical Bondal--Orlov formula gives a semiorthogonal decomposition analogous to the standard flip. In particular, we obtain an equivalence of categories
\[
\phi_{*}\varphi^{*}:
\mathrm{D}^{b}(\mathcal{X}^{+})/\mathcal{K}^{(b)}_{-}
\xrightarrow{\sim}
\mathrm{D}^{b}(\mathcal{X})/\mathcal{L}_{-},
\]
where
\[
\mathcal{L}_{-}
:=
\left\langle
\operatorname{Im}\bigl(\kappa^{(a)}_{*}\mathcal{O}_{L\times C}(-e)\bigr),
\ldots,
\operatorname{Im}\bigl(\kappa^{(a)}_{*}\mathcal{O}_{L\times C}(-1)\bigr)
\right\rangle .
\]
Here the objects $\kappa^{(a)}_{*}\mathcal{O}_{L\times C}(-t)$ are in general not pre-tilting objects, the functor $\phi_{*}\varphi^{*}$ is admissible.
\end{proposition}

If $F\in \mathcal{A}^{(b)}_{-(b-d)}$, we need a similar computation for $\varphi^{*}\kappa^{(b)}_{*}F$. It admits a filtration whose factors are of the form
\[
\iota_{*}\bigl(\mathcal{D}^{(a)}_{t}\boxtimes F(t)\bigr)[1],
\qquad 0\le t<a .
\]
Pushing forward along $\phi$, and using
\[
\operatorname{Ext}^{*}_{L^{+}}
\bigl(
\mathcal{O}_{L^{+}}(-t),F
\bigr)
=
0,
\qquad 0\le t<b-d,
\]
from (\ref{SODtypeII}), we obtain that $\phi_{*}\varphi^{*}\kappa^{(b)}_{*}F$ admits a filtration whose factors are of the form
\[
\kappa^{(a)}_{*}\bigl(V_{t}\otimes \mathcal{D}^{(a)}_{t}\bigr),
\qquad b-d\le t<a .
\]
By Lemma~\ref{equalcollection}, all these factors belong to $\mathcal{L}_{-}$. Hence we obtain a well-defined functor. Since left proof is a repetition of the previous section, we omit it. We note that the natural resolution of $L\times C\subset \mathcal{X}$ is no longer Koszul and is in general infinite; consequently, we cannot obtain any pre-tilting type result for objects of the form $\kappa^{(a)}_{*}\mathcal{O}_{L\times C}$ analogous to the Type I.

\subsubsection{Type III section}

In this section, we keep the setup of the previous section, but consider an equation $f$ satisfying the following conditions:
\begin{enumerate}
\item $f$ is equivariant and has $\mathbb{G}_m$-degree $d$, namely $f \in R_d$;
\item $f \in \langle x_i,y_i\rangle$, and moreover $f(x_i,y_i=0)=f(x_i=0,y_i)=0$;
\item $e:=\sum_{i=1}^{k}a_i-\sum_{i=1}^{l}b_i\ge 0$.
\end{enumerate}

Such an $f$ defines hypersurfaces $\mathcal{X}\subset \bar{\mathcal{X}}$ and $\mathcal{X}^{+}\subset \bar{\mathcal{X}}^{+}$. By taking the strict transform along the weighted blow-up and weighted blow-down, and inspecting the exceptional divisor, one usually has
\[
w\mathrm{Bl}^{(b_i)}\mathcal{X}
\not\simeq
w\mathrm{Bl}^{(a_i)}\mathcal{X}^{+}.
\]
Instead, we consider
\[
\widehat{\mathcal{W}}
:=
\bar{\phi}^{-1}(\mathcal{X})
\simeq
\bar{\varphi}^{-1}(\mathcal{X}^{+}),
\]
fitting into the projection diagram
\[
\begin{tikzcd}[column sep=small, row sep=small]
& \widehat{\mathcal{W}} \arrow[dl,"\phi"'] \arrow[dr,"\varphi"] & \\
\mathcal{X} && \mathcal{X}^{+}
\end{tikzcd}
\]
and impose the following assumption:
\begin{itemize}
\item[4.] $\mathcal{X}$ and $\mathcal{X}^{+}$ are well-formed stacks, or at least have trivial stabilizers in codimension one.
\end{itemize}
In this situation, we have the stacky flipping loci $L\times C$ and $L^{+}\times C$, where
\[
L:=\mathcal{P}(a_i),
\qquad
L^{+}:=\mathcal{P}(b_i)
,
\]
while the exceptional divisor $E$ of $\widehat{\mathcal{W}}$ is naturally identified with $L\times L^{+}\times C$. We further note that $\widehat{\mathcal{W}}$ is the total transform of the weighted blow-up. This implies that the morphisms $\phi$ and $\varphi$ are perfect, so that both pullback and pushforward preserve bounded complexes. Consequently, we obtain the following.

\begin{proposition}\label{BOfloppropIV}
The classical Bondal--Orlov formula gives a semiorthogonal decomposition analogous to the standard flip:
\[
\mathrm{D}^b(\mathcal{X})
=
\left\langle
\left\langle
\operatorname{Im}(\kappa^{(a)}_*\mathcal{O}_{L\times C}(-e)),\ldots,
\operatorname{Im}(\kappa^{(a)}_*\mathcal{O}_{L\times C}(-1))
\right\rangle,
\operatorname{Im}\,\phi_*\varphi^*
\right\rangle .
\]
The functors $\kappa^{(a)}_*\mathcal{O}_{L\times C}(-t)$ are general not pre-tilting objects, but their essential images still generate admissible subcategories of $\mathrm{D}^b(\mathcal{X})$. Moreover, the functor
\[
\phi_{*}\varphi^{*}:
\mathrm{D}^{b}(\mathcal{X}^{+})
\longrightarrow
\mathrm{D}^{b}(\mathcal{X})
\]
is admissible and fully faithful.
\end{proposition}
We briefly indicate the argument. The total transformation does not affect the computation of the derived pullback (cf.~Lemma~\ref{flipullbackcom}), since the exceptional locus remain unchanged, the proof is therefore identical after a Tor-independent base change.

\subsubsection{Type IV section}

In this section, we keep the setup of the previous section, but consider a collection of equations $f_l$ for $1 \le l \le j$, satisfying the following conditions:
\begin{enumerate}
\item each $f_l$ is invariant and has $\mathbb{G}_m$-degree $d_l=0$, namely $f_l \in R_0$;
\item $f_l \in \langle x_i,y_i,z_i\rangle$, and moreover $f_l(x_i,y_i=0)=f_l(x_i=0,y_i)=g_l(z_i)\neq  0$;
\item $\{g_l\}_{1\le l\le j}$ is a regular sequence in $k[[z_i]]$;
\item $e:=\sum_{i=1}^{k}a_i-\sum_{i=1}^{l}b_i-\sum_{l=1}^{j} d_l \ge 0$.
\end{enumerate}
Such a collection $\langle f_1,\cdots, f_j\rangle$ defines a complete intersection $\mathcal{X}\subset \bar{\mathcal{X}}$. By taking the strict transform along the weighted blow-up and weighted blow-down, we obtain
\[
\begin{tikzcd}[column sep=small, row sep=small]
& \mathcal{W} \arrow[dl,"\phi"'] \arrow[dr,"\varphi"] & \\
\mathcal{X} && \mathcal{X}^{+}
\end{tikzcd}
\]
and impose the following assumption:
\begin{itemize}
\item[5.] $\mathcal{X}$ and $\mathcal{X}^{+}$ are well-formed stacks, or at least have trivial stabilizers in codimension one.
\end{itemize}
In this situation, we have the stacky flipping loci $L\times \operatorname{Spec} A$ and $L^{+}\times \operatorname{Spec} A$, where
\[
L:=\mathcal{P}(a_i),
\qquad
L^{+}:= \mathcal{P}(b_i),
\]
while the exceptional divisor $E$ of $\mathcal{W}$ is identified with $L\times L^{+}\times \operatorname{Spec} A$, where $A:=k[[z_i]]/\langle g_1,\cdots, g_j\rangle$ is an Artinian algebra.   In fact, we are considering a certain thickening of the geometric flipping locus.\\

On the other hand, we note that the strict transform and the total transform of $f_l$ coincide in all directions of the weighted blow-up, since $val(f_l)=0$. This implies that the morphisms $\phi$ and $\varphi$ are perfect, so that both pullback and pushforward preserve bounded complexes. We therefore obtain
\begin{proposition}\label{BOfloppropV}
The classical Bondal--Orlov formula gives a semiorthogonal decomposition analogous to the standard flip:
\[
\mathrm{D}^b(\mathcal{X})
=
\left\langle\left\langle
\operatorname{Im}(\kappa^{(a)}_*\mathcal{O}_{L}(-e)),\ldots,
\operatorname{Im}(\kappa^{(a)}_*\mathcal{O}_{L}(-1))\right\langle,
\operatorname{Im}\,\phi_*\varphi^*
\right\rangle,
\]
where $\kappa^{(a)}_*\mathcal{O}_{L\times C}(-t)$ are general not pre-tilting, but their essential images still generate admissible subcategories of $\mathrm{D}^b(\mathcal{X})$. Moreover, the functor
\[
\phi_{*}\varphi^{*}:
\mathrm{D}^{b}(\mathcal{X}^{+})
\longrightarrow
\mathrm{D}^{b}(\mathcal{X})
\]
are admissible and fully faithful.
\end{proposition}

\subsubsection{Mixed section}

In this section, we keep the setup of the previous section, but consider a collection of equations $f^{(t)}_l$ for non-negative integer $r^{(t)}$ and $t=1,\ldots,4$, satisfying the following conditions:
\begin{enumerate}
\item each $f^{(t)}_l$ has $\mathbb{G}_m$-degree $d^{(t)}_l$ or $-d^{(t)}_l$ for some non-negative integer $d^{(t)}_l$, according to Types I--IV, namely $f^{(t)}_l \in R_{d^{(t)}_l}$;
\item for $t=1,\ldots,4$, the collections $\{f^{(t)}_l\}_{l=1}^{r^{(t)}}$ satisfy the standard type conditions corresponding to flip types I–IV, and the total collection forms a regular sequence with $r^{(4)}=j$;
\item for $t=1,2$, the sequences $f^{(t)}_1,\ldots,f^{(t)}_{r^{(t)}}$ form regular sequences and satisfy the measured condition (\ref{cond:measured}) under weighted blow-up, i.e. their leading terms define the exceptional divisors and cut out complete intersections;
\item for $t=3$, we assume that $r^{(3)}\leq 1$ and that its total transform
is a regular element with respect to the other strict transforms.
\item $e:=(\sum_{i=1}^{k}a_i-\sum_{l} d^{(1)}_l) -(\sum_{i=1}^{l}b_i-\sum_{l} d^{(2)}_l) \ge 0$.
\end{enumerate}
Such a collection $\{f^{(t)}_l\}$ for each $t$ generates an ideal $I$, which defines substacks
\[
\mathcal{X}\subset \bar{\mathcal{X}}, \qquad \mathcal{X}^{+}\subset \bar{\mathcal{X}}^{+}.
\]
A more conceptual description is given via VGIT diagram:   \[
\begin{tikzcd}[column sep=small, row sep=small]
\mathcal{X}:=\left[ Z(I)\subset \operatorname{Spec} R //_{-}\mathbb{G}_m \right]
\arrow[dr]
&&
\mathcal{X}^+:=\left[ Z(I)\subset \operatorname{Spec} R //_{+}\mathbb{G}_m \right]
\arrow[dl] \\
&
\mathcal{Z}:=\left[ Z(I)\subset \operatorname{Spec} R //_{0}\mathbb{G}_m \right]
&
\end{tikzcd}
\]
Since the semistable loci for the $+$- and $-$-linearizations are both contained in the semistable locus for the $0$-linearization, we obtain natural morphisms.
\begin{itemize}
\item[6.] $\mathcal{X}$ and $\mathcal{X}^{+}$ are well-formed stacks, or at least have trivial stabilizers in codimension one.
\end{itemize}
We can define $\widehat{\mathcal{W}}:=\mathcal{X} \times_{\mathcal{Z}} \mathcal{X}^{+}$ and obtain the following diagram:
\[
\begin{tikzcd}[column sep=small, row sep=small]
& \widehat{\mathcal{W}} \arrow[dl, "\phi"'] \arrow[dr, "\varphi"] & \\
\mathcal{X} && \mathcal{X}^{+}
\end{tikzcd}
\]
In this situation, we have the stacky thicken  flipping loci $L\times \operatorname{Spec} A$ and $L^{+}\times \operatorname{Spec} A$, where
\[
L:=Z\big(I(x_i,y_i=z_i=0)\big)\subset \mathcal{P}(a_i),
\qquad
L^{+}:=Z\big(I(x_i=z_i=0,y_i)\big)=\mathcal{P}(b_i).
\]
They are complete intersection centers in the weighted projective stack, where $A=R/\langle x_i,y_i,I\rangle$. The exceptional divisor $E$ of $\mathcal{W}$ is naturally identified with $L\times L^{+}\times \operatorname{Spec} A$. We consider the following semi-orthogonal decomposition
\[
\mathrm{D}^{b}(L)
=
\left\langle
\mathcal{A}^{(a)}_{-1},
\mathcal{O}_{L},\ldots,\mathcal{O}_{L}(a-d^{(1)}-1)
\right\rangle,
\]
where $d^{(1)}:=\sum_{l=1}^{r^{(1)}} d^{(1)}_l$, and similarly for
\[
\mathrm{D}^{b}(L^{+})
=
\left\langle
\mathcal{A}^{(b)}_{-(b-d^{(2)})},
\mathcal{O}_{L^{+}}(-(b-d^{(2)})+1),\ldots,\mathcal{O}_{L^{+}}
\right\rangle,
\]
where $d^{(2)}:=\sum_{l=1}^{r^{(2)}} d^{(2)}_l$, then define the full saturated subcategories
\[
\mathcal{K}^{(a)}_{-1}
:=
\left\langle \kappa^{(a)}_{*}\mathcal{A}^{(a)}_{-1} \right\rangle
\subset \mathrm{D}^{b}(\mathcal{X}),
\qquad
\mathcal{K}^{(b)}_{-}
:=
\left\langle \kappa^{(b)}_{*}\mathcal{A}^{(b)}_{-(b-d^{(2)})} \right\rangle
\subset \mathrm{D}^{b}(\mathcal{X}^{+}).
\]

\begin{proposition}\label{BOfloppropM}
The classical Bondal--Orlov formula gives a semiorthogonal decomposition analogous to the standard flip. Let
\[
\mathcal{L}_{-}
:=
\left\langle
\operatorname{Im}(\kappa^{(a)}_{*}\mathcal{O}_{L}(-e)),\ldots,
\operatorname{Im}(\kappa^{(a)}_{*}\mathcal{O}_{L}(-1))
\right\rangle .
\]
Then we obtain an equivalence
\[
\phi_{*}\mathcal{R}\varphi^{*}:
\mathrm{D}^{b}(\mathcal{X}^{+})/\mathcal{K}^{(b)}_{-}
\longrightarrow
\mathcal{L}_{-}\backslash\mathrm{D}^{b}(\mathcal{X})/\mathcal{K}^{(a)}_{-1}.
\]
The objects $\kappa^{(a)}_{*}\mathcal{O}_{L}(-t)$ are in general not pre-tilting objects. However, if there is no contribution from Type II hypersurfaces, then $\mathcal{L}_{-}$ becomes an admissible subcategory of $\mathrm{D}^{b}(\mathcal{X})/\mathcal{K}^{(a)}_{-1}$, and the functor $\phi_{*}\mathcal{R}\varphi^{*}:\mathrm{D}^{b}(\mathcal{X}^{+})/\mathcal{K}^{(b)}_{-}\to \mathrm{D}^{b}(\mathcal{X})/\mathcal{K}^{(a)}_{-1}$ is admissible and fully faithful.
\end{proposition}
As in the argument above, we work in a suitable algebraic neighbourhood in order to obtain a good spanning class. Once such spanning classes are available, we can pass to the global construction without changing any of the arguments.
\begin{proposition}
Using such spanning classes, the above properties extend to any well-formed stack birational map $\mathcal{X} \dashrightarrow \mathcal{X}^{+}$ satisfying the formal local models described above.
\end{proposition}

\begin{remark}
We expect that the above embeddings should extend directly, without passing to quotient categories. However, if one insists on using pullback--pushforward functors along linear centers, it not seems sufficient to produce the desired result. On the other hand, allowing more complicated resolution diagrams may make such direct embeddings possible.
\end{remark}

\section{Terminal singularities and Derived category}

In this chapter, we mainly consider three-dimensional algebraic varieties over $\mathrm{k} := \mathds{C}$. We recall some basic notions, using \cite{KM} and \cite{Kawakita} as the main references.

\begin{definition}
Let $X$ be a normal variety. We say that $X$ is $\mathds{Q}$-Gorenstein if there exists an integer $m>0$ such that $rK_{X}$ is Cartier. The smallest such integer $r$ is called the index of $X$.
\end{definition}

\begin{definition}
Let $X$ be a variety, $f:Y\to X$ a birational morphism from a normal variety $Y$, and $E\subset Y$ an irreducible divisor. Such an $E$ is called a divisor over $X$. The closure of $f(E)$ in $X$ is called the center of $E$ on $X$, denoted by $\operatorname{center}_{X}(E)$. As above, the center depends only on the valuation $val(E,Y)$. If $\operatorname{center}_{X}(E)$ has codimension at least two in $X$, we say that the corresponding valuation is algebraic.
\end{definition}

\begin{definition}
Let $X$ be a $\mathds{Q}$-Gorenstein normal variety, and let $val$ be an algebraic valuation over $X$. Choose a birational morphism $f:Y\to X$ from a normal variety $Y$ such that $val=v(E,Y)$ for some prime divisor $E\subset Y$. The discrepancy of $val$ is defined to be the coefficient $a(val)$ of $E$ in the formula
\[
K_{Y}\sim_{\mathds{Q}} f^{*}K_{X}+a(val)E+\cdots ,
\]
where the remaining terms are supported on the other exceptional divisors. This definition is independent of the choice of the birational model $Y$.
\end{definition}

\begin{definition}
Let $X$ be a $\mathds{Q}$-Gorenstein algebraic, analytic, or formal variety. We say that $X$ has only terminal singularities if, for every algebraic valuation $val$ over $X$, the discrepancy $a(val)$ is strictly positive.
\end{definition}

In dimension three we have a rather precise understanding of such singularities. The following classical result is due to M.~Reid.

\begin{proposition}[Lemma 1.6.2, \cite{Kawakita}]
If $X$ is a three-dimensional terminal variety, then $X$ has only isolated singularities.
\end{proposition}

\begin{theorem}[Theorem 2.3.4, \cite{Kawakita}]
A threefold singularity is terminal of index one if and only if it is an isolated compound Du Val ($\mathrm{cDV}$) singularity.
\end{theorem}

\begin{definition}
Let $x \in X$ be the analytic (or formal) germ of a threefold singularity. We say that $x \in X$ is a compound Du Val (cDV) singularity if a general hyperplane section at $x$ is a Du Val singularity. We say that $x \in X$ is of type c$A_n$, c$D_n$, or c$E_n$ according as the Du Val singularity at $x$ is of type $A_n$, $D_n$, or $E_n$ respectively.
\end{definition}

For a threefold  terminal singularity of higher index, one can analytically (or formally) construct its index-one cover: Let $R$ be the formal local ring of the singularity and $X=\operatorname{Spec}R$. Since $X$ is $\mathds{Q}$-Gorenstein of index $r$, we have that $rK_X$ is Cartier. Choosing a nowhere vanishing section
\[
s \in \mathcal{O}_X(-rK_X),
\]
we define the index-one cyclic cover
\[
\varepsilon : X^{+} := \operatorname{Spec}_X \left( \bigoplus_{i=0}^{r-1} \mathcal{O}_X(iK_X) \right) \longrightarrow X
\]
which is étale in codimension one and makes $K_{X^{+}}$ Cartier. We have:
\begin{proposition}[Corollary 2.2.21, \cite{Kawakita}]
Let
\[
\varepsilon : X^{+} := \operatorname{Spec}_X \left( \bigoplus_{i=0}^{r-1} \mathcal{O}_X(iK_X) \right) \longrightarrow X
\]
be the index-one cyclic cover. Then we have:
\begin{enumerate}
\item $\varepsilon^{*}K_X \sim_{\mathds{Q}} K_{X^{+}}$;
\item $X^{+}$ has only terminal singularities of index one;
\item $X^{+}$ admits a natural $\mu_r$-cyclic cover, which is ramified over points.
\item $X^{+}/\mu_r \simeq X$.
\end{enumerate}
\end{proposition}

Hence we can verify that $X^{+}$ is formally normal and satisfies Serre’s condition $(\mathrm{R}_2)$. Moreover, it is a formal hypersurface with isolated singularities, hence it is algebraizable. In addition, it carries a compatible $\mu_r$-action which is free in codimension one. By Section \ref{Canonicalstack} we have:
\begin{corollary}\label{terminalstack}
A three-dimensional terminal variety is a well-formed variety. Moreover, via Vistoli’s construction, it induces a natural morphism to a well-formed stack.  We obtain an equivalence of categories with codimension preserved morphisms
\[
\widetilde{(-)}:\{\text{terminal three-dimensional varieties}\}
\longrightarrow
\{\text{terminal three-dimensional canonical stacks}\}.
\]
\end{corollary}
Subsequent works S.Mori gave a complete classification of the defining equations of such singularities.
\begin{theorem} [Theorem 2.4.8, \cite{Kawakita}]
Let $o \in X$ be a terminal threefold singularity of index $r > 1$, realized as
$o \in X = (\xi = 0) \subset \mathbb{A}^4 / \mu_r(a_1,a_2,a_3,a_4)$
with $\xi$ equivariant. Then, up to changing the expression of the quotient type $\frac{1}{r}(a_1,a_2,a_3,a_4)$ and orbifold coordinates, one of the following holds ($\mathfrak{m}$ is the maximal ideal at the origin):
\begin{align*}
\mathrm{cA}/r &\quad \xi = x_1x_2 + f(x_3, x_4^r),\quad \tfrac{1}{r}(1,-1,0,b),\ (b,r)=1. \\
\mathrm{cAx}/4 &\quad \xi = x_1^2 + x_2^2 + f(x_3, x_4^2),\quad f \in \mathfrak{m}^2,\quad \tfrac{1}{4}(1,3,2,1). \\
\mathrm{cAx}/2 &\quad \xi = x_1^2 + x_2^2 + f(x_3, x_4),\quad f \in \mathfrak{m}^4,\quad \tfrac{1}{2}(1,0,1,1). \\
\mathrm{cD}/3 &\quad \xi = x_1^2 + f(x_2,x_3,x_4),\quad f \in \mathfrak{m}^3,\quad
\text{cubic part: } x_2^3+x_3^3+x_4^3,\ x_2^3+x_3x_4^2,\ \text{or } x_2^3+x_3^3,\quad \tfrac{1}{3}(0,1,2,2). \\
\mathrm{cD}/2 &\quad \xi = x_1^2 + f(x_2,x_3,x_4),\quad f \in \mathfrak{m}^3,\quad
x_2x_3x_4 \text{ or } x_2^2x_3 \text{ appears},\quad \tfrac{1}{2}(1,1,0,1). \\
\mathrm{cE}/2 &\quad \xi = x_1^2 + x_3^2 + f(x_3,x_4)x_2 + g(x_3,x_4),\quad
f \in \mathfrak{m}^4,\ g \in \mathfrak{m}^4\setminus\mathfrak{m}^5,\quad \tfrac{1}{2}(1,0,1,1).
\end{align*}
Terminal cyclic quotient singularities are included in the case $\mathrm{cA}/r$.
\end{theorem}
Then we could consider birational maps in this category and study the induced behavior on derived categories.
\subsection{Divisorial contraction}
\begin{definition}
Let $X$ be a $\mathds{Q}$-Gorenstein terminal threefold. A birational morphism
\[
f : Y \longrightarrow X
\]
between terminal varieties is called a divisorial contraction if its exceptional locus $E \subset Y$ is a prime divisor and $-E$ is $f$-relatively ample.
\end{definition}
The simplest case is when $X$ is smooth at a point $x$, i.e. $(X,x)$ is a regular germ. One may ask what are the possible divisorial contractions $f: Y \to X$ with center $x \in X$. The answer is given by M.Kawatika:

\begin{theorem}[Theorem 3.5.1, \cite{Kawakita}]
Let $f: E \subset Y \to x \in X$ be a threefold divisorial contraction contracting a divisor $E$ to a smooth point $x \in X$. Then there exists a regular system of parameters $x_1,x_2,x_3$ in $\mathcal{O}_{X,x}$ such that $f$ is the weighted blow-up with weights $\operatorname{wt}(x_1,x_2,x_3)=(1,r_1,r_2)$, where $r_1,r_2$ are coprime positive integers.
\end{theorem}

In the derived category setting, this situation is well known to satisfy the following properties, e.g. Section \ref{Weightedblowup}.

\begin{corollary}
Let $f:Y\to X$ be a divisorial contraction as above, with exceptional divisor $E$ and $e:=r_1+r_2+1$. Then:
\begin{enumerate}
\item We have the discrepancy formula
\[
K_Y \sim_{\mathds{Q}}  f^{*}K_X + (e-1)E.
\]

\item For the associated canonical stacks, there is a semiorthogonal decomposition
\[
\mathrm{D}^{b}(\widetilde{Y})
=
\left\langle
\mathcal{O}_{\widetilde{E}}(-e+1),\ldots,\mathcal{O}_{\widetilde{E}}(-1),f^{*}\mathrm{D}^{b}(\widetilde{X})
\right\rangle.
\]

\item The Bondal--Orlov localization formula holds:
\[
\mathrm{D}^{b}(Y)/\mathcal{L}_- \simeq \mathrm{D}^{b}(X).
\]
\end{enumerate}
\end{corollary}

Another case is when we consider a germ of a point $x$, where $(X,x)$ has a quotient singularity of type $\frac{1}{r}(1,a,r-a)$ for coprime integers $0<a<r$. We may ask to classify all divisorial contractions $f:Y\to X$ with center $x\in X$. The answer is given by Y.Kawamata:
\begin{theorem}[Theorem 3.1.5, \cite{Kawakita}]
Let $f:Y\to X$ be a threefold divisorial contraction to the germ $x\in X$ of a terminal quotient singularity of type $\frac{1}{r}(1,a,r-a)$. Then $f$ is the Kawamata blow-up.
\end{theorem}
By Section \ref{Kawamatablowup}, we have:
\begin{corollary}
Let $f:Y\to X$ be a divisorial contraction as above. Then:
\begin{enumerate}
\item We have the discrepancy formula
\[
K_Y \sim_{\mathds{Q}}  f^{*}K_X + \frac{1}{r} E.
\]
\item For the associated canonical stacks, there is a semiorthogonal decomposition
\[
\mathrm{D}^{b}(\widetilde{Y})
=
\left\langle
\mathcal{O}_{\widetilde{E}}(-1),\,
\varpi_{*}\pi^{*}\mathrm{D}^{b}(\widetilde{X})
\right\rangle.
\]

\item The Bondal--Orlov localization formula holds:
\[
\mathrm{D}^{b}(Y)/\mathcal{L}_-\simeq \mathrm{D}^{b}(X).
\]
\end{enumerate}
\end{corollary}
The above result can be further generalized to divisorial contractions over terminal singularities of higher index. More precisely, let $x\in X$ be a germ of a terminal singularity of index $r>1$, and consider divisorial contractions $f:Y\to X$ with center $x\in X$ and discrepancy $1/r$. Then we have the following due to T.Hayakawa:
\begin{theorem}[Theorem 4.1 \cite{Hayakawa1999}, Theorem 1.1 \cite{Hayakawa2000}]
Let $f:Y\to X$ be a threefold divisorial contraction as above.
\begin{enumerate}
\item Assume that $x\in X$ is not of type $(\mathrm{cD}/2)$. Then there exists a local formal embedding
\[
X \hookrightarrow \mathbb{A}^{4}/\mu_r(a_1,\ldots,a_4)
\]
defined by an equivariant equation $\xi$.
\item Assume that $x\in X$ is of type $(\mathrm{cD}/2)$. Then there exists a local formal embedding
\[
X \hookrightarrow \mathbb{A}^{5}/\mu_r(a_1,\ldots,a_5)
\]
defined by equivariant regular sequence $\xi_1,\xi_2$. In particular, they form a measured basis.
\end{enumerate}

In both cases, the morphism $f$ is the Kawamata blow-up associated to the corresponding embedding.
\end{theorem}
Likewise, by Section~\ref{Kawamatablowup} we have:
\begin{corollary}
Let $f:Y\to X$ be a divisorial contraction as above. Then:
\begin{enumerate}
\item We have the discrepancy formula
\[
K_Y \sim_{\mathds{Q}}  f^{*}K_X + \frac{1}{r} E.
\]
\item For the associated canonical stacks, there is a semiorthogonal decomposition up to a quotient:
\[
\mathrm{D}^{b}(\widetilde{Y})/\mathcal{K}_{-}
=
\left\langle
\mathcal{O}_{\widetilde{E}}(-1),\,
\varpi_{*}\mathcal{R}\pi^{*}\mathrm{D}^{b}(\widetilde{X})
\right\rangle.
\]

\item The Bondal--Orlov localization formula holds:
\[
\mathrm{D}^{b}(Y)/\mathcal{L}_- \simeq \mathrm{D}^{b}(X).
\]
\end{enumerate}
\end{corollary}
Then Hayakawa  \cite{Hayakawa2005} considered divisorial contractions with discrepancy $r/r$ to the above non-Gorenstein terminal points $(x,X)$. Kawakita \cite{Kawakita2005}  further generalized this work to arbitrary divisorial contractions with discrepancy greater than $1/r$. More precisely, we have the following result.
\begin{theorem}[Theorem 3.5.8, \cite{Hayakawa2000}]
Let $f:Y\to X$ be a threefold divisorial contraction to a non-Gorenstein terminal singular point $x\in X$ of index $r$, with discrepancy different from $1/r$. Then there exists a local formal embedding
\[
X \hookrightarrow \mathbb{A}^{5}/\mu_r(a_1,\ldots,a_5)
\]
defined by an equivariant regular sequence $\xi_1,\xi_2$, they form a measured basis. The morphism $f$ is the  Kawamata blow-up or weighted blow-up  associated to the corresponding embedding.
\end{theorem}
By Section~\ref{Kawamatablowup} and \ref{Weightedblowup}, there is:
\begin{corollary}
Let $f:Y\to X$ be a divisorial contraction as above. Then:
\begin{enumerate}
\item We have the discrepancy formula
\[
K_Y \sim_{\mathds{Q}} f^{*}K_X + (e-r)E\quad \text{or}\quad K_Y \sim_{\mathds{Q}} f^{*}K_X + (e-1)E,
\]
depending on the type of blow-up, where $e:=\sum_i a_i-d_1-d_2$ and $d_i:=val(\xi_i)$ for $i=1,2$.
\item For the associated canonical stacks, there is a semiorthogonal decomposition up to a quotient:
\[
\mathrm{D}^{b}(\widetilde{Y})/\mathcal{K}_{-}
=
\left\langle
\mathcal{O}_{\widetilde{E}}(-(e-r)),\ldots,
\mathcal{O}_{\widetilde{E}}(-1),\,
\varpi_*\mathcal{R}\pi^{*}\mathrm{D}^{b}(\widetilde{X})
\right\rangle
\,\text{or}\,
\left\langle
\mathcal{O}_{\widetilde{E}}(-(e-1)),\ldots,
\mathcal{O}_{\widetilde{E}}(-1),\,
\mathcal{R}\pi^{*}\mathrm{D}^{b}(\widetilde{X})
\right\rangle.
\]

\item The Bondal--Orlov localization formula holds:
\[
\mathrm{D}^{b}(Y)/\mathcal{L}_-\simeq \mathrm{D}^{b}(X).
\]
\end{enumerate}
\end{corollary}
The remaining case concerns divisorial contractions to a Gorenstein point. Kawakita divided this situation into two classes. The first is the ordinary type, for which he obtained a classification analogous to the above results \cite[Theorem 3.5.5-7]{Kawakita}. The remaining case is the exceptional type, whose classification was later given by Yamamoto \cite{Yamamoto2018}. Although Kawakita expressed different opinions regarding the proof \cite[Remark 3.5.17]{Kawakita}, at least many examples of exceptional types are known. Thus we obtain the following general classification principle.
\begin{theorem}[\cite{Yamamoto2018}]
Let $f:Y\to X$ be a generic threefold divisorial contraction to a Gorenstein terminal singular point $x\in X$. Then there exists a local formal embedding
\[
X \hookrightarrow \mathbb{A}^{5}/\mu_r(a_1,\ldots,a_5)
\]
defined by an equivariant regular sequence $\xi_1,\xi_2$, they form a measured basis. The morphism $f$ is the weighted blow-up associated to the corresponding embedding.
\end{theorem}
\begin{corollary}\label{corollaryin1}
Let $f:Y\to X$ be a generic  divisorial contraction as above. Then:
\begin{enumerate}
\item We have the discrepancy formula
\[
K_Y \sim_{\mathds{Q}} f^{*}K_X + (e-1)E.
\]
where $e:=\sum_i a_i-d_1-d_2$ and $d_i:=val(\xi_i)$ for $i=1,2$.
\item For the associated canonical stacks, there is a semiorthogonal decomposition up to a quotient:
\[
\mathrm{D}^{b}(\widetilde{Y})/\mathcal{K}_{-}
=
\left\langle
\mathcal{O}_{\widetilde{E}}(-(e-1)),\ldots,
\mathcal{O}_{\widetilde{E}}(-1),\,
\mathcal{R}\pi^{*}\mathrm{D}^{b}(X)
\right\rangle.
\]

\item The Bondal--Orlov localization formula holds:
\[
\mathrm{D}^{b}(Y)/\mathcal{L}_- \simeq \mathrm{D}^{b}(X).
\]
\end{enumerate}
\end{corollary}
\begin{remark}
In nearly all of the proofs above, the author implicitly assumes that the measured condition (\ref{cond:measured}) holds, i.e., that the constructed regular sequence \( f_i \) satisfies that \( L(f_i) \) generate the entire exceptional divisor. One can verify this case by case through their proofs; perhaps it is obvious.
\end{remark}

If the center of a divisorial contraction is a smooth curve, H.K.Chen, J.J.Chen and J.A.Chen \cite[Definition 3.2]{CCC} introduced a useful operation called \emph{tilting of valuations}. Let $R$ be a domain over $\mathrm{k}$, $v$ a pre-valuation on $R$, and $\mathfrak a$ its associated $p$-filtration. Fix an integer $a\ge 0$. One defines a pre-valuation $v^{\sharp}$ on $R[x]$ or $R[[x]]$ by
\[
v^{\sharp}\!\left(\sum_{i\ge0} c_i x^i\right)
=
\min_{i\ge0}\{v(c_i)+ia\}.
\]
Its associated $p$-filtration $\mathfrak a^{\sharp}$ is given by
\[
\mathfrak a^{\sharp}_n
=
\sum_{i+aj\ge n}\mathfrak a_i[x]\cdot(x^j).
\]
Pick any $s\in R$, and let
\[
\pi:R[x]\to R
\]
be the surjective morphism sending $x$ to $s$. Denote by $\mathfrak a^{+}:=\pi(\mathfrak a^{\sharp})$ the associated $p$-filtration of $v^{+}:=v^{\sharp}|_{R}$. Then:
\begin{enumerate}
\item $\mathfrak a^{+}_n\supset \mathfrak a_n$ for all $n$, equivalently $v^{+}\succeq v$;
\item if $a>v(s)$, then $v^{+}\neq v$;
\item if $a\le v(s)$, then $v^{+}=v$.
\end{enumerate}

Using this construction, they proved the following result.

\begin{theorem}[Theorem 3.6, \cite{CCC}]
Let $X$ be a germ of a Gorenstein singularity and $f:Y\to X$ a divisorial contraction with exceptional divisor $E$. Then $f$ can be realized as a weighted blow-up. More precisely, there exists an embedding $X\hookrightarrow W_0$ into a four-dimensional affine space $W_0$, together with a further embedding $W_0\hookrightarrow W$ into an affine space $W$ obtained by finitely many tilting steps with respect to the divisorial valuation, such that $f$ is the weighted blow-up associated to this new embedding.
\end{theorem}
Chens extended the above results to the non-Gorenstein case using an equivariant argument in \cite[Proposition 3.19]{CCC}.\\

However, we still need two additional conditions. First, the embedding should satisfy the measured condition (\ref{cond:measured}) . Chens’ work introduced a more refined tilting algorithm  \cite[Section 5]{CCC} to guarantee this property, at least for several explicitly constructed examples \cite[Tables]{CCC}\footnote{It is not clear to the author why this holds in general.}. Second, we require that the exceptional locus $E\to C$ is a flat (\ref{flatcond}) morphism. Again, for the examples constructed by the authors, this condition can be verified case by case, e.g., Hironaka's criteria. If we call these assumptions the generic conditions, then we obtain the following result via Section \ref{Weightedblowup}.

\begin{corollary}\label{corollaryin2}
Let $f:Y\to X$ be a generic divisorial contraction with a smooth curve center $C$. Then:
\begin{enumerate}
\item We have the discrepancy formula
\[
K_Y \sim_{\mathds{Q}} f^{*}K_X + E,
\]
where $e:=\sum_i a_i-\sum_i d_i=2$ and $d_i:=val(\xi_i)$.

\item For the associated canonical stacks, there is a semiorthogonal decomposition up to a quotient:
\[
\mathrm{D}^b(\widetilde{Y})/\mathcal{H}_{-1}
=
\left\langle\operatorname{Im}\big(\iota_*\big(f^*\mathrm{D}^b(\mathcal{C})\otimes \mathcal{O}_{E_{\widetilde{Y}}}(-1)\big)\big),\,
\operatorname{Im}\,\mathcal{R}\pi^*
\right\rangle,
\]

\item The Bondal--Orlov localization formula holds:
\[
\mathrm{D}^{b}(Y)/\mathcal{L}_- \simeq \mathrm{D}^{b}(X).
\]
\end{enumerate}
\end{corollary}
\begin{proof}
It suffices to show that $\mathcal{K}_{-1}=0$. Since $Y$ has only terminal singularities, its singular locus is isolated. Hence every exceptional locus $E_{\mathcal{Z}}$ outside $E$ contains no non-Gorenstein locus. Away from $E$, the morphism $Y\to X$ is a Gorenstein divisorial contraction. By Cutkosky's theorem, for example \cite[Theorem 4.1.9]{Kawakita}, we know that $Y$ is the blow-up of $X$ along a curve $C$. Therefore $\mathcal{K}_{-1}=0$ outside $E$.
\end{proof}
\subsection{Feasible resolution}
Subsequent work by Hayakawa \cite {Hayakawa2005b} studied how to construct resolutions of three-dimensional terminal singularities via such divisorial contractions. Of course, one needs to require that each step of the resolution is well-controlled and sufficiently compatible with terminal singularities. This led to the development of the notion of \emph{economical or feasible resolutions}.
\begin{definition} [Definition 1.2, \cite{Chen2016}]
Given a three-dimensional terminal singularity $P\in X$, we say that there exists a \emph{feasible resolution} for $P\in X$ if there is a sequence
\[
X_n \to X_{n-1} \to \cdots \to X_1 \to X_0 = X \ni P,
\]
such that $X_n$ is non-singular and each morphism $X_{i+1}\to X_i$ is a divisorial contraction to a point with minimal discrepancy, i.e. a contraction to a point $P_i\in X_i$ of index $r_i\ge 1$ with discrepancy $1/r_i$.
\end{definition}

J.K.Chen \cite{Chen2016} gave a particularly elegant answer to this problem:

\begin{theorem}  [Theorem 1.3, \cite{Chen2016}]
Given a three-dimensional terminal singularity, there exists a feasible resolution, such that each step is either a Kawamata blow-up or a weighted blow-up of a hypersurface embedding along a point as center.
\end{theorem}
From this theorem, we obtain two fundamental consequences. We first recall that for any proper birational morphism $f:Y\to X$ between smooth varieties, or any proper contraction satisfying $f_*\mathcal{O}_Y=\mathcal{O}_X$, there exists a natural semiorthogonal decomposition
\[
\mathrm{D}^{b}(Y)=\left\langle \langle \ker f_{*}\rangle,\ \operatorname{Im}(f^{*}\mathrm{D}^{b}(X))\right\rangle.
\]
Using resolutions, we can extend this result to certain singular varieties $X$.
\begin{corollary}\label{corollaryin3}
For any proper birational morphism (or a proper contraction) $f:Y\to X$ with $X$ a terminal threefold and $Y$ has rational singulairties, the Bondal--Orlov localization formula holds:
\[
\mathrm{D}^{b}(Y)/\langle \ker f_{*}\rangle \simeq \mathrm{D}^{b}(X).
\]
\end{corollary}
\begin{proof}
We consider a feasible resolution $\pi_X: X^{+}\to X$ as constructed above. Let $Y \dashrightarrow X^{+}$ be the induced rational map. We take a common resolution of indeterminacy:
\[
\begin{tikzcd}[column sep=small, row sep=small]
& W \arrow[dl, "\pi_Y"'] \arrow[dr, "\pi_{X^{+}}"] & \\
Y \arrow[dr, "f"'] & & X^{+} \arrow[dl, "\pi_X"] \\
& X &
\end{tikzcd}
\]
For any $F \in \mathrm{D}^{b}(X)$, by the construction of a feasible resolution, $X$ is obtained as a sequence of Kawamata blow-ups and weighted blow-ups. In particular using results in Section \ref{Kawamatablowup} and  \ref{Weightedblowup}, one obtains an induced filtration on $\pi_X^{*}F$ by objects of the form $\mathcal{R}\mathcal{}^{\mathbf{s}}\pi^{*}F$, where $\mathbf{s}$ is a multi-index corresponding to the resolution steps, and such that
\[
\varprojlim_{\mathbf{s}\to\infty}  \mathcal{R}^{\mathbf{s}}\pi_X^{*}F \simeq \pi_X^{*}F.
\]
This allows us to define a functor
\[
\mathcal{R}f^{*}: \mathrm{D}^{b}(X) \longrightarrow \mathrm{D}^{b}(Y)/\langle \ker f_{*} \rangle,
\qquad
F \mapsto \pi_{Y*}\bigl(\pi_{X^+}^{*}\mathcal{R}^{\mathbf{s}}\pi_{X}^{*}F\bigr),
\]
well-defined via the above common resolution. Since $f$ is a contraction, we can see that
\[
\varprojlim_{\mathbf{s}\to\infty} \mathcal{R}^{\mathbf{s}}f^{*}F \simeq f^{*}F,\qquad f_{*}\mathcal{R}f^*F \simeq F.
\]By the standard argument in global adjunction \cite[Lemma 3.9]{Hao2025b}, it follows that it an equivalence.
\end{proof}

\begin{corollary}
For any proper birational stacky morphism $\pi:\widetilde{Y}\to\widetilde{X}$
with $X,Y$ terminal threefolds, the stacky  Bondal--Orlov localization formula holds:
\[
\mathrm{D}^{b}(\widetilde{Y})/\langle \ker \pi_{*}\rangle
\simeq
\mathrm{D}^{b}(\widetilde{X}).
\]
\end{corollary}

\begin{proof}
Similarly, we consider feasible resolutions
\[
\pi_X:X^{+}\to X,
\qquad
\pi_Y:Y^{+}\to Y,
\]
together with a common resolution $W$. We obtain the following commutative diagram:
\[
\begin{tikzcd}[column sep=small, row sep=small]
&W \arrow[dl, "\pi_{X^+}"'] \arrow[dr, "\pi_{Y^+}"] & \\
X^+ \arrow[d, "\pi_X"'] & &Y^+ \arrow[d, "\pi_Y"] \\
X \arrow[rr, "f"] && Y
\end{tikzcd}
\]
Since $\pi_X$ is a composition of Kawamata blow-ups and weighted blow-ups, we obtain a corresponding composition of embedding functors, which we denote by
\[
\Theta_X : \mathrm{D}^{b}(\widetilde{X})
\longrightarrow
\mathrm{D}^{b}(X^{+})/\mathcal{K}_{X},
\]
where $\mathcal{K}_{X}$ is the triangulated subcategory generated by the exceptional contributions appearing in each step. For instance, in the case of two successive Kawamata blow-ups, the root stack formalism allows us to complete the following diagram:
\[
\begin{tikzcd}[column sep=small, row sep=small]
\sqrt[r_1]{\sqrt[r_2]{X^{(2)}}}
\arrow[d, "\pi_{2}"']
\arrow[r, "\varpi_{2}"]
&
\sqrt[r_2]{X^{(2)}}
\arrow[d, "\pi_{1}"]
\arrow[r, "\varpi_{1}"]
&
X^{(2)}
\\
\sqrt[r_1]{X^{(1)}}
\arrow[d, "\pi_{0}"']
\arrow[r, "\varpi_{0}"]
&
X^{(1)}
&
\\
X
&
{}
\end{tikzcd}
\]
By Lemma \ref{flatbasechange0}, we clearly have
\[
\varpi_{1*}\pi_{1}^{*}\varpi_{0*}\pi_{0}^{*}
\simeq
\varpi_{1*}\varpi_{2*}\pi_{2}^{*}\pi_{0}^{*},
\]
this applies to the feasible resolutions on both directions, after replacing $W$ by a sufficiently refined root stack $\widehat{W}$, we obtain
\[
\begin{tikzcd}[column sep=small, row sep=small]
& \widehat{W} \arrow[dl, "\pi_{\widetilde{Y}}"'] \arrow[dr, "\pi_{\widetilde{X}}"] & \\
\widetilde{Y}\arrow[rr,"\pi"] & & \widetilde{X}
\end{tikzcd}
\]
together with induced functors
\[
\mathcal{R}\pi^{*} :
\mathrm{D}^{b}(\widetilde{X})
\longrightarrow
\mathrm{D}^{b}(\widetilde{Y})/\langle \ker \pi_{*}\rangle,
\qquad
\mathcal{R}\pi^{*}(F)
:=
\pi_{\widetilde{Y}*}\mathcal{R}^{\mathbf{s}}\pi_{\widetilde{X}}^{*}(F).
\]
Notice that we have $\pi_{*}\mathcal{O}_{\widetilde{Y}}\simeq \mathcal{O}_{\widetilde{X}}$, together with
\[
\varprojlim_{\mathbf{s}\to\infty} \mathcal{R}^{\mathbf{s}}\pi^{*}F \simeq \pi^{*}F,
\qquad
\pi_{*}\mathcal{R}\pi^{*}F \simeq F.
\]
By global adjunction \cite[Lemma 3.9]{Hao2025b}, the conclusion follows.
\end{proof}
We note that not every proper birational morphism $f:Y\to X$ admits a canonical lift to the associated canonical stacks, unless $f$ is an isomorphism in codimension one.
\begin{corollary}
For any flipping, flipped contraction or $\mathds{Q}$-factorialization  between terminal threefolds, the Bondal--Orlov localization formula holds both at the variety and canonical stack level.
\end{corollary}
Another important case is that of a divisorial contraction to a curve. Let $f:Y\to X$ be a divisorial contraction between terminal threefolds. Restricting to the complement of small locus, we obtain an open subset on which
$f^{0}:Y^{0}\to X^{0}$ is a smooth divisorial contraction onto a smooth curve. Using exactly the same argument as in Lemma~\ref{canonicalstackkey}, one proves that there exists a unique stacky lift
\[
\pi:\widetilde{Y}\to\widetilde{X}.
\]
As a consequence, we obtain the following.
\begin{corollary}[Cor. \ref{corollaryin2}]
For any divisorial contraction to a curve between terminal threefolds, the Bondal--Orlov localization formula holds both at the variety level and at the canonical stack level.
\end{corollary}
We now consider the most general situation. Let $f:Y\rightarrow X$ be a proper birational map between terminal threefolds, and assume that \(Y\) is \(\mathds Q\)-factorial. Suppose that
\[
K_Y\sim f^*K_X+\sum_{i=1}^k \frac{a_i}{s_i}E_i.
\]
For simplicity, we assume \(k=1\), and write $K_Y\sim f^*K_X+\frac{a}{s}E$.  Let \(x\in X\) be a terminal germ of index \(r\), and let \(X^+\to X\) and \(Y^+\to Y\) be the corresponding index-one covers. Since \(Y\) is \(\mathds Q\)-factorial, locally \(E\sim tK_Y\) for some \(t\in\mathds Z\), the pullback of \(E\) to \(Y^+\) is Cartier. We obtain a root stack/cyclic cover
\[\rho:Y^+(E^{1/s})\to Y^+\to Y.\]
\(Y^+(E^{1/s})\) is the root of \(Y^+\) along \(E\). Consider $Z \longrightarrow Y^+(E^{1/s}) \times_X X^+$ map to the fiber product with an extra quotient of torsion, we have
\[
\rho^*K_X
\sim
K_{Y^+}-aM,
\]
is Cartier, where \(M\) denotes the reduced pullback of \(E\). On the other hand side \(r\rho^*K_X\sim 0\), choosing a nowhere-vanishing section of \(r\rho^*K_X\) yields an étale cover \(Z\to Y^+(E^{1/s})\). Consequently, \(f\) lifts to a morphism
\(\sqrt[s]{\widetilde E/\widetilde Y}\to\widetilde X\). Moreover, this construction fits into the commutative diagram
\[
\begin{tikzcd}
\sqrt[s]{\widetilde E/\widetilde Y}
\arrow[r,"\varpi"]
\arrow[d,"\pi"]
&
\widetilde Y
\arrow[d]
\\
\widetilde X
\arrow[r]
&
X .
\end{tikzcd}
\]
\begin{corollary}[Cor.~\ref{corollaryin1}]\label{corollaryin5}
For any proper birational map $f$ between \(\mathds{Q}\)-factorial terminal threefolds,
the Bondal--Orlov localization formula holds at the stack level under \(\pi_*\),
if we replace \(\widetilde Y\) with \(\sqrt[s]{\widetilde E/\widetilde Y}\)
with respect to the discrepancy.
\end{corollary}

Finally, we extend \Cref{terminalstack} to the following more general form.
\begin{corollary}
We obtain an equivalence of categories
\[
\widetilde{(-)}\colon
\{\text{terminal 3d $\mathds{Q}$-factorial varieties}\}
\longrightarrow
\{\text{terminal 3d $\mathds{Q}$-factorial canonical stacks}\}
\]
via Vistoli's construction, and proper birational map corresponds to the equivalence class of morphisms
\[
\begin{aligned}
\left\{
\begin{array}{c}
\text{proper birational morphisms}\\
Y \xrightarrow{\,f\,} X
\end{array}
\right\}
&\longmapsto
\left\{\left[
\begin{tikzcd}[column sep=small, row sep=small]
& \sqrt[\mathbf{s}]{\widetilde{Y}} \arrow[dl,"\varpi"] \arrow[dr,"\pi"] & \\
\widetilde{Y} & & \widetilde{X}
\end{tikzcd}
\right]\right\}
\end{aligned}
\]
under this equivalence respect to discrepancies.
\end{corollary}
\begin{conjecture}
Perhaps that $\pi_* \varpi^!$ yields a stacky localization formula in the appropriate setting.
\end{conjecture}

\begin{remark}
\begin{enumerate}
\item The above results extend to any variety $X$ admitting a controlled resolution such that each step carries a constructive Bondal--Orlov localization description, e.g., an extremal Mori fiber space arising from a terminal threefold.

\item Unfortunately, the author is unable to give any explicit description of the kernel category in general, due to the lack of direct computations of derived pullback of  skyscraper sheaves in the general case.
\end{enumerate}
\end{remark}
\subsection{L.C.I. of toric flip}
We further consider examples of rational maps between terminal threefolds, the most important examples being flips. By Mori's theory, genuine flips always exist and are unique, and transform the model into one with more positive canonical class.
\begin{definition}
A flipping contraction $f:X\to Z$ is a small contraction from a terminal variety such that $-K_X$ is $f$-ample. The flip of $f$ is a small contraction $g:X^{+}\to Z$ such that $K_{X^{+}}$ is $g$-ample.
\[
\begin{tikzcd}[column sep=small,row sep=small]
X \arrow[dr,"f"'] \arrow[rr,dashed] & & X^{+} \arrow[dl,"g"] \\
& Z &
\end{tikzcd}
\]
where the birational transformation $X\dashrightarrow X^{+}$ is also called the flip.
\end{definition}
We need to determine when the standard type flips considered previously are genuine flips in the category of terminal threefold singularities. We have the following result of G.Brown \cite{Brown1999}.
\begin{theorem} [Theorem 8, \cite{Brown1999}]
Any genuine flip between terminal 3-folds as hypersurfaces of toric flips without quasireflections is one of the following:
\begin{enumerate}

    \item \(x_1 y_1 + g(x_2, x_3)\) \quad \((a_1, a_2, 1, -b_1, -a_2; a_1 - b_1)\) \quad \(a_1 > a_2, b_1\)

    \item \(x_1 y_1 + x_3^n\) \quad \((a_1, a_2, a_3, -b_1, -a_2; a_1 - b_1)\) \quad \(a_1 > a_2, a_3, b_1\)

    \item \(x_2^2 + x_1 y_2^2 + x_1^n y_1^{2n-1}\) \quad \((4, 1, 1, -2, -1; 2)\)

    \item \(z^n + x_2 y_1\) \quad \((a, 1, -1, -b, 0; 0)\) \quad \(a > b\)

    \item \(z^2 + x_1 y_3^3 + x_3^3 y_1\) \quad \((3, 1, -2, -1, 0; 0)\)

    \item \(x_2 y_1 + y_3^2 + x_1^n y_2^{2n+1}\) \quad \((4, 1, -3, -2, -1; -2)\)
\end{enumerate}
only the essential monomials are shown here; there are further restrictions on the parameters.
\end{theorem}
In subsequent work, L.Mallinson \cite{Mallinson2022} and H.P.Chang \cite{Chang2024} independently generalized the above conclusion to the complete intersection case.
\begin{theorem}[Theorem 1.3 and 1.4, \cite{Chang2024}]
\begin{enumerate}
    \item Any genuine flip of a threefold that is a complete intersection of a toric flip of codimension 2 is one of the following:
    \begin{enumerate}
        \item Flop: \(x_1 y_1 + g_1(z_1, z_2)\) \quad \(x_2 y_i + g_2(z_1, z_2)\) \quad \((1, 1, -1, -1, 0, 0; 0, 0)\)
        \item Flip: \(x_1 y_1 + h(x_2^{a_3}, x_3^{a_2})\) \quad \(x_2 y_2 + z^m\) \quad \((a_1, a_2, a_3, -b_1, -a_2, 0; a_1 - b_1, 0)\)
    \end{enumerate}
    \item Any genuine flip of a threefold that is a complete intersection of a toric flip has codimension \(\le 2\)..
\end{enumerate}
\end{theorem}
We obtain the following corollary, for the specific notation, refer to the Section \ref{Standardtypeflip}.
\begin{corollary}\label{corollaryin4}
We easily see that cases \((1)\), \((2)\), and \((3)\) are of type~I,
in this case we have a semiorthogonal decomposition
\[
\mathrm{D}^b(\widetilde{X})/\mathcal{K}^{(a)}_{-1}
=
\left\langle
\operatorname{Im}\bigl(\kappa^{(a)}_* \mathcal{O}_{L}(-e)\bigr),
\ldots,
\operatorname{Im}\bigl(\kappa^{(a)}_* \mathcal{O}_{L}(-1)\bigr),
\operatorname{Im}\,\phi_{*}\mathcal{R}\varphi^{*}\mathrm{D}^b(\widetilde{X}^{+})
\right\rangle .
\]
Cases \((4)\) and \((5)\) are of type~IV, we have
\[
\mathrm{D}^b(\widetilde{X})
=
\left\langle\left\langle
\operatorname{Im}\bigl(\kappa^{(a)}_*\mathcal{O}_{L}(-e)\bigr),
\ldots,
\operatorname{Im}\bigl(\kappa^{(a)}_*\mathcal{O}_{L}(-1)\bigr)\right\rangle,
\operatorname{Im}\,\phi_{*}\varphi^{*}\mathrm{D}^b(\widetilde{X}^{+})
\right\rangle.
\]
Case \((6)\) is of type~II, we obtain an equivalence
\[
\phi_{*}\varphi^{*}:
\mathrm{D}^{b}(\widetilde{X}^{+})/\mathcal{K}^{(b)}_{-}
\xrightarrow{\sim}
\mathrm{D}^{b}(\widetilde{X})/\mathcal{L}_{-}.
\]
Type \((a)\) belongs to a mixed type of type~IV, we have
\[
\phi_{*}\varphi^{*}:
\mathrm{D}^{b}(\widetilde{X}^{+})
\xrightarrow{\sim}
\mathrm{D}^{b}(\widetilde{X}).
\]
Type \((b)\) is in general of mixed type, we have
\[
\phi_{*}\mathcal{R}\varphi^{*}:
\mathrm{D}^{b}(\widetilde{X}^{+})/\mathcal{K}^{(b)}_{-}
\longrightarrow
\mathcal{L}_{-}\backslash
\mathrm{D}^{b}(\widetilde{X})/\mathcal{K}^{(a)}_{-1}.
\]
\end{corollary}

\newtheorem*{finalremark}{Final Remark}
\begin{finalremark}
In another part of this article, we will study further simple factorizations and apply them to concrete examples, continuing the perspective developed in our previous work \cite{Hao2025a}.
\end{finalremark}

\section{Acknowledgements}

The author would like to thank Hsin-Ku Chen, Jheng-Jie Chen, and Hung-Pin Chang for discussions on explicit classifications, Pedro N\'u\~nez for the suggestion to consider coarse moduli, as well as the members of the Northwest Tokyo Algebraic Geometry Group for their helpful discussions and concern. Additionally, the author has recently become aware of an independent work (arXiv:2608.22539) that uses AI techniques and $t$-structures to establish a similar conclusion.

\quad\\
\author{HAO XINGBANG}

\newcommand{\Addresses}{{
  \bigskip
  \footnotesize

  HAO XINGBANG, \textsc{Waseda University, Ookubo, Shinjuku-ku, Tokyo, 169-8555, Japan.}\par\nopagebreak
  \textit{E-mail address}: \texttt{hao@fuji.waseda.jp}
}}
\Addresses

\end{document}